\documentclass[a4paper,11pt]{amsart}

\usepackage[left=2.7cm,right=2.7cm,top=3.5cm,bottom=3cm]{geometry}
\usepackage{amsthm,amssymb,amsmath,amsfonts,mathrsfs,amscd,dsfont,verbatim}
\usepackage[utf8]{inputenc}
\usepackage[T1]{fontenc}
\usepackage[all,cmtip]{xy}
\usepackage{latexsym}
\usepackage{longtable}
\usepackage{mathtools}
\usepackage[pagebackref]{hyperref}
\usepackage{graphicx}

\numberwithin{equation}{section}
\newfont{\cyr}{wncyr10 scaled 1100}
\newfont{\cyrr}{wncyr9 scaled 1000}

\theoremstyle{plain}
\newtheorem{theorem}{Theorem}[section]
\newtheorem{proposition}[theorem]{Proposition}
\newtheorem{lemma}[theorem]{Lemma}
\newtheorem{corollary}[theorem]{Corollary}

\theoremstyle{definition}
\newtheorem{definition}[theorem]{Definition}

\theoremstyle{remark}
\newtheorem{remark}[theorem]{Remark}

\newcommand{\Q}{\mathbb{Q}}

\newcommand{\Z}{\mathbb{Z}}
\newcommand{\R}{\mathbb{R}}
\newcommand{\C}{\mathbb{C}}

\newcommand{\F}{\mathbb{F}}

\newcommand{\G}{\mathbb{G}}
\newcommand{\A}{\mathbb{A}}
\newcommand{\I}{\mathbb{I}}
\newcommand{\J}{\mathbb{J}}
\newcommand{\X}{\mathbb{X}}

\newcommand{\defeq}{\vcentcolon=}

\newcommand{\pwseries}[1]{[[ #1]]}
\newcommand{\p}{\mathfrak{p}}

\DeclareMathOperator{\Spec}{Spec}
\DeclareMathOperator{\Pic}{Pic}
\DeclareMathOperator{\End}{End}

\DeclareMathOperator{\Norm}{Norm}
\DeclareMathOperator{\Hom}{Hom}

\DeclareMathOperator{\Gal}{Gal}
\DeclareMathOperator{\GL}{GL}

\DeclareMathOperator{\SL}{SL}
\DeclareMathOperator{\Sel}{Sel}
\DeclareMathOperator{\M}{M}

\DeclareMathOperator{\im}{im}

\DeclareMathOperator{\Ta}{Ta}

\DeclareMathOperator{\Sym}{Sym}

\newcommand{\val}{\mathrm{val}}

\newcommand{\cyc}{{\mathrm{cyc}}}
\newcommand{\tr}{\mathrm{tr}}
\newcommand{\ord}{\mathrm{ord}}
\newcommand{\new}{\mathrm{new}}

\newcommand{\wild}{\mathrm{wild}}

\newcommand{\univ}{\mathrm{univ}}

\newcommand{\rec}{\mathrm{rec}}

\newcommand{\unr}{\mathrm{unr}}
\newcommand{\ES}{\mathrm{ES}}
\newcommand{\reg}{\mathrm{reg}}
\newcommand{\Katz}{\mathrm{Katz}}
\newcommand{\ST}{\mathrm{ST}}
\newcommand{\Cot}{\mathrm{Cot}}
\newcommand{\Oh}{\mathrm{Oh}}
\newcommand{\et}{\mathrm{\acute{e}t}}

\usepackage[usenames]{color}
\definecolor{Indigo}{rgb}{0.2,0.1,0.7}
\definecolor{Violet}{rgb}{0.5,0.1,0.7}
\definecolor{White}{rgb}{1,1,1}
\definecolor{Green}{rgb}{0.1,0.9,0.2}

\usepackage{tikz-cd}

\newcommand{\longmono}{\mbox{\;$\lhook\joinrel\longrightarrow$\;}}

\newcommand{\longepi}{\mbox{\;$\relbar\joinrel\twoheadrightarrow$\;}}

\newcommand{\mat}[4]{\left(\begin{array}{cc}#1&#2\\#3&#4\end{array}\right)}
\newcommand{\smallmat}[4]{\bigl(\begin{smallmatrix}#1&#2\\#3&#4\end{smallmatrix}\bigr)}

\newcommand{\dirlim}{\mathop{\varinjlim}\limits}
\newcommand{\invlim}{\mathop{\varprojlim}\limits}

\begin{document}

\include{thebibliography}

\title[On quaternionic ordinary families of modular forms and $p$-adic $L$-functions]{On quaternionic ordinary families of modular forms and $p$-adic $L$-functions}
\today
\date{}
\author{M. Longo, P. Magrone, E. R. Walchek}

\thanks{}

\begin{abstract} We use Serre--Tate expansions of modular forms to construct power series attached to quaternionic ordinary families of modular forms. We associate to these power series a big $p$-adic $L$-function interpolating the $p$-adic $L$-functions constructed by Burungale and Magrone at classical specializations. A crucial ingredient is the generalization of some results of Ohta to the quaternionic setting. 
\end{abstract}

\address{Dipartimento di Matematica, Universit\`a di Padova, Via Trieste 63, 35121 Padova, Italy}
\email{matteo.longo@unipd.it}

\subjclass[2010]{}

\keywords{quaternionic modular forms, Serre--Tate expansion, $p$-adic $L$-functions}

\maketitle

%\tableofcontents

\section{Introduction}\label{Intro}

The idea of using Serre--Tate expansions to attach power series to quaternionic modular forms goes back to \cite{Mori2, Mori, Mori1} and recently has been fruitfully exploited, also in the $\GL_2$-case, to construct $p$-adic $L$-functions associated to quaternionic modular forms; see \cite{Brako, CH}, for the $\GL_2$-case and \cite{Mori2, Brooks, Burungale, Magrone}, for the quaternionic division algebra case. For a family of $p$-adic modular forms on $\GL_2$, Serre--Tate expansions have been combined in \cite{Castella} with the Hida--Ohta theory of ordinary forms to obtain a big $p$-adic $L$-function that interpolates at classical specializations the $p$-adic $L$-function constructed in \cite{CH}. 

The goal of this paper is to develop a theory of big $p$-adic $L$-functions for indefinite rational quaternion algebras, similar to that developed in \cite{Castella} over $\GL_2$. To fix the notation, let $p$ be a prime number, $N$ be a positive integer not divisible by $p$, and $\I$ be a primitive branch of the Hida family attached to a residual representation which is absolutely irreducible and $p$-distinguished; this is a $p$-adic family of modular forms on $\GL_2$ (see \S\ref{PBHF} for details on the terminology). We also assume a Gorensteiness assumption on $\I$ (\emph{cf.} \S\ref{PBHF}, or 
\cite[Assumption 6.2]{LV-MM}).

Let $K/\Q$ be an imaginary quadratic extension of discriminant $-D_K$ prime to $N$ in which $p$ splits. Assume that $N=MD$ is a factorization of $N$ into coprime integers $M$ and $D$, with $D$ a square-free product of an even number of prime factors which are inert in $K$, and $M\geq 4$ only divisible by primes that are split in $K$.
Also, let $\widetilde{\I}=\I\otimes\Z_p^\unr$, where $\Z_p^\unr=W(\overline{\F}_p)$ is the ring of Witt vectors of a separable algebraic closure of the field with $p$ elements.  
The goal is to construct a big $p$-adic $L$-function $\mathscr{L}_{\I,\boldsymbol{\xi}}\in \widetilde{\I}[[{\Gamma_\infty}]]$ attached to $\I$ and a $p$-adic family of Hecke characters $\boldsymbol{\xi}$ such that when $\kappa\in\Spec(\I)$ is an arithmetic morphism of even weight $k\geq 2$ with $k\equiv 2\pmod{2(p-1)}$ the value $\mathscr{L}_{\I,\boldsymbol{\xi}}(\kappa)$ of the $\kappa$-specialization of $\mathscr{L}_{\I,\boldsymbol{\xi}}$ coincides (up to explicit factors) with the $p$-adic $L$-function $\mathscr{L}_{f_\kappa,\xi_\kappa}$ attached to the \emph{normalized} Jacquet--Langlands lift $f_\kappa$ of the $\kappa$-specialization of the given Hida family, and the $\kappa$-specialization $\xi_\kappa$ of $\boldsymbol{\xi}$. Writing $\mathscr{L}_{\I,\boldsymbol{\xi}}(\kappa)$ for $\kappa(\mathscr{L}_{\I,\boldsymbol{\xi}})$, our main result, Theorem \ref{mainthh}, then says that for $\kappa$ as before, we have 
 \[(\mathscr{L}_{\I,\boldsymbol{\xi}}(\kappa))=(\mathscr{L}_{f_\kappa,\xi_\kappa})\]
 as ideals in $\widetilde{\I}[[\Gamma_\infty]]$. 

Much of the work consists in clarifying the meaning of the previous sentences, in which we have at least one crucial difficulty, namely, the correct definition of the \emph{normalized} Jacquet--Langlands lift to a quaternion algebra of a given elliptic modular form. 
We do not develop this theme in full generality, 
but we use an approach which combines Serre--Tate expansions (which is the basic tool to actually define $\mathscr{L}_{f,\xi}$ for a quaternionic modular form $f$ and an appropriate Hecke character $\xi$ in
\cite{Burungale, Magrone, Mori1}), and an extension to the quaternionic case of Eichler--Shimura theory for ordinary families of modular forms developed by Ohta in a series of papers \cite{OhtaJ, Ohta-ES, OhtaC, OhtaMA}. In the $\GL_2$-case, one can normalize an eigenform by imposing that the first coefficient of its Fourier expansion is equal to $1$; more generally, one can define a notion of \emph{integral} modular forms by requiring that the Fourier coefficients belong to some fixed subring of $\C_p$. Adopting this viewpoint, and making use of the Atkin--Lehner involution, Ohta defines modules of projective limits of ordinary modular forms, which are identified with Wiles $\Lambda$-adic modular forms. 
We exploit this approach in the quaternionic case replacing Fourier expansions with Serre--Tate expansions. We thus obtain an integral theory of modular forms, and inverse limits of such, giving rise to a $\widetilde{\Lambda}_\mathcal{O}$-module denoted $\mathfrak{S}_\infty^*$ in the paper; here $\mathcal{O}$ is a subring of $\overline{\Q}_p$ containing $\Z_p^\unr$, $\Lambda_\mathcal{O}=\mathcal{O}[[1+p\Z_p]]$ and $\widetilde{\Lambda}_\mathcal{O}=\Lambda_\mathcal{O}\times\mathcal{O}[(\Z/N\Z)^\times]$.
We are able to attach to each element of $\mathfrak{S}_\infty^*$ a power series in $\Lambda_\mathcal{O}[[T_x]]$ following an approach of Ohta, where $T_x$ is the Serre--Tate parameter around a point $x$ in the Igusa tower over the special fiber of a Shimura curve $X$ of tame level $M$ attached to the rational quaternion algebra of discriminant $D$. This power series gives rise to the sought-for measure $\mathscr{L}_{\I,\boldsymbol{\xi}}$. The reader may notice that
working over $\Z_p^\unr$ is necessary, as in Ohta's work, since $\C_p$-periods of $p$-adic Hodge--Tate theory are 
defined over $\Z_p^\unr$ due to the ordinarity assumption; we also note that $\Z_p^\unr$ 
is the ring where Serre--Tate deformation theory
of ordinary abelian varieties is defined. The main result, as mentioned above, is Theorem \ref{mainthh}, to which we refer the reader for details.

The construction of similar $p$-adic families of $L$-functions can also be obtained by other methods, as in \cite{JLZ, LRZ, Seveso}. However, we would like to point out that our approach, even if restricted to the ordinary case, builds directly on weight $2$ specializations of a given $p$-adic family of modular forms, following the original approach of Ohta and Hida. This framework may be viewed as complementary to the methods of \cite{JLZ, LRZ, Seveso}, which avoid a direct comparison with weight $2$ specializations. The benefit of our more direct approach is that it may lead to a direct comparison with big Heegner points introduced in \cite{LV-MM} by interpolating Heegner points in weight $2$, thus proving a quaternionic analogue of \cite{Castella-MathAnn, Castella}. 
The ultimate goal of our construction of $\mathscr{L}_{\I,\boldsymbol{\xi}}$
is the study of the Greenberg Selmer group $\Sel(K,\mathbf{T})$ of the
Hida big Galois representation $\mathbf{T}$ attached to $\I$ over a quadratic imaginary field $K$ as above.
Although the big family of Galois representations $\mathbf{T}$ can also be obtained by taking inverse limits of Jacobins of modular curves (by taking $D=1$), the construction using the 
inverse limit of Tate modules of Jacobians of Shimura curves attached to \emph{indefinite} quaternion algebras of discriminant $D>1$ (as in 
Section \ref{sec.ESisom}, especially \eqref{defT}) allows for the use of Heegner points (or, in general weight, Heegner cycles)  attached to towers of Shimura curves associated to quaternion algebras of discriminant $D$, which are required for one of the most useful approaches to the study of arithmetic properties of 
$\Sel(K,\mathbf{T})$. To the best of the authors knowledge, Heegner points attached to towers of modular curves are not suitable to the study of $\Sel(K,\mathbf{T})$ in the case $D>1$, for which Heegner points on Shimura curves attached to quaternion algebras of discriminant $D$ seem to be necessary: 
for similar approaches (and further discussions) in the case of Selmer groups 
$\Sel(K,T)$ of $p$-adic representations $T$ attached to modular forms of even weight $k\geq 2$  and level $\Gamma_0(N)$, see for example \cite{NekCM}, \cite{NekLR}, \cite{Longo-AIF}.

We fall short of an analogue of Wiles theory of $\Lambda$-adic modular forms in this context, and more generally of developing a full analogue of Ohta's theory to obtain a quaternionic counterpart of \cite[Theorem 2.3.6]{Ohta-ES}. The reason is the lack of a clean connection between Hecke actions on modular forms and their $T_x$-expansions. Even if explicit descriptions of Hecke operators are available (see especially \cite[\S4.4]{Brooks}), some other inputs are missing; specifically, we cannot prove a full control theorem for our modules $\mathfrak{S}_\infty^*$ (as in \cite[\S7.3, Theorem 6]{hida-elementary} and \cite[Theorem 2.5.3]{Ohta-ES}) which does not allow for showing the freeness of the components of $\mathfrak{S}_\infty^*$ over $\Lambda_\mathcal{O}$, or that $\mathfrak{S}_\infty^*$ is a $\Lambda_\mathcal{O}$-module of finite rank; some of these obstructions are listed in the paper, see for example Remarks \ref{remohta}, \ref{remwiles} and \ref{remiso}. We believe that it would be interesting to develop a more complete Hida theory of quaternionic modular forms using Serre--Tate expansions, but at the moment it seems quite difficult to make progress in this direction. 

We finally remark that the extension of Ohta's theorem to the quaternionic case has been obtained by means of perfectoid techniques in the paper \cite{CHJ} by Chojecki--Hansen--Johansson (see also \cite{BHW} in the Hilbert case). These results work very generally for finite slope families, but it seems difficult to adapt them to deal with integrality conditions as done in this paper. Furthermore, Ohta's theory has also been extended to more general Shimura varieties (including the Hilbert and Siegel cases) in the recent preprint by Sangiovanni Vincentelli--Skinner \cite{SSV}. 

We fix throughout the text embeddings $\iota_\infty:\overline{\Q}\hookrightarrow\C$ and $\iota_p:\overline{\Q}\hookrightarrow\overline{\Q}_p$.

\subsubsection*{Acknowledgements}
{The authors would like to  heartly thank the referees for the careful reading of the paper and for many helpful
comments that lead to improve the clarity of our manuscript. The authors would like to thank F. Castella, D. Loeffler, A. Mori and S. Vigni for enlightening discussions. M.L. is supported by PRIN 2022 and INDAM GNSAGA; E.R.W. is supported by PEDECIBA.} 
 
\section{Shimura curves}
The goal of this section is to collect the needed results on quaternionic multiplication abelian surfaces, and see Shimura curves as solutions of moduli problems involving these objects. This will make it possible to introduce integral models of Shimura curves. Basic references are \cite{Buz, Carayol-Shimura, Brooks, Mori1}. 
 
\subsection{Quaternion algebras}\label{subsec-idempotents}
Fix an integer $M\geq 4$ and a square-free integer $D>1$ coprime to $M$ and which is the product of an even number of distinct primes. 
Fix also an odd prime number $p\nmid MD$. 
Let $B$ be the quaternion algebra over $\Q$ of discriminant $D$. For each place $v$ of $\Q$, let $\Q_v$ denote the completion of $\Q$ at $v$ and set
$B_v=B\otimes_{\Q}\Q_v$; in particular, since $D$ is a product of an even number of distinct primes, 
$B_\infty=B\otimes_\Q\R\simeq\M_2(\R)$ (\emph{i.e.} $B$ is an \emph{indefinite} quaternion algebra) and for each prime number $\ell\nmid D$ we have $B_\ell\simeq \M_2(\Q_\ell)$ (here for a ring $R$ we denote $\M_2(R)$ the $R$-algebra of $2\times 2$ matrices with entries in $R$). Fix a maximal order $\mathcal{O}_B$ of $B$ and fix
isomorphisms $i_\infty\colon B_\infty\simeq\M_2(\R)$ 
and $i_\ell\colon B_\ell\simeq\M_2(\Q_\ell)$ for all primes $\ell\mid Mp$
such that $i_\ell(\mathcal{O}_B\otimes_\Z\Z_\ell)= \M_2(\Z_\ell)$. 
For each integer $m\geq0$, the isomorphisms $i_\ell$ 
define Eichler orders $R_m$ of level $Mp^m$ such that 
$R_m\subseteq R_{m-1}$ for all $m\geq 1$. 
Let $U_m=\widehat{R}^\times_{m}=(R_{m}\otimes_\Z\widehat\Z)^\times$, where $\widehat{\Z}$ is the profinite completion of $\Z$. In particular, for each divisor $d$ of $Mp^m$ and each integer $m\geq 0$, we obtain isomorphisms \[i_d\colon \mathcal{O}_B\otimes_\Z(\Z/d\Z)\cong \M_2(\Z/d\Z).\]
Let $t\in\mathcal{O}_B$ be such that $t^2=-D<0$, which exists because $B$ splits over $\Q(\sqrt{-D})$, and define the involution $\dagger$ given by $b^\dagger\defeq t^{-1}\bar{b}t$, where $b\mapsto \bar{b}$ denotes the main involution on $B$; then $t\mapsto t^\dagger$ is a \emph{positive} involution (\emph{i.e.} 
$\mathrm{tr}(xx^\dagger)>0$ for all $x\in B$, where $\mathrm{tr}(a)=a+\bar{a}$, see \cite[Theorem 1.1]{Mori} for details). Fix finally the idempotent $e\in\mathcal{O}_{B,p}=\mathcal{O}_B\otimes_\Z\Z_p$
satisfying $i_p(e)=\smallmat 1000$; sometimes we simply 
write $e$ for $i_p(e)$.

\begin{remark}
Following \cite[(1.10)]{Mori} (see also \cite[Proposition 1.4]{Mori1}, \cite[\S2.3]{Burungale}, 
\cite[\S2.1]{Brooks}, \cite[\S2.1]{Longo-ShimuraMaass}) the idempotent $e$
can be chosen by fixing an \emph{Hashimoto model} for $B$. 
\end{remark}

\subsection{QM abelian surfaces}\label{sec-QMAS}
We introduce a class of abelian surfaces which plays a central role in the theory of Shimura curves. 

\begin{definition} 
Let $S$ be a scheme. A \emph{quaternionic multiplication (QM) abelian surface} $(A,\iota)$ over $S$ is an abelian scheme $A\rightarrow S$ of relative dimension $2$ equipped with an injective algebra homomorphism $\iota\colon\mathcal{O}_B\hookrightarrow\End_S(A)$. An \emph{isogeny} (resp. an \emph{isomorphism}) of QM abelian surfaces is an isogeny (resp. an isomorphism) of abelian schemes commuting with the $\mathcal{O}_B$-action. 
\end{definition}

Each QM abelian surface over $S$ can be equipped with a unique principal polarization $\theta_A$ such that for each geometric point $s$ of $S$ the corresponding Rosati involution of $\End(A_s)$, where $A_s$ is the fiber of $A\rightarrow S$ at $s$, coincides with the involution $x\mapsto x^\dagger$ on $\mathcal{O}_B$ (\cite[Lemma 1.1]{Milne}; see also \cite[Lemma 5]{DT}, \cite[\S1]{Buz}). If $\pi\colon A_1\rightarrow A_2$ is an isogeny, taking the dual isogeny $\pi^\vee\colon A_2^\vee\rightarrow A_1^\vee$ and composing with the principal polarizations of $A_1$ and $A_2$ gives an isogeny, still denoted  $\pi^\vee\colon A_2\rightarrow A_1$. We say that $\pi$ has \emph{degree $d$} if the composition of $\pi^\vee\circ\pi$ is, locally on $A_1$, the multiplication by a unique integer $d$.

\subsection{Shimura curves with na\"{\i}ve level structures} \label{shimura section} 
Given a group $G$ and a scheme $S$, we write $G_S$ for the constant group scheme of value $G$ over $S$; when the context is clear, we often simplify the notation and write $G$ for $G_S$. 

\begin{definition}
Let $m\geq 0$ be an integer, $d\mid Mp^m$ be a positive integer and $(A,\iota)$ a QM abelian surface over a $\Z[1/d]$-scheme $S$. A \emph{na\"{\i}ve full level $d$ structure} on $A$ is an isomorphism \[\alpha\colon\mathcal{O}_B\otimes_\Z(\Z/{d}\Z)\overset\sim\longrightarrow A[{d}]\] of $S$-group schemes locally for the étale topology of $S$ which commutes with the left actions of $\mathcal{O}_B$ given by $\iota$ on $A[d]$, and the multiplication from the left of $\mathcal{O}_B$ on the constant $S$-group scheme $\mathcal{O}_B\otimes_\Z(\Z/{d}\Z)$. 
\end{definition}

\begin{remark}
A full level ${d}$ structure is equivalent via $i_d$ to an isomorphism $\M_2(\Z/d\Z)\overset\sim\rightarrow A[d]$ of finite flat group schemes over $S$, which commutes with the left action of $\mathcal{O}_B$ given by $\iota$ on $A[d]$ and by left matrix multiplication on $\M_2(\Z/d\Z)$. Also note that, if $k$ is an algebraically closed field, giving a full level $d$ structure on a QM abelian surface defined over $S=\Spec(k)$ is equivalent to fixing a $\Z/d\Z$-basis of the group $A[d](k)$.
\end{remark} 

The group $(\mathcal{O}_B\otimes_\Z(\Z/d\Z))^\times$ acts from the left on the set of full level $d$ structures on a QM abelian surface $(A,\iota)$ as follows. If $g\in(\mathcal{O}_B\otimes_\Z(\Z/d\Z))^\times$, then right multiplication $r_g(x)=xg$ by $g$ defines an isomorphism of 
$\mathcal{O}_B\otimes_\Z(\Z/d\Z)$ which commutes with the left action of $\mathcal{O}_B\otimes_\Z(\Z/d\Z)$ on itself by left multiplication; for a na\"{\i}ve full level $d$ structure  $\alpha\colon\mathcal{O}_B\otimes_\Z(\Z/d\Z)\overset\sim\rightarrow A[d]$ on $(A,\iota)$, we see that $\alpha_g=\alpha\circ r_g$ is a na\"{\i}ve full level $d$ structure on $(A,\iota)$, and the map $\alpha\mapsto \alpha_g$ gives a left action of $(\mathcal{O}_B\otimes_\Z(\Z/d\Z))^\times$ on the set of na\"{\i}ve full level $d$ structures on $(A,\iota)$. 
For any subgroup $U$ of $\widehat{\mathcal{O}}_B^\times$ (where $\widehat{\mathcal{O}}_B=\mathcal{O}_B\otimes_\Z\widehat{\Z}$), we obtain a left action of $U$ on full level $d$ structures by composing the action of $(\mathcal{O}_B\otimes_\Z(\Z/d\Z))^\times$ with the canonical projection $U\subseteq \widehat{\mathcal{O}}_B^\times\overset{\hat\pi_d}\longepi(\mathcal{O}_B\otimes_\Z(\Z/d\Z))^\times$. We finally define a right action of $U$ by composing this left action with the map $g\mapsto g'=\mathrm{norm}(g)g^{-1}$. We are especially interested in level $V_1(d)$ and $V_0(d)$ structures, where $V_0(d)\subseteq\widehat{\mathcal{O}}_B^\times$ is the inverse image via $\hat{\pi}_d$ of the subgroup of $(\mathcal{O}_B\otimes_\Z(\Z/d\Z))^\times$ consisting of matrices which are upper triangular modulo $d$, and $V_1(d)$ is the subgroup of $V_0(d)$ consisting of elements $g$ such that $\hat\pi_d(g)$ is of the form $\smallmat {*}{*}01$. 

\begin{definition}
Let $(A,\iota)$ be a QM abelian surface  over a $\Z[1/d]$-scheme $S$ and $V$ a subgroup of $\widehat{\mathcal{O}}_B^\times$. A \emph{na\"{\i}ve level $V$ structure} is an equivalence class of full level $d$ structures under the right action of $V$. We say that two triples $(A,\iota,\alpha)$ and $(A',\iota',\alpha')$ consisting of QM abelian surfaces equipped with level $ V$ structures are \emph{isomorphic} if there is an isomorphism of QM abelian surfaces $\varphi\colon A\rightarrow A'$ such that $\varphi\circ\alpha=\alpha'$.  \end{definition}

\begin{theorem}
Let $m\geq 0$ be an integer. The functor which takes a {$\Z[1/DMp^m]$}-scheme $S$ to the set of isomorphism classes of such triples $(A,\iota,\alpha)$ consisting of a QM abelian surface $(A,\iota)$ equipped with a na\"ive level $V_1(Mp^m)$ structure over $S$ is representable by a {$\Z[1/MDp^m]$}-scheme ${X}_{m}$, which is projective, smooth, of relative dimension $1$ and geometrically connected.
\end{theorem}
\begin{proof}[References for the proof] 
The representability result is due to Morita \cite[Main Theorem 1]{Morita}. A proof for a more general $V\subseteq V_1(d)$ can be found in \cite[\S2]{Buz} (see especially \cite[Corollary 2.3 and Propositions 2.4 and 2.5]{Buz}) combining the representability result of \cite[Theorem \S14, Exposé III]{Boutot} and the proof in \cite[Lemma 2.2]{Buz} of the rigidity of the relevant moduli problem. See also \cite[\S4]{DT} and \cite[Theorem 2.2]{Brooks}. 
\end{proof}

We have an isomorphism of compact Riemann surfaces 
\[X_m(\C)= B^\times\backslash(\mathcal{H}^\pm\times\widehat{B}^\times)/U_m\simeq \Gamma_m\backslash\mathcal{H}\]
where $U_m=\widehat{R}_m^\times$ was introduced in \S\ref{subsec-idempotents} and 
$\Gamma_m$ is the subgroup of the group of norm-one elements $(R_m^\times)_1$ of $R_m^\times$ such that their image in $\M_2(\Z/Mp^m\Z)$ is congruent to $\smallmat 1*01$ modulo $Mp^m$; here the action of $\Gamma_m$ on the complex upper half plane $\mathcal{H}$ is via fractional linear transformations through the isomorphism $i_\infty\colon B\otimes_\Q\R\simeq\M_2(\R)$ fixed before, while the action of $B^\times$ and $U_m$ on the product $\mathcal{H}^\pm\times\widehat{B}^\times=(\C-\R)\times\widehat{B}^\times$ is via $b(z,g)u=(b(z),bgu)$; finally, the isomorphism can be checked using the strong approximation theorem \cite[Theorem 5.2.10]{Miyake} for quaternion algebras. 

\subsection{Shimura curves with Drinfeld level structures} 
We sometimes need to use integral models of $X_m$ defined over a ring in which $p$ is \emph{not} invertible. Let $m\geq 1$ an integer. Recall that we have a left action of $\mathcal{O}_B$ on $A[p^m]$ and therefore we also obtain a left action of $\mathcal{O}_{B,p}=\mathcal{O}_B\otimes_\Z\Z_p$ on $A[p^m]$; thus, through $i_p$, we have a left action of $\M_2(\Z_p)$ on $A[p^m]$. 
We have a decomposition
\[A[p^m]=\ker(e)\oplus eA[p^m].\]

\begin{remark} 
The element $w\in \mathcal{O}_{B,p}$ satisfying $i_p(w)=\smallmat 0110$ induces an isomorphism of group schemes $w\colon\ker(e)\overset\sim\rightarrow\ker(1-e),$ and we have $eA[p^m]=\ker(1-e)$ and $(1-e)A[p^m]=\ker(e)$.
\end{remark}

\begin{definition}\label{def-drinfeld}
Let $(A,\iota)$ be a QM abelian surface defined over a scheme $S$. A \emph{level $\Gamma_1(p^m)$ structure on $A$} is the datum of a cyclic finite flat $S$-subgroup scheme $H$ of $eA[p^m]$ which is locally free of rank $p^m$, equipped with the choice of a $S$-generator $P$ of $H$ in the sense of \cite[\S1.4]{KM}.
\end{definition}

A simple generalization of \cite[Lemma 4.4]{Buz} to higher powers of $p$ shows, for a QM abelian surface $(A,\iota)/S$ over a $\Z[1/MDp]$-scheme $S$, the existence of a canonical isomorphism between level $V_1(p^m)$ structures and the choice of a generator $Q$ of a finite flat subgroup scheme $T$ of $\ker(e)\subseteq A[p^m]$; in our notation, the generator of $eA[p^m]$ is then $P=wQ$ and the subgroup is $H=wT$.

We denote by $T=(A,\iota,\alpha,(H,P))$ quadruplets consisting of a QM abelian surface $(A,\iota)/S$ over a scheme $S$ equipped with a na\"{\i}ve $V_1(M)$ level structure $\alpha$ and a level $\Gamma_1(p^m)$ structure $(H,P)$ on $A$. Two such quadruplets $T=(A,\iota,\alpha,(H,P))$ and $T'=(A',\iota',\alpha',(H',P'))$ are said to be \emph{isomorphic} if there is an isomorphism $\varphi\colon A\rightarrow A'$ of QM abelian surfaces which takes $\alpha$ to $\alpha'$ and such that $\varphi(H)=H'$ and $\varphi(P)=P'$.  

\begin{theorem} 
The functor which takes a $\Z_{(p)}$-scheme $S$ to the set of isomorphism classes of such quadruplets $T=(A,\iota,\alpha,(H,P))$ over $S$ is representable by a $\Z_{(p)}$-scheme $\mathcal{X}_{m}$, which is proper and finite over $X_0\otimes\Z_{(p)}$, where the tensor product is over $\Z[1/MD]$. Moreover, there is a canonical isomorphism of $\Q$-schemes between the generic fiber of $\mathcal{X}_{m}$ and $X_m$.
\end{theorem}

\begin{proof}[References for the proof]
The proof of this result is similar to the proof of \cite[Proposition 4.1]{Buz} which only considers the case $m=1$; the extension to the general case does not present difficulties. The reader may also consult \cite[Theorem 81]{Clark}, \cite[\S1.1]{Mori1} or \cite[\S7.3]{Carayol-Shimura}.  
\end{proof}

We use the same symbol $\mathcal{X}_m$ for $\mathcal{X}_m\otimes\Z_p$, where the tensor product is over $\Z_{(p)}$. We also remark that $\mathcal{X}_m$ admits a \emph{regular} model ${\mathcal{X}}^\mathrm{reg}_m$ over $\Z_p[\zeta_{p^m}]$, where $\zeta_{p^m}$ is a primitive $p^m$-root of unity, which is finite and flat over $\mathcal{X}_0\otimes_{\Z_p}\Z_p[\zeta_{p^m}]$; the scheme $\mathcal{X}_m^\reg$ is the normalization of 
$\mathcal{X}_0\otimes_{\Z_p}\Z_p[\zeta_{p^m}]$ in $X_m$. 
This model is constructed by means of Drinfeld bases (\cite[\S7.2]{Carayol-Shimura}) 
and canonically balanced Drinfeld level $\Gamma_1(p^m)$ structures (\cite[Definition 3.4]{Buz}), which we do not explicitly introduce in this paper. 
For references, the reader may consult \cite[Theorem 4.10]{Buz} for the case $m=1$ (where $\mathcal{X}_1^\reg$ is denoted 
$X^D(U,\mathrm{Bal.can.}\Gamma_1(\ell))$, and $\ell=p$ in our notation)
and \cite[\S7.3]{Carayol-Shimura} for $m\geq 1$ (where $\mathcal{X}_m^\reg$ is denoted $\boldsymbol{M}_{n,H}$; it should also be noticed that \cite{Carayol-Shimura} works over totally real number fields $F$, and \emph{excludes} the 
case $F=\Q$, where the correspondent results follow from the work of Morita).

\subsection{Igusa curves}\label{secIgusa} 
We start with the notion of trivialization. Let $(A,\iota)$ be an ordinary QM abelian surface over a $\Z_p$-scheme $S$. For each integer $m\geq 1$, let $A[p^m]^0$ be the connected component of the $p^m$-torsion subgroup scheme $A[p^m]$ of $A$, and let ${\mu}_{p^m}$ denote the group scheme of $p^m$-roots of unity. An \emph{arithmetic trivialization} on $A[p^m]$ is an isomorphism 
\[\beta\colon{\mu}_{p^m} \overset\sim\longrightarrow eA[p^m]^0\] 
of finite flat connected group schemes over $S$; note that the existence of an arithmetic trivialization actually implies that $A$ is ordinary. Similarly, let $A[p^\infty]^0$ be the connected component of the $p$-divisible group of $A$, and let ${\mu}_{p^\infty}=\dirlim_m\mu_{p^m}$ denote the group scheme of $p$-power roots of unity. An \emph{arithmetic trivialization} on $A[p^\infty]$ is an isomorphism 
\[\beta\colon{\mu}_{p^\infty} \overset\sim\longrightarrow eA[p^\infty]^0\] 
of finite flat connected group schemes over $S$. Equivalently, an arithmetic trivialization of $A[p^\infty]$ is an isomorphism of formal group schemes 
\[\phi_\beta\colon e\widehat{A}\overset\sim\longrightarrow\widehat{\mathbb{G}}_m\]
where $\widehat{\mathbb G}_m$ is the multiplicative formal group (recall that the connected component $A[p^\infty]^0$ of the $p$-divisible group $A[p^\infty]$ of an abelian variety $A$ is just the formal completion of $A$ along the zero section, see \cite[Example (b) in \S 2.3]{Tate-pdiv}). Observe that the choice of a $S$-generator of $\mu_{p^\infty}$, which we fix once and for all, allows us to view trivializations as particular Drinfeld level $\Gamma_1(p^m)$ structures; see \cite[\S1.12]{KM} for the interpretation of $\mu_{p^m}$ in terms of Drinfeld level structures for the multiplicative group.

Let $\mathbb{X}_0$ denote the special fiber of $\mathcal{X}_0$, let $\mathrm{Ha}$ be the Hasse invariant of $\mathbb{X}_0$ and let $\widetilde{\mathrm{Ha}}$ be a lift of $\mathrm{Ha}$ to $\mathcal{X}_0$ (\cite[\S7]{Kassaei}). Then the \emph{ordinary locus} $\mathcal{X}_0^\ord=(\mathcal{X}_0\otimes\Z_p)[1/\widetilde{\mathrm{Ha}}]$ (tensor product over $\Z_{(p)}$) is an affine open $\Z_p$-subscheme of $\mathcal{X}_0\otimes\Z_p$ representing  the moduli problem which associates to any $\Z_p$-scheme $S$ the isomorphism classes of triplets $(A,\iota,\alpha)$ where the $S$-scheme $A$ is ordinary and $\alpha$ is a level $V_1(M)$ structure. Denote by $\mathrm{Ig}_{m}$ the $\F_p$-scheme representing the moduli problem which associates to any $\F_p$-scheme $S$ the set of isomorphism classes of quadruplets $(A,\iota,\alpha,\beta )$ consisting of an ordinary QM abelian surface $(A,\iota)$ over $S$ equipped with a level $V_1(M)$ structure $\alpha$ and an arithmetic trivialization $\beta$ of $A[p^m]$ (see \cite[Chapter 8]{Hida-padicbook}, \cite[\S2.1]{Hida-control} or \cite[\S2.5]{Burungale}). Since the étale part of $A[p^m]$ is locally constant in the ordinary locus $\X_0^\ord$ of $\X_0$ (whose points correspond to ordinary QM abelian surfaces), we have a finite \'etale map $\mathrm{Ig}_m\rightarrow\X_0^\ord$.
For each $m\geq 0$ and each $\F_p$-scheme $S$, the canonical maps ${\mu}_{p^m}\hookrightarrow {\mu}_{p^{m+1}}$ of $S$-group schemes induce a canonical map $\mathrm{Ig}_{m+1}\rightarrow\mathrm{Ig}_{m}$ and we can define the $\F_p$-scheme \[\widehat{\mathrm{Ig}}=\invlim_m\mathrm{Ig}_{m}.\] 

\subsection{Special fibers}\label{secspecialfibers}
We now discuss the reduction of $\mathcal{X}_m$ to $\F_p$. 
A complete reference for this result is 
\cite[Theorem 81]{Clark}, reviewed in \cite[\S1.1]{Mori1}. Partial 
results have been obtained by Morita \cite{Morita}, Carayol \cite{Carayol-Shimura}, Buzzard \cite{Buz}. 

A result of Morita \cite{Morita} shows that the special fiber $\X_0$ of $\mathcal{X}_0$ is smooth; 
see also \cite[\S6.1]{Carayol-Shimura}, where $X_0$ is denoted $M_{0,H}$ and $\mathcal{X}_0$ is 
denoted $\boldsymbol{M}_{0,H}$. 
Following the exposition of \cite[\S1.1]{Mori1},
\[\mathcal{X}_m\otimes\F_p\simeq\coprod_{r=0}^m\mathcal{Z}(r)\]
where the components $\mathcal{Z}(r)$ have multiplicity $\phi(p^{m-r})$, and $\phi$ is the Euler totient function. 
The reduced scheme $\mathcal{Z}(r)^\mathrm{red}$ is the Igusa curve 
$\mathrm{Ig}_r$, and all the components meet transversally at each supersingular point. 
Denote $\mathcal{X}_m^\mathrm{arith}$ the smooth open subscheme obtained by discarting the non-reduced components at all prime divisors $\ell\mid Mp$. Then 
$\mathcal{X}_m^\mathrm{arith}$ represents the functor on 
$\Z_p$-schemes $S$ to isomorphism classes of quadruplets $(A,\iota,\alpha,\beta)$ 
over $S$,
where $\beta$ is an arithmetic trivialization of $A[p^m]$ as in \S\ref{secIgusa}, 
and $\alpha$ is is an \emph{arithmetic trivialization} of $A[M]$, 
defined (similarly as in the case of arithmetic trivializations of $A[p^m]$) as an isomorphism 
$\alpha\colon{\mu}_{M} \overset\sim\rightarrow eA[M]^0$ 
of finite flat connected group schemes over $S$. 

\section{Modular forms}\label{secmodforms}
The goal of this section is to introduce Serre--Tate expansions of quaternionic modular forms, and inverse limits of such. We also develop a theory of integral modular forms using Serre--Tate expansions. Furthermore, we introduce Katz modular forms in the quaternionic setting to give a meaning to Serre--Tate expansions. The viewpoint is elementary with respect to other sources (\cite{CHJ}, for example). The basic references are \cite{Brasca-Eigen2, Burungale, Mori2, Mori, Mori1, Brooks, Ohta-ES}. 

\subsection{Modular forms on Shimura curves}\label{sec.modforms}
Let $k\geq 2$ be an even integer. 
We define the \emph{$\C$-vector space of modular forms of weight $k$ over $X_m$} to be the finite-dimensional $\C$-vector space $S_k(\Gamma_m,\C)$ of holomorphic functions $f\colon\mathcal{H}\rightarrow\C$ such that $f|\gamma=f$ for all $\gamma\in \Gamma_m$; here we use standard notations: for a matrix $\gamma=\smallmat abcd\in\GL_2(\R)$ with $\det(\gamma)>0$ we denote
\begin{equation}\label{MFhol}
(f|\gamma)(z)=\det(\gamma)^{k/2}(cz+d)^{-k}f(\gamma(z))
\end{equation}
and we use the fixed isomorphism $i_\infty:B_\infty\simeq\M_2(\R)$ 
to obtain an action of the elements of $B^\times$ of positive reduced norm on the complex upper half plane $\mathcal{H}$. 

The map $f\mapsto \omega_f$ which takes an holomorphic function $f$ in $\mathcal{H}$ to the differential form $f(z)dz^{\otimes{k/2}}$ gives an isomorphism of $\C$-vector spaces 
\begin{equation}\label{hol-dif}
S_k(\Gamma_m,\C)\simeq H^0(X_m(\C),\Omega^{\otimes k/2}).\end{equation}
Here, for a scheme $X$, we denote by $\Omega_X^1$ the sheaf of regular differentials and, for an integer $i\ge 0$, by $\Omega_X^{\otimes i}$ its $i$-th tensor product; we often drop $X$ from the notation when it is clear. More generally, using that $X_m$ has a model defined over $\Q$, for any field extension $F/\Q$ we can consider the $F$-vector space \[H^0(X_m,\Omega^{\otimes k/2})\otimes F\simeq H^0(X_{m/F},\Omega^{\otimes k/2})\] (here $X_{m/F}=X_m\otimes F$, the tensor products are over $\Q$, and the isomorphism follows from the universal coefficient theorem). For each embedding $F\hookrightarrow\C$, we can use \eqref{hol-dif} to define a $F$-subspace $S_k(\Gamma_m,F)$ of $S_k(\Gamma_m,\C)$. 

We also need to define modular forms for more general subgroups than $\Gamma_m$. For each integer $m\geq 1$, define $\Gamma_{m-1}^{(m)}=\Gamma_{m-1}\cap (R_m^\times)_1$ to be the subgroup of $\Gamma_{m-1}$ consisting of those $\gamma$ whose image in 
$\M_2(\Z_p)$ is congruent to $\smallmat {*}{*}{0}{*}$ modulo $p^m$;
so we have $\Gamma_m\subseteq\Gamma_{m-1}^{(m)}\subseteq\Gamma_{m-1}$ and $\Gamma_m$ is normal in $\Gamma_{m-1}^{(m)}$. Let $X_{m-1}^{(m)}$ be the corresponding Shimura curve, defined over $\Q$, whose complex points are identified with $\Gamma_{m-1}^{(m)}\backslash\mathcal{H}$ and let $S_k(\Gamma_{m-1}^{(m)},\C)$ be defined as in \eqref{MFhol} or \eqref{hol-dif} with respect to the elements $\gamma\in\Gamma_{m-1}^{(m)}$ instead of those of the whole group $\Gamma_{m-1}$; then we have canonical inclusions $S_k(\Gamma_{m-1},\C)\subseteq S_k(\Gamma_{m-1}^{(m)},\C)\subseteq S_k(\Gamma_m,\C)$ of $\C$-vector spaces. Again, more generally, we can define $S_k(\Gamma_{m-1}^{(m)},F)=H^0(X_{m-1/F}^{(m)},\Omega^{\otimes{k/2}})$, where as before $X_{m-1/F}^{(m)}$ is the base change of $X_{m-1}^{(m)}$ to $F$. 

\subsection{Hecke operators}\label{secHeckeop}
Hecke operators on $S_k(\Gamma_m,\C)$ are defined using double coset decompositions $T(\alpha)=\Gamma_m\alpha\Gamma_m=\coprod_i\Gamma_m\alpha_i$ 
for $\alpha\in \Delta_m$ 
and $T(\alpha)$ in the Hecke ring $\mathfrak{h}(R_m^\times,\Delta_m)$,
where $\Delta_m$ is the semigroup of $b\in B^\times$ whose image in $\M_2(\Z_\ell)$ is congruent modulo $\ell^{\val_\ell(Mp^m)}$ to a matrix of the form $\smallmat a{*}0{*}$ with $a\in\Z_\ell^\times$, for all $\ell\nmid D$ (see \cite[\S3.1]{shimura}, \cite[\S5.3]{Miyake}, \cite[\S2.1]{LV-IJNT}). The action of $T(\alpha)$ on $f\in S_k(\Gamma_m,\C)$ is via $f|T(\alpha)=\det(\alpha)^{k/2-1}\sum_if|\alpha_i$. If $T=T(\alpha)$ and $\alpha^\iota=\Norm(\alpha)\alpha^{-1}$, where $\Norm$ is the reduced norm map, we set $T^*=T(\alpha^\iota)$. 

By strong approximation (\emph{cf.} \cite[Theorem 5.2.10]{Miyake}), the Hecke algebra $\mathfrak{h}(R_{m}^\times,\Delta_{m})$ is canonically isomorphic to the tensor product of local Hecke algebras $\bigotimes_\ell \mathfrak{h}(R_{m,\ell}^\times,\Delta_{m,\ell})$ for all primes $\ell$, where $\Delta_{m,\ell}$ is the semigroup of elements of $R_{m,\ell}$ of non-zero norm for $\ell\nmid Mp$ whose image in $\M_2(\Z_\ell)$ is congruent modulo $\ell^{\val_\ell(Mp^m)}$ to a matrix of the form $\smallmat a{*}0{*}$ with $a\in\Z_\ell^\times$, for all $\ell\nmid D$ (\emph{cf.} \cite[Theorem 5.3.5]{Miyake}). This correspondence can be used to describe Hecke operators more explicitly. Let $T_\ell$ for $\ell\nmid MDp^m$ and $U_\ell$ for $\ell\mid Mp^m$ denote the Hecke operator $T(\alpha_\ell)$ for $\alpha_\ell$ an element of norm equal to the prime number $\ell$. Then $U_\ell$ corresponds to the element having $R_{m,\ell}^\times\smallmat 100\ell R_{m,\ell}^\times=\sum_{i=0}^{\ell-1}R_{m,\ell}^\times\smallmat 1i0\ell$ in the $\ell$-factor and $1$ elsewhere, while $T_\ell$ is the sum of the previous double coset with $R_{m,\ell}^\times\smallmat \ell001R_{m,\ell}^\times$.
For $d\in (\Z/MDp^m\Z)^\times$, we denote by $\langle d\rangle$ the operator corresponding to the choice of the element having $R_{m,\ell}^\times\smallmat{\ell^{\val_\ell(d)}}{0}{0}{\ell^{\val_\ell(d)}}$ on the $\ell$-factor.

\subsection{Atkin--Lehner involution}\label{sectionAL} 
The Atkin--Lehner involution can be defined as in \cite[\S3.3]{Ogg} (see also \cite[Section 1]{dVP}, \cite[\S2.2]{LRV}) and is described by the map $f\mapsto f|\tau_m$, where $\tau_m$ is an element $R_m$ of norm equal to $Mp^m$ which normalizes the subgroup $(R_m^\times)_1$. The involution $\tau_m$ is not an element of the Hecke algebra $\mathfrak{h}(R_m^\times,\Delta_m)$; however, if $\widetilde\Delta_m$ is the semigroup of elements of non-zero norm in $R_{m}$, then $\tau_m$ is an element of the Hecke algebra $\mathfrak{h}(R_m^\times,\widetilde{\Delta}_m)$ which has a similar adelic description with $\ell$-factor $R_{m,\ell}^\times\tau_{m,\ell}R_{m,\ell}^\times=R_{m,\ell}^\times\tau_{m,\ell}$ with $\tau_{m,\ell}=\smallmat 01{-\ell^{\val_\ell(Mp^m)}}0$.

Let $(R^\times_{m,\ell})_1$ be the subgroup of norm $1$ elements of $R_{m,\ell}^\times$ and ${\Gamma}_{m,\ell}$ the subgroup consisting of those elements which are congruent to $\smallmat 1*01$ modulo $\ell^{\val_\ell(Mp^m)}$ and fix a system of representatives $\Sigma_m=\{\varsigma_a\mid a\in(\Z/Mp^m\Z)^\times\}$ of $(R^\times_m)_1/\Gamma_m$. Then by the strong approximation theorem we can set up an isomorphism of groups $(R^\times_m)_1/\Gamma_m\simeq \prod_{\ell\mid Mp}(R^\times_{m,\ell})_1/{\Gamma}_{m,\ell}$, which we can normalize in such a way that $\varsigma_a$ is sent to the element with $\ell$-component equal to $\smallmat {\ell^{\val_\ell(a)}}{*}{0}{\ell^{-\val_\ell(a)}}$ for integers $a$ such that $1\leq a<Mp^{m}$, with $(a,Mp^m)=1$. 

\begin{lemma}\label{Ohtaformulas}
For $f\in S_k(\Gamma_m,\C)$, $T\in\mathfrak{h}(R_m^\times,\Delta_m)$ 
and $d\in (\Z/MDp^m\Z)^\times$ we have: 
\begin{enumerate}
\item $f|\tau_m|T^*=f|T|\tau_m$. 
\item $f|\tau_m|\langle d\rangle^*=f|\langle d\rangle|\tau_m$. 
\item $f|\tau_m|\varsigma_a=f|\varsigma_a^{-1}|\tau_m$.    \end{enumerate}
\end{lemma}
\begin{proof}
Given the above adelic description of the Hecke operators and the Atkin--Lehner involution, they reduce to a simple check. However, one can indirectly check these formulas noticing that they are well known for elliptic modular forms (see for example \cite[(2.1.8)]{Ohta-ES}), and (since it is enough to check them on Hecke eigenforms) use Jacquet--Langlands correspondence, which is equivariant with respect to the Hecke and diamond operators, and with respect to the Atkin--Lehner involution by \cite[Theorem 1.2]{BD96}.\end{proof}

Fix a system of representatives $\sigma_a$ for $\Gamma_{m-1}^{(m)}/\Gamma_m$. Recall that $\Gamma_{m,p}$ is the subgroup of $R_{m,p}^\times$ consisting of elements congruent to $\smallmat 1*01$ modulo $p^m$, and let $\Gamma_{m-1,p}^{(m)}$ be the subgroup of $R_{m,p}^\times$ consisting of elements which are congruent to $\smallmat 1*01$ modulo $p^{m-1}$ and congruent to $\smallmat **0*$ modulo $p^m$. Again by the strong approximation theorem we have a canonical isomorphism of groups $\Gamma_{m-1}^{(m)}/\Gamma_m\simeq \Gamma_{m-1,p}^{(m)}/\Gamma_{m,p}$ and we may choose the system of representatives so that $\sigma_a$ is sent to a matrix $\smallmat a{*}0{a^{-1}}$ with $a\in U^{(m-1)}/U^{(m)}$, where for each integer $m\geq 1$ we set $U^{(m)}\defeq 1+p^{m}\Z_p$. We may also fix a system of representatives $\rho_j$ for $\Gamma_{m-1}/\Gamma_{m-1}^{(m)}$ so that under the canonical isomorphism $\Gamma_{m-1}/\Gamma_{m-1}^{(m)}\simeq\Gamma_{m-1,p}/\Gamma_{m-1,p}^{(m)}$ these $\rho_j$ are sent to matrices $\smallmat 1{0}{jp^{m-1}}{1}$ for $j\in\Z/p\Z$. Then we may write \[\Gamma_{m-1}=\coprod_{a,j}\Gamma_m(\sigma_a\rho_j)=\coprod_{j,a}(\rho_j\sigma_a)\Gamma_m\] (with $a$ and $j$ as before). Define for all integers $m\geq 1$ the trace maps
\begin{itemize}
\item $\mathrm{Tr}_m^{(m)}\colon S_2(\Gamma_m,\C)\rightarrow S_2(\Gamma_{m-1}^{(m)},\C)$ by $\mathrm{Tr}_m^{(m)}(f)=\sum_a f|\sigma_a$;
\item $\mathrm{Tr}_{m-1}^{(m)}\colon S_2(\Gamma^{(m)}_{m-1},\C)\rightarrow S_2(\Gamma_{m-1},\C)$
by $\mathrm{Tr}_{m-1}^{(m)}(f)=\sum_j f|\rho_j$;
\item $\mathrm{Tr}_m\colon S_2(\Gamma_m,\C)\to S_2(\Gamma_{m-1},\C)$ by $\mathrm{Tr}_m=\mathrm{Tr}^{(m)}_{m-1}\circ\mathrm{Tr}^{(m)}_m$.
\end{itemize}

\begin{lemma}\label{lemmatrace}
We have the following formulas:
\begin{enumerate}
\item $\mathrm{Tr}_m^{(m)}(f|\tau_m)=(\mathrm{Tr}_m^{(m)}(f))|\tau_m$,  
for $f\in S_2(\Gamma_m,\C)$. 
\item $\mathrm{Tr}_{m-1}^{(m)}(f)|\tau_{m-1}=f|\tau_{m}|U_p$, 
for $f\in S_2(\Gamma^{(m)}_{m-1},\C)$. 
\end{enumerate}
\end{lemma}
\begin{proof}
As in Lemma \ref{Ohtaformulas}, given the explicit description of Hecke operators and chosen representatives, the proof reduces to standard matrix computations; a short proof of these formulas can be obtained by using the Jacquet--Langlands correspondence, which is equivariant with respect to the operators involved, and noticing that for elliptic modular forms these formulas are well known (\cite[(2.3.3)]{Ohta-ES}).
\end{proof}

\subsection{Newforms, $p$-stabilizations and $p$-depletions} \label{stabilization}
For $d\mid Mp$, let $v_d$ be an element of the Hecke algebra $\mathfrak{h}(R_m^\times,\widetilde{\Delta}_m)$ which has an adelic description with trivial $\ell$-factors at primes $\ell\neq d$ and at primes $\ell\mid d$ the $\ell$-factor is $R_{m,\ell}^\times \smallmat {\ell^{\val_\ell(d)}}0{0}1R_{m,\ell}^\times=R_{m,\ell}^\times v_\ell$.  
Using the Petersson inner product (\cite[\S6.1]{Miyake}) and the elements $v_d$, we may introduce the notion of 
quaternionic \emph{newforms} similarly to the more familiar case of elliptic modular forms:
if we denote $S_k^\mathrm{old}(\Gamma_m,\C)$ the $\C$-vector space spanned by 
$g|v_d$ for all $d\mid Mp^m$ with $d\neq Mp^m$ and all forms of level $d$, 
the $\C$-vector space of newforms  $S_k^\mathrm{new}(\Gamma_m,\C)$
is the orthogonal complement of $S_k^\mathrm{old}(\Gamma_m,\C)$ in 
$S_k(\Gamma_m,\C)$ with respect to the Petersson inner product: 
see \cite[\S1.5]{Hijikata} for details (here \emph{forms of level $d$} 
are defined as before by replacing $MDp^m$ by its divisor $d$). 
Alternatively we can define newforms of $S_k(\Gamma_m,\C)$ as those quaternionic modular forms which correspond
via the Jacquet--Langlands correspondence (\cite[Proposition 2.12]{Hida-Abelian}, \cite[\S2.2]{LRdVP}; see also \cite[\S1.6]{BD}) to an elliptic newform of weight $k$ 
and level $\Gamma(MD,p^m)$, where $\Gamma(MD,p^m)$ is the 
subgroup of $\SL_2(\Z)$ consisting of matrices that are upper triangular modulo $MD$ 
and are congruent to $\smallmat 1*01$ modulo $p^m$: this amounts to requiring that 
$S_k^\mathrm{new}(\Gamma_m,\C)$ is spanned by quaternion eigenforms whose Hecke eigenvalues at primes $\ell\nmid MDp$ are equal to the Hecke eigenvalues of an elliptic newform of weigh $k$ 
and level $\Gamma(MD,p^m)$.  

Let $f^\sharp$ be a newform in $S_k(\Gamma_0,\C)$.
Assume that the $T_p$-operator acts on $f^\sharp$ as multiplication by a $p$-adic unit; let $\alpha$ and $\beta$ be the roots of the Hecke polinomial, and suppose that $\alpha$ is a $p$-adic unit. Then we define the \emph{ordinary $p$-stabilization} of $f^\sharp$ to be \[f=f^\sharp-\beta f^\sharp|v_p.\] Then $f\in S_k(\Gamma_1,\C)$ (see \cite[\S2.1]{Hijikata}) and is an eigenform for $U_p$ with eigenvalue $\alpha$: this is a standard computation using the formulas in \cite[\S2.4]{Hijikata}.

\begin{comment}
 We add a proof of the last fact. 
    We first observe that $f|v_p|U_p=f$. We have \[f(z)|U_p=p^{k/2-1}\sum_if(z)|\alpha_i=p^{k/2-1}\sum_{i=0}^{p-1} p^{-k/2}f(\alpha_i(z))=\frac{1}{p}\sum_{i=0}^{p-1}f(\alpha_i(z)),\] so 
    \[f(z)|v_p|U_p=\frac{1}{p}\sum_{i=0}^{p-1}f(v_p\alpha_i(z))=\frac{1}{p}\sum_{i=0}^{p-1}f(z)=f(z)\]
    because $\smallmat p001\smallmat 1x0p=p\smallmat 1x01$. We have $T_p=U_p+p^{k/2-1}v_p$ (\cite[page 82]{Hijikata}), so we get 
    \[\begin{split}f|U_p&=(f^\sharp-\beta f^\sharp|v_p)|U_p\\&=f^\sharp|(T_p-p^{k-1}v_p)-\beta f^\sharp v_pU_p\\&=
(\alpha+\beta)f^\sharp-\alpha\beta f^\sharp|v_p-\beta f^\sharp\\&=\alpha (f^\sharp-\beta f^\sharp|v_p)\end{split}.\]
\end{comment}

Define the \emph{$p$-depletion} of $f\in S_k(\Gamma_0,\C)$ or $f\in S_k(\Gamma_1,\C)$ to be 
\[f^{[p]}=f-f|U_p|v_p.\] Then a standard computation using that $v_pU_p$ is the identity operator shows that $f^{[p]}=(f^\sharp)^{[p]}$ in the notation of the preceding paragraph.

\subsection{Katz modular forms}\label{KatzMF}
{Following \cite{Brasca, Brooks} we introduce the notion of Katz modular forms in the quaternionic setting.

Let $A$ be a QM abelian surface over a $\Z_{(p)}$-scheme $S$. Then $\pi_* \Omega_{A/S}$, where $\Omega_{A/S}$ is the bundle of relative differentials and $\pi:A\rightarrow S$ is the structural map, inherits an action of $\mathcal{O}_B$. Tensoring the action of $\mathcal{O}_B$ on $\pi_* \Omega_{A/S}$ with the scalar action of $\Z_{(p)}$ we obtain an action of $\mathcal{O}_{B}\otimes_\Z\Z_{(p)}$ on $\pi_* \Omega_{A/S}$. Define the invertible sheaf $\underline{\omega}_{A/S}=e\pi_* \Omega_{A/S}$. When $S$ is clear, we just write $\underline{\omega}_A$ for $\underline{\omega}_{A/S}$. 

Let $R$ be a $\Z_{(p)}$-algebra. A \emph{test object} over $R$ for $\Gamma_m$ (resp. for $\Gamma_{m-1}^{(m)}$) for an integer $m\geq 0$ (resp. $m\ge 1$) is a collection of objects $T=(A, \iota,\alpha,(P,H))$ consisting of
\begin{itemize}
\item a QM abelian surface $(A,\iota)$ defined over $S=\Spec(R)$;
\item a level $V_1(M)$ structure $\alpha$ on $A/S$;
\item a Drinfeld level $\Gamma_1(p^m)$ (resp. $\Gamma_0(p^{m})\cap\Gamma_1(p^{m-1})$) structure $(P,H)$ on $A/S$ consisting of a cyclic finite flat $S$-subgroup scheme $H\subseteq eA[p^m]$ locally free of rank $p^m$ and a point $P\in H(S)$ of exact order $p^m$ (resp. $p^{m-1}$).
%\item a non-vanishing global section of the line bundle $\underline{\omega}_{A/S}$ of relative differentials.
\end{itemize}

\begin{remark} 
Suppose that $A$ is ordinary. Then arithmetic trivializations give rise to Drinfeld level structures; in particular, one can take as Drinfeld level $\Gamma_1(p^m)$ (resp. $\Gamma_0(p^{m})\cap\Gamma_1(p^{m-1})$)  structures a trivialization $\beta\colon \mu_{p^m}\simeq eA[p^m]^0$ (resp. a pair of trivializations $\beta_{m-1}\colon \mu_{p^{m-1}}\simeq eA[p^{m-1}]^0$ and $\beta_m\colon\mu_{p^m}\simeq eA[p^m]^0$ compatible in the sense that $\beta_m(\zeta^p)=\beta_{m-1}(\zeta)$; it suffices to specify $\beta_m$).\end{remark}

We say that two test objects are \emph{isomorphic} if there is an isomorphism of QM abelian surfaces which induces isomorphisms of $V_1(M)$ and Drinfeld level $\Gamma_1(p^m)$ (resp. $\Gamma_0(p^{m})\cap\Gamma_1(p^{m-1})$) structures. For a morphism $\varphi\colon R_0 \rightarrow R'_0$ of $R$-algebras and a test object $T=(A,\iota,\alpha,(P,H))$, we write $T_\varphi=(A_\varphi,\iota_\varphi,\alpha_\varphi,(P_\varphi,H_\varphi))$ for the base change of $T$ to $R'_0$ via $\varphi$.

\begin{definition}\label{def-modforms} Let $R$ be a $\Z_{(p)}$-algebra and $k$ be an integer. A \emph{$R$-valued Katz modular form of weight $k$ on $\Gamma_m$} (resp. $\Gamma^{(m)}_{m-1}$) is a rule $\omega_f$ that assigns to every isomorphism class of test objects $T=(A,\iota,\alpha,\beta)$ for $\Gamma_m$ (resp. $\Gamma^{(m)}_{m-1}$) over a $R$-algebra $R_0$ 
a global section $\omega_f(T)$ of $\underline{\omega}_A^{\otimes k}$ satisfying the following base change
compatibility condition: for any morphism $\varphi\colon R_0 \rightarrow R'_0$ of $R$-algebras, we have $\omega_f(T_\varphi) = \varphi^*(\omega_f(T))$.
\end{definition}

Denote by $S_k^\Katz(\Gamma_m,R)$ (resp. $S_k^\Katz(\Gamma^{(m)}_{m-1},R)$) the $R$-module of $R$-valued Katz modular forms of weight $k$ and level $\Gamma_m$ (resp. $\Gamma^{(m)}_{m-1}$).

Equivalently, a $R$-valued Katz modular form of weight $k$ and level $\Gamma_m$ (resp. $\Gamma^{(m)}_{m-1}$) is a rule $f$ that assigns to every isomorphism class of test objects $T=(A,\iota,\alpha,\beta)$ for $\Gamma_m$ (resp. $\Gamma^{(m)}_{m-1}$) over a $R$-algebra $R_0$, and a section $\omega$ of $\underline{\omega}_A^{\otimes k}$  
a value $f(T,\omega)\in R_0$ satisfying the following conditions: 
\begin{itemize}
    \item Base change condition: $f(T_\varphi,\omega) = \varphi(f(T,\varphi^*(\omega))$ for any morphism $\varphi\colon R_0 \rightarrow R'_0$ of $R$-algebras; 
    \item Weight $k$ condition: $f(T,\lambda\omega) =
\lambda^{-k}f(T,\omega)
		$, for any $\lambda\in R_0^\times$.	 
\end{itemize} The correspondence $f\leftrightarrow\omega_f$ showing the equivalence between the two definitions is given by the formula $\omega_f(T)=f(T,\omega)\omega$.

Let $\mathcal{A}_m\rightarrow\mathcal{X}_m$ be the universal object of the moduli space $\mathcal{X}_m$. Then we have an isomorphism of $R$-modules (which we write $f\mapsto \omega_f$)
\[S_k^\Katz(\Gamma_m,R)\overset\sim\longrightarrow H^0(\mathcal{X}_m,\underline{\omega}_{\mathcal{A}_m}^{\otimes k})\] of $R$-modules induced by the evaluation of Katz modular forms on the universal object over $R$ associated with $\mathcal{A}_m$. See \cite[\S3.1]{Brooks} for more details and for equivalent definitions.

\begin{lemma}\label{lemmamodfors} Let $F$ be a subfield of $\C$ and $k\geq 2$ an even integer. We have a canonical isomorphism  of $F$-vector spaces $S_k(\Gamma_m,F)\simeq S_k^\Katz(\Gamma_m,F)$.
\end{lemma}
\begin{proof} The result follows 
from the isomorphism of line bundles $\underline{\omega}_{\mathcal{A}_m/F}^{\otimes 2}\simeq \Omega_{X_m/F}^1$ given by the Kodaira--Spencer map (see \cite[Theorem 2.5]{Mori}). 
\end{proof}

In light of Lemma \ref{lemmamodfors}, for any $\Z_{(p)}$-algebra $R$ we set \[S_k(\Gamma_m,R)\defeq S_k^\Katz(\Gamma_m,R).\] In what follows we are mostly interested in the cases where $R$ is (the ring of integers of) a subfield of $\C$ or (the valuation ring of) a $p$-adic field.}

\subsection{Power series expansions}\label{sec-powerseriesexp}
To study integrality conditions on modular forms we will combine the Katz interpretation of modular forms with the theory of Serre--Tate coordinates, which leads to a notion of power series expansions for quaternionic modular forms, following \cite{Mori2, Mori, Mori1} (see also \cite{Brako, Brooks}). We let $\Z_p^\unr=W(\overline{\F}_p)$ be the ring of Witt vectors of the (separable) algebraic closure of $\F_p$.

We first discuss Serre--Tate coordinates and deformation theory in the quaternionic setting, developed in \cite[\S3]{Mori} and \cite[\S4]{Brooks}, following ideas of Serre--Tate and Katz \cite{Kat}. Let $A$ be an ordinary QM abelian surface over $\overline{\F}_p$; denote by $A^\vee$ the dual abelian variety and by $\Ta_p(A)$ and $\Ta_p(A^\vee)$ their $p$-adic Tate modules. Recall that for any deformation $\mathcal{A}$ over an Artinian ring $R$ with residue field $\overline{\F}_p$, the Weil pairing induces a $\Z_p$-bilinear form
\[q_{\mathcal{A}}\colon \Ta_p(A) \times \Ta_p(A^\vee) \longrightarrow \widehat{\G}_m(R),\] 
where $\widehat{\G}_m$ is the formal multiplicative group scheme over $\overline{\F}_p$; thus, we have $\widehat{\G}_m({R})=1+\mathfrak{m}_{R}$, where $\mathfrak{m}_{R}$ is the maximal ideal of ${R}$. 
Choosing $\Z_p$-bases $\{x_1,x_2\}$ of $\Ta_p(A)$ and $\{y_1,y_2\}$ of $\Ta_p(A^\vee)$ and setting $T_{i,j}(\mathcal{A})=q_{\mathcal{A}}(x_i,y_j)-1$ gives the \emph{Serre--Tate coordinates} of the deformation $\mathcal{A}$. 
The deformation functor is pro-representable by a universal object $\mathcal{A}^\univ/\mathcal{R}^\univ$ (\cite[Theorem 2.1)]{Kat}). Serre--Tate coordinates are also defined for deformations $\mathcal{A}$ over complete Noetherian local rings $\mathcal{R}$ with maximal ideal $\mathfrak{m}_\mathcal{R}$ by setting 
$q_{\mathcal{A}}=\invlim q_{\bar{\mathcal{A}}_n}$ and $T_{i,j}(\mathcal{A})=q_{\mathcal{A}}(x_i,y_j)-1$, where $\bar{\mathcal{A}}_n=\mathcal{A}\otimes_\mathcal{R}(\mathcal{R}/\mathfrak{m}_\mathcal{R}^n)$ is a deformation of $A$ over the Artinian ring $\mathcal{R}/\mathfrak{m}_\mathcal{R}^n$ and the limit is taken over $n$. 
The map $T_{i,j}\mapsto q_{\mathcal{A}^\univ}(x_i,y_j)-1=T_{i,j}(\mathcal{A}^\mathrm{univ})$ induces an isomorphism between the power series ring $\Z_p^\unr[[T_{i,j}]]$ in four variables $T_{i,j}$ with $i,j=1,2$ and the universal deformation ring $\mathcal{R}^\univ$; note that this isomorphism is non-canonical, depending on the choice of bases.   

Since $A$ is also equipped with QM, we can consider the subfunctor of the deformation functor of $A$ which sends an Artinian local ring $R$ with residue field $\overline{\F}_p$ to the set of deformations of $A$ as a QM abelian surface with level $V_1(M)$ structure; this subfunctor is pro-representable by a ring $\mathcal{R}^\univ_\mathrm{QM}$ which is isomorphic (again, non-canonically) to $\Z_p^\unr[[T]]$ (see \cite[Proposition 3.3]{Mori} and \cite[Proposition 4.5]{Brooks}). The isomorphism is again obtained by using Serre--Tate coordinates for QM deformations as follows. The Tate module $\Ta_p(A)$ of $A$ inherits an action of $\mathcal{O}_B$ and hence of $\mathcal{O}_B \otimes \Z_p$; the idempotent $e$ induces a splitting 
\[\Ta_p(A) = \ker(e)\oplus e\Ta_p(A) \] 
and we can find a $\Z_p$-basis $\left\lbrace x_1,x_2\right\rbrace $ of $\Ta_p(A)$ such that $ex_1 = x_1$ and $ex_2 = 0$. If we let $x_1^\vee=\theta_A(x_1)$, then the isomorphism $\Z_p^\unr[[T]]\simeq\mathcal{R}^\univ_\mathrm{QM}$ is given by the map $T\mapsto T_{\mathcal{A}^\univ}(x_1,x_1^\vee)$. 

We now introduce power series expansions of modular forms by using Serre--Tate coordinates. Fix once and for all 
a point \begin{equation}\label{point-fix}
{x}=(\bar{x},\beta)\in \widehat{\text{Ig}}(\overline{\F}_p)\end{equation} in the Igusa tower, \emph{i.e.}, the isomorphism class of a quadruple $(A, \iota, \alpha,\beta)$, lying above a point $\bar{x} =(A, \iota, \alpha)$ in $\mathbb{X}^{\ord}_0(\overline{\F}_p)$ and equipped with an arithmetic trivialization $\beta\colon \mu_{p^\infty}\simeq eA[p^\infty]^0$. 
For each integer $m\geq 1$, denote $x_m\in \mathrm{Ig}_m$ the 
image of $x$ via the canonical projection. The trivialization $\beta$ determines by Cartier duality (see \cite[\S 3.1]{Magrone}) a point \[x_{\beta}^\vee\in e^\dagger{\Ta_p(A^\vee)({\overline{\F}_p})}.\]  
Take \[x_\beta\defeq\theta_{A^\vee}(x_\beta^\vee) \in e\Ta_p(A)(\overline{\F}_p),\] where $\theta_{A^\vee}$ is the dual of $\theta_A$. 

Let $\mathcal{R}^\univ_x=\mathcal{R}^\univ_{\mathrm{QM},x}$ be the QM universal deformation ring of the abelian variety $A$. We fix the Serre--Tate coordinates around $\bar{x}$ to be those associated with the choice of $x_\beta$, \emph{i.e.} we set \[T_x(\mathcal{A})=q_{\mathcal{A}}(x_\beta,x_\beta^\vee)-1\] for each deformation of QM abelian surfaces $\mathcal{A}$ of $A$. As before, denote by $\mathcal{A}^\univ_x/\mathcal{R}^\univ_x$ the universal object; so  $\mathcal{A}^\univ_x$ is a QM abelian variety over $\mathcal{R}_x^\univ \cong \Z_p^\unr[[T_x]]$, equipped with a quaternionic action $\iota_x^\univ$. The level $V_1(M)$ structure $\alpha$ can be uniquely lifted to a level $V_1(M)$ structure $\alpha^\univ_x$ on $\mathcal{A}^\univ_x$ by \'etaleness because $p\nmid M$. Since deformation theory of ordinary abelian varieties over $\overline{\F}_p$ is equivalent to deformation theory of their $p$-divisible groups (see \cite[Theorem 1.2.1]{KatzST}), the trivialization $\beta$ can also be lifted to a trivialization $\beta_x^\univ$ of $\mathcal{A}_x^\univ$, and this defines a Drinfeld level $\Gamma_1(p^m)$ structure $\beta_{x,m}^\univ$ for each $m\geq 0$.
So the  quadruplet 
\[T_{\mathcal{A}_x^\univ}=(\mathcal{A}_x^\univ,\iota_x^\univ,\alpha_x^\univ,\beta_x^\univ)\] gives rise to a universal test object over $\Z_p^\unr[[T_x]]$, by which we mean 
that the quadruplet $(\mathcal{A}_x^\univ,\iota_x^\univ,\alpha_x^\univ,\beta_{x,m}^\univ)$
is a test object for $\Gamma_m$, for each $m\geq 0$. For each $\Z_p^\unr$-algebra $R$, base change to $R^\univ$ gives rise to a universal test object over $R^\mathrm{univ}$, denoted with the same symbol to simpify the notation.
We have an isomorphism ${\Ta_p(A^\vee)({\overline{\F}_p})}\simeq \Hom(\widehat{\mathcal{A}}^\univ,\widehat{\mathbb{G}}_m)$; the point $x_{\beta}^\vee$ determines a homomorphism $\varphi_\beta$ and we can define $\omega_{\mathcal{A}_x^\univ}=\varphi_\beta^*(dT/T)$, where $dT/T$ is the standard differential on the formal multiplicative group. 

\begin{definition}\label{q-expansion}
Let $R$ be a $\Z_p^\unr$-algebra, $f\in S_k(\Gamma_m,R)$ and $x=(\bar{x},\beta)\in \widehat{\text{Ig}}(\overline{\F}_p)$. The formal series $f(T_x) = f(T_{\mathcal{A}_x^\univ},\omega_{\mathcal{A}_x^\univ})$ in $R[[T_x]]$, is the \emph{$T_x$-expansion of $f$}. \end{definition}
Thus $f(T_x)$ 
denotes the value of $f$ at 
$(\mathcal{A}_x^\univ,\iota_x^\univ,\alpha_x^\univ,\beta_{x,m}^\univ,\omega_{\mathcal{A}_x^\univ})$. The reader is referred to \cite[\S2.2]{Mori1} for more details on this definition. 

\begin{remark}
We usually distinguish between modular forms and associated power series by writing $f$ for the first and $f(T_x)$ for the second. 
\end{remark}

We also recall the following fact. The operator $v_p$ introduced in \S\ref{stabilization} corresponds to the operator $V$ in \cite[\S3.6]{Brooks} on Katz modular forms; to uniformize the notation with \emph{loc. cit.} we also let $U=U_p$ be the corresponding operator. The action of the operators $U$ and $V$ on power series has been studied in \cite[\S4.4]{Brooks}; in particular, by \cite[Proposition 4.17]{Brooks}, we know that if $f(T_x)\in \Z_p^\unr[[T_x]]$, then $f^{[p]}(T_x)\in\Z_p^\unr[[T_x]]$ as well, and if $f(T_{x})= \sum_{n\geq0}a_nT_{x}^n$, then $f^{[p]}(T_{x})=\sum_{p\nmid n}a_nT_{x}^n$ (see also \cite[Lemma 5.2]{Burungale}).
 
\subsection{Integral modular forms}
We now use power series expansion to introduce a new integrality condition on modular forms. We work throughout with a field extension $F/\Q_p$ which contains the maximal unramified extension of $\Q_p$, and let $\mathcal{O}$ be its valuation ring. Recall the point $x=(A,\iota,\alpha,\beta)=(\bar{x},\beta)\in\widehat{\mathrm{Ig}}$ fixed in \eqref{point-fix}, 
recall that $x_m\in\mathrm{Ig}_m$ denotes the image of $x$ via the canonical projection. 
We suppose from now on that the abelian variety $A$ has a smmoth model $\widetilde{A}$ 
defined over $\mathcal{O}$, which means that the abelian scheme $\widetilde{A}\rightarrow \Spec(\mathcal{O})$ is 
smooth and, if $\mathfrak{m}$ is the maximal ideal of 
$\mathcal{O}$ and $k=\mathcal{O}/\mathfrak{m}$ is its residue field, 
then $A=\widetilde{A}\otimes_\mathcal{O}k$ (this is the case 
of CM points used in \cite{Mori}, \cite{Mori1} and in the subesequent paper \cite{LMW}). 
Since $\widetilde{A}$ is a deformation of $A$, there 
is a map $\mathcal{R}_x^\mathrm{univ}\rightarrow\mathcal{O}$ 
such that $\widetilde{A}=\mathcal{A}_x^\mathrm{univ}\otimes_{\mathcal{R}_x^\mathrm{univ}}\mathcal{O}$, and together with the images of the polarization and the level structures 
we obtain a sequence of points $\tilde{x}_m=(\widetilde{A},\tilde\iota,\tilde\alpha,\tilde\beta_m)$
with $\tilde{x}_m\in \mathcal{X}_m^\mathrm{arith}(\mathcal{O})\subseteq\mathcal{X}_m^\mathrm{sm}(\mathcal{O})$ for each $m\geq 0$, 
compatible with respect to the covering maps $X_m\rightarrow X_{m-1}$ for all $m\geq 1$
(recall that $(\widetilde{A},\tilde{\iota},\tilde{\alpha})\in \mathbb{X}_0^{\mathrm{ord}}(\overline{\F}_p)$, $\tilde{\beta}_m$ is an arithmetic trivialization, 
$\mathcal{X}_m^\mathrm{sm}$ denotes the smooth locus of $\mathcal{X}_m$, 
and $\mathcal{X}_m^\mathrm{arith}$ was introduced in \S\ref{secspecialfibers}). 
By smoothness,  
$\Spec(\mathcal{R}_x^\mathrm{univ})$ 
is isomorphic to the formal completion $\widehat{\mathcal{O}}_{\mathcal{X}_m,x_m}$ 
of $\mathcal{O}_{\mathcal{X}_m}$ at $x_m$. See \cite[\S2.2]{Mori1} for a more complete discussion. 

%We also fix a point $x$ in $\mathcal{X}_m(\mathcal{O})$ and we assume that $\bar{x}\in \mathbb{X}_m(\overline{\F}_p)$ belongs to $\mathrm{Ig}_m(\overline{\F}_p)$, where recall that $\mathbb{X}_m$ is the special fiber of $\mathcal{X}_m$. 
For $f\in S_k(\Gamma_m,F)$, we may then consider $f(T_{x})\in F[[T_{x}]]$. 

\begin{definition}
Define the \emph{Serre--Tate $\mathcal{O}$-valued modular forms} $S_k^\ST(\Gamma_m,\mathcal{O})$ to be the $\mathcal{O}$-submodule of $S_k(\Gamma_m,F)$ consisting of those $f$ such that $f(T_x)$ belongs to $\mathcal{O}\pwseries{T_x}$.   
\end{definition}

Using the compatibility under base change, we see that  
\begin{equation}\label{KatzinST}
S_k(\Gamma_m,\mathcal{O})\subseteq S_k^\ST(\Gamma_m,\mathcal{O}).
\end{equation}
The canonical covering map ${X}_{m/F}\twoheadrightarrow X_{m-1/F}^{(m)}$ induces a canonical injective map by pull-back
\[S_k(\Gamma_{m-1}^{(m)},F)\longmono S_k(\Gamma_{m},F),\] and we can define 
\[S_k^\ST(\Gamma_{m-1}^{(m)},\mathcal{O})=S_k(\Gamma_{m-1}^{(m)},F)\cap S_k^\ST(\Gamma_m,\mathcal{O}).\] 
We also note the following result concerning integrality conditions. Recall the system of representatives $\{\sigma_a\}$ of $\Gamma_{m-1}^{(m)}/\Gamma_m\simeq \Gamma_{m-1,p}^{(m)}/\Gamma_{m,p}$ chosen in \S\ref{sectionAL} so that $\sigma_a$ is sent to the matrix $\smallmat a{*}0{a^{-1}}$ with $a\in U^{(m-1)}/U^{(m)}$.

\begin{lemma}\label{LemmaKatz1}
If $f\in S_k^\ST(\Gamma_m,\mathcal{O})$ then $f|\sigma_a\in S_k^\ST(\Gamma_m,\mathcal{O})$.
\end{lemma}
\begin{proof} If ${x}=(\bar{x},\beta)\in \widehat{\text{Ig}}(\overline{\F}_p)$ is the fixed point in the Igusa tower, then the action of $\sigma_a$ corresponds to the choice of a different trivialization $\beta_a$, and therefore we need to look at the Serre--Tate expansion at the point $x_a=(\bar{x},\beta_a)\in \widehat{\text{Ig}}(\overline{\F}_p)$; however, different choices of trivializations gives rise to different isomorphisms of the universal deformation space of the QM abelian surface $A$ corresponding to $\bar{x}$, with the same power series ring $\Z_p^\unr[[X]]$. Since $f(T_x)\in \mathcal{O}[[T_x]]$ and the isomorphisms above are equivariant with respect to base change to $\mathcal{O}$, then $f|\sigma_a(T_x)=f(T_{x_a})$ also belongs to $\mathcal{O}[[T_x]]$. 
\end{proof}

We now look more closely to the case of weight $k=2$. Let $J_m$ be the Jacobian variety of $X_m$. Let $\mathcal{J}_{m/\mathcal{O}}$ be the N\'eron model of $J_{m/F}\defeq J_m\otimes_\Q F$; then we have a canonical isomorphism of $F$-schemes $\mathcal{J}_{m/\mathcal{O}}\otimes_{\mathcal{O}}F\simeq J_{m/F}$. Let $\mathcal{J}_{m/\mathcal{O}}^0$ be the connected component of the identity of $\mathcal{J}_{m/\mathcal{O}}$ (\cite[\S1b)]{Mazur-Rational}). 
We then have an injective map of $\mathcal{O}$-modules $\Cot(\mathcal{J}_{m/\mathcal{O}}^0)\hookrightarrow\Cot(J_{m/F})$ on the level of cotangent spaces and a canonical isomorphism of $F$-vector spaces \[\Cot(J_{m/F})\simeq H^0(X_{m/F},\Omega^1_{X_{m/F}}).\] Composing these maps gives an injective map of $\mathcal{O}$-modules 
\begin{equation}\label{Cot-MF}
\Cot(\mathcal{J}_{m/\mathcal{O}}^0)\longmono S_2(\Gamma_m,F).\end{equation}
\begin{lemma}\label{lemma_Cot-MF}
    The image of \eqref{Cot-MF} is contained in $S_2(\Gamma_m,\mathcal{O})$.
\end{lemma}
\begin{proof}
   Fix an embedding 
    $X_m\hookrightarrow J_m$ over $\Q$; in the modular curves case, 
    this is obtained by fixing a cusp, while in this case it
    is usually constructed as a multiple of the map which takes a point
    $x\in X_m$ to the divisor $[x]-[\xi]$, where $\xi$ is the 
    \emph{Hodge class} of $X_m$, \emph{i.e} the unique class in $\mathrm{Pic}(X_m)\otimes_\Z\Q$ such that 
    the Hecke operator $T_\ell$ acts on $\xi$ as multiplication by $\ell+1$, 
    for all primes $\ell\nmid MD$; the reader is referred to \cite[page 30]{Zhang-Annals}, \cite[page 187]{Zhang-AJM} for details. If $\mathcal{X}_m^\mathrm{sm}$ denotes as before the smooth locus of 
    $\mathcal{X}_m$, the universal property of N\'eron models 
    shows that there exists a map $\mathcal{X}_m^\mathrm{sm}\rightarrow\mathcal{J}_m$, 
    defined over $\mathcal{O}$. 
    The $T_x$-expansion of the image of an element $\omega\in \Cot(\mathcal{J}_{m/\mathcal{O}}^0)$ 
    via the map 
    \eqref{Cot-MF} is then the pull-back of $\omega$ via the map 
\[\Spec(\mathcal{R}_x^\mathrm{univ})\simeq\Spec(\widehat{\mathcal{O}}_{\mathcal{X}_{m},x_m})\longrightarrow
\mathcal{X}_m^\mathrm{sm}\longrightarrow\mathcal{J}_m,\] 
hence it belongs to $S_2(\Gamma_m,\mathcal{O})$. 
\end{proof}

Define finally 
\[S_k^*(\Gamma_m,\mathcal{O})=\{f\in S_k(\Gamma_m,F): f|\tau_m\in S_k^\ST(\Gamma_m,\mathcal{O})\},\]
\[S_k'(\Gamma_m,\mathcal{O})= S_k^\ST(\Gamma_m,\mathcal{O})\cap S_k^*(\Gamma_m,\mathcal{O}). \]

\begin{remark}\label{propGross} Suppose that $\mathcal{O}\supseteq\Z_p[\zeta_{p}]$. Then the image of \eqref{Cot-MF} is equal to $S_2'(\Gamma_1,\mathcal{O})$. 
The proof of this fact is adapted from \cite[Proposition 8.4]{Gross-Tameness}. By Lemma \ref{lemma_Cot-MF} and the fact that $\tau_1$
preserves $H^0(\mathcal{X}_{1/\mathcal{O}},\Omega^1)$, 
the 
image of \eqref{KatzinST} is contained in $S_2'(\Gamma_1,\mathcal{O})$. We prove the opposite inclusion: for that, fix $\omega$ a regular differential on $X_1$ which belongs to $S_2'(\Gamma_1,\mathcal{O})$. Let $\widetilde{\mathcal{X}}_1$ be a minimal regular resolution of $\mathcal{X}_{1/\mathcal{O}}^\reg$. 
Then we have the following isomorphisms (\emph{cf.} \cite[\S3.4]{Ohta-ES})
\[H^0(\mathcal{X}_1^\reg,\Omega^1)\simeq H^0(\widetilde{\mathcal{X}}_1,\Omega^1)\simeq\Cot(\mathcal{J}_1).\]
The first isomorphism follows from the discussion of \cite[II, 3]{Mazur-Eisenstein} and \cite[\S6]{Wiles1}, while the second follows from the discussion in \cite[\S2(e)]{Mazur-Rational}; note that both results can be applied to the case of Shimura curves because $\mathcal{X}_1^\reg$ is Cohen--Macaulay, purely of relative dimension $1$, thanks to \cite{Buz}. Now recall that $\tau_1$ induces an automorphism of $\mathcal{X}_1^\reg$ which interchanges the two irreducible components of the special fiber.
Using the isomorphism $H^0(X_{1/F},\Omega^1)\simeq H^0(X_{1/F}^\reg,\Omega^1)$, we can see $\omega$ as a meromorphic section in $H^0(\mathcal{X}_{1},\Omega^1)$, and hence a meromorphic section in $H^0(\mathcal{X}_{1}^\reg,\Omega^1)$. 
The divisor $D$ where $\omega$ is not regular must be contained in the special fiber of $\widetilde{\mathcal{X}}_{1}$; however, the divisor $D$ has trivial intersection with both the irreducible components of the special fiber $\mathbb{X}_{1}^\mathrm{reg}$ of $\mathcal{X}_1^\reg$, and the result follows. 
\end{remark}
 
\subsection{Inverse limit of modular forms} 
We assume as before that $F$ is a subfield of $\C_p$ which contains the maximal unramified extension of $\Q_p$, and let $\mathcal{O}$ be its valuation ring. 

The trace maps introduced in \S\ref{secHeckeop} can be extended to higher weight and reinterpreted in modular terms as follows:

\begin{itemize}
\item $\mathrm{Tr}_m^{(m)}\colon S_k(\Gamma_m,F)\rightarrow S_k(\Gamma_{m-1}^{(m)},F)$ is given by $\mathrm{Tr}_m^{(m)}(f)(T)=\sum_i f(T_i)$ where, given $T=(A,\iota,\alpha,(P,H),\omega)$ a test object for $\Gamma_{m-1}^{(m)}$, $T_i=(A,\iota,\alpha,(P_i,H),\omega)$ are all test objects for $\Gamma_m$ such that $P=pP_i$;

\item $\mathrm{Tr}_{m-1}^{(m)}\colon S_k(\Gamma_{m-1}^{(m)},F)\rightarrow S_k(\Gamma_{m-1},F)$ is given by $\mathrm{Tr}_{m-1}^{(m)}(f)(T)=\sum_i f(T_i)$ where, given a test object $T=(A,\iota,\alpha,(P,H),\omega)$ for $\Gamma_{m-1}$, $T_i=(A,\iota,\alpha,(P,H_i),\omega)$ are all test objects for $\Gamma_{m-1}^{(m)}$ such that $H_i$ are all finite flat subgroups of rank $p^m$ with $P\in H_i$;

\item $\mathrm{Tr}_m\colon  S_k(\Gamma_m,F)\rightarrow S_k(\Gamma_{m-1},F)$ is given by $\mathrm{Tr}_m=\mathrm{Tr}_{m-1}^{(m)}\circ \mathrm{Tr}_m^{(m)}$ which 
takes $f$ to $\mathrm{Tr}_m(f)(T)=\sum_if(T_i)$ where, given a test object $T=(A,\iota,\alpha,(P,H),\omega)$ for $\Gamma_{m-1}$, $T_i=(A,\iota,\alpha,(P_i,H_i),\omega)$ is a test object for $\Gamma_{m}$ with $pP_i=P$ and $pH_i=H$.
\end{itemize}

\begin{definition}\label{DEFINT} Let $\mathfrak{S}_{k,m}^*$ be the $\mathcal{O}$-submodule of $S_k^*(\Gamma_m,\mathcal{O})$ consisting of those $f$ such that $f|\tau_m| U_p^m\in S_k^\ST(\Gamma_m,\mathcal{O})$. Let $\mathfrak{S}_{k,\infty}^* $ denote the $\mathcal{O}$-submodule of $\invlim S_k(\Gamma_m,F)$ (inverse limit with respect to the trace maps $\mathrm{Tr}_m$ for $m\geq 1$) consisting of those $(f_m)_{m\geq 1}$ with $f_m\in \mathfrak{S}_{k,m}^*$. 
\end{definition}

For $k=2$, the case we are especially interested in, we simply drop the weight from the notation in the definition above writing $\mathfrak{S}^*_{m}$ for $\mathfrak{S}^*_{2,m}$ and $\mathfrak{S}^*_{\infty}$ for $\mathfrak{S}^*_{2,\infty}$.

{
We give a description of some elements in $\mathfrak{S}_{k,m}^*$. 

\begin{proposition}\label{propinteg}
Let $S$ be a $\mathcal{O}$-submodule of $S^*_k(\Gamma_m,\mathcal{O})$ which is $U_p^*$-stable. Then 
$S\subseteq\mathfrak{S}_{k,m}^*$. 
\end{proposition}

\begin{proof} Since $S\subseteq S_k^*(\Gamma_m,\mathcal{O})$, we only need to check that $f|\tau_m|U_p^m\in S_k^\ST(\Gamma_m,\mathcal{O})$ for each $f\in S$. 
Since $S$ is $U_p^*$-stable, we see that for each $f\in S$ we have $f|(U_p^*)^m=f|\tau_m|U_p^m|\tau_m\in S$. Since $S\subseteq S_k^*(\Gamma_m,\mathcal{O})$, then $(f|\tau_m|U_p^m|\tau_m)|\tau_m=f|\tau_m|U_p^m\in S_k^\ST(\Gamma_m,\mathcal{O})$. 
\end{proof}

\begin{remark}
The proof of Proposition \ref{propinteg} actually shows more, namely that if $f\in S\subseteq S^*_k(\Gamma_m,\mathcal{O})$ and $S$ is $U_p^*$-stable, then $f|\tau_m|U_p^n\in S$ for \emph{any} integer $n\geq 0$. 
\end{remark}

\begin{remark}\label{remohta}
It is not clear if $U_p^*$ stabilizes the whole of any of the $\mathcal{O}$-submodules $S_k^\ST(\Gamma_m,\mathcal{O})$, $S^*_k(\Gamma_m,\mathcal{O})$ or $S'_k(\Gamma_m,\mathcal{O})$. The general problem is that, even if explicit formulas are available for the action of the $U_p$ operator (like in, for example, \cite[\S4.4]{Brooks}), the space $S_k^\ST(\Gamma_m,\mathcal{O})$ seems difficult to study because the $U_p$ and $U_p^*$ actions involve Serre--Tate expansion at points in the Igusa tower which are \emph{different} from the one which is used to {define}  $S_k^\ST(\Gamma_m,\mathcal{O})$  itself. It might be possible that smaller submodules are needed to develop a theory analogue to that of Hida and Ohta. 
However, in the positive direction, note the following simple fact. 
    If $f\in S'_k(\Gamma_1,\Z_p^\unr)$, then in particular $f\in S_k^\ST(\Gamma_1,\Z_p^\unr)$, and it may be possible to show 
    that in this case $f|\tau_m|U_p\in S_k^\ST(\Gamma_1,\Z_p^\unr)$ as well, using an argument in 
    \cite[Proposition 4.17]{Brooks} (this argument involves the operator $V$, which is close to our involution $\tau_m$). 
    Therefore, one may indeed check that $S_k'(\Gamma_1,\Z_p^\unr)\subseteq \mathfrak{S}_{k,1}^*$. Finally, note that for $m=1$ and $k=2$,  
    $S_2'(\Gamma_1,\Z_p^\unr)$ maps to $S_2'(\Gamma_1,\mathcal{O})$ for any $\mathcal{O}\supseteq\Z_p^\unr[\zeta_p]$, 
    and this last $\mathcal{O}$-module is identified with $\Cot(\mathcal{J}^0_{1/\mathcal{O}})$ by Remark \ref{propGross}. In the following, the case of 
    cotangent spaces is actually the one we will mostly be focused in, but it would be interesting to develop a more general theory of integrality using Serre--Tate expansions.   
\end{remark}
}

%\begin{remark}
    %Let $f\in S_k(\Gamma_m,F)$ be an ordinary form such that $f|\tau_m\in S_k^\ST(\Gamma_m,\mathcal{O})$ and %$f|\tau_m=wf$ for some $w\in F^\times$.  
    %Then $f|\tau_m|U_p^n=a_p^nwf$ for some $a_p\in\mathcal{O}^\times$, and for all $0\leq n\leq m$, so  $f\in %\mathfrak{S}_m^*$. 
%\end{remark}

\subsection{Power series}\label{sec:measures}
{Generalizing the $\Gamma^{(m)}_{m-1}$ groups of \S\ref{sec.modforms}, define $\Gamma^{(m)}_j=\Gamma_j\cap(R^\times_m)_1$ with $1\le j\le m$ to be the subgroup of $\Gamma_j$ consisting of those $\gamma$ whose image in $\M_2(\Z_\ell)$ is congruent to $\smallmat{*}{*}{0}{*}$ modulo $\ell^{\val_\ell(Mp^m)}$ for all $\ell\nmid D$, and similarly for the local at $\ell$ counterparts. Under the strong approximation isomorphisms $\Gamma^{(m)}_j/\Gamma_m\simeq\Gamma^{(m)}_{j,p}/\Gamma_{m,p}$ we may fix a system of representatives $\{\sigma_a:a\in U^{(j)}/U^{(m)}\}$ such that $\sigma_a$ is sent to a matrix of the form $\smallmat{a}{*}{0}{a^{-1}}$.}

Fix an element 
$(f_m)_{m\geq 1}$ in $\mathfrak{S}_\infty^*$. 
Recall $U^{(m)}=1+p^m\Z_p$ for each $m\geq 1$, and let $\widehat{U^{(1)}/U^{(m)}}$ be the group of $\C_p^\times$-valued 
characters of the quotient group $U^{(1)}/U^{(m)}$. Define for each finite order character $\varepsilon\colon U^{(1)}/U^{(m)}\rightarrow \mathcal{O}^\times_\varepsilon$, where $\mathcal{O}_\varepsilon$ is the 
extension of $\mathcal{O}$ generated by the values of $\varepsilon$, the modular form
\begin{equation}\label{gm}
g_m(\varepsilon)=\sum_{a\in U^{(1)}/U^{(m)}}\varepsilon(a)\cdot f_m|\tau_m|U_p^m|\sigma_a^{-1}.\end{equation}
By definition $f_m|\tau_m|U_p^m\in S_2^\ST(\Gamma_m, \mathcal{O}_\varepsilon)$. 
Therefore, by Lemma \ref{LemmaKatz1}, $g_m(\varepsilon)\in S_2^\ST(\Gamma_m, \mathcal{O}_\varepsilon)$ 
and thus we may take its $T_x$-expansion 
$g_m(\varepsilon)(T_x)\in \mathcal{O}_\varepsilon[[T_x]].$ 

Let $\widehat{U}_{\mathrm{fin}}^{(1)}$ denote the subgroup of finite order characters 
of the $\overline{\Q}_p^\times$-valued character group of $U^{(1)}$. 
Define a function $\Phi\colon \widehat{U}_{\mathrm{fin}}^{(1)}\rightarrow \mathcal{O}_\varepsilon[[T_x]]$ by  
$\Phi(\varepsilon)=g_m(\varepsilon)(T_x)$
for \emph{any} $m$ such that $\varepsilon$ factors through $U^{(1)}/U^{(m)}$. To check the independence 
on the choice of $m$, note that
\[\begin{split}
  \sum_{a\in U^{(1)}/U^{(m+1)}}\varepsilon(a)\cdot f_{m+1}|\tau_{m+1}|U_p^{m+1}|\sigma_a^{-1}&=
  \sum_{b\in U^{(1)}/U^{(m)}}\varepsilon(b)\cdot\mathrm{Tr}_{m+1}(f_{m+1}|\tau_m|U_p^m|\sigma_b^{-1})\\
  &=\sum_{b\in U^{(1)}/U^{(m)}}\varepsilon(b)\cdot f_m|\tau_m|U_p^m|\sigma_b^{-1}
\end{split}\]where the last equation follows from Lemma \ref{lemmatrace}.

\begin{proposition}\label{propohta}
For each integer $n\geq 1$, the map which takes  
$\varepsilon$ to the $n$-th coefficient $a_n(\varepsilon)$ 
of the $T_x$-power series $\Phi(\varepsilon)= \sum_{n\geq 0}a_n(\varepsilon)T_x^n$ satisfies 
condition $(*)$ in \cite[page 70]{Ohta-ES}, \emph{i.e.}
\[\sum_{\varepsilon\in\widehat{U^{(1)}/U^{(m)}}}\varepsilon(\alpha)^{-1}a_n(\varepsilon)\in p^{m-1}\mathcal{O}\]
for all $m\geq 1$ and all $\alpha\in U^{(1)}$. 
\end{proposition}
\begin{proof} Write $\varphi(a)=(f_m|\tau_m|U_p^m|\sigma_a^{-1})$ for
$a\in U^{(1)}/U^{(m)}$ to simplify the notation. Then we have 
\[\begin{split}
  \sum_{\varepsilon\in\widehat{U^{(1)}/U^{(m)}}}  \varepsilon(\alpha)^{-1}g_m(\varepsilon) &=
\sum_{\varepsilon\in\widehat{U^{(1)}/U^{(m)}}} \varepsilon(\alpha)^{-1}
\left(\sum_{a\in U^{(1)}/U^{(m)}}\varepsilon(a)\varphi(a)\right)\\
&=\sum_{a\in U^{(1)}/U^{(m)}}\varphi(a)\left(\sum_{\varepsilon\in\widehat{U^{(1)}/U^{(m)}}}\varepsilon(a\alpha^{-1})\right)
\end{split}\] and $\sum_{\varepsilon\in\widehat{U^{(1)}/U^{(m)}}}\varepsilon(a)\equiv 1\pmod {p^{m-1}}$. 
\end{proof}

Let $\Lambda_\mathcal{O}\defeq \mathcal{O}[[U^{(1)}]]$. It follows from Proposition \ref{propohta} and \cite[Lemma (2.4.2)]{Ohta-ES} that for each $(f_m)_{m\geq 1}$ in $\mathfrak{S}_\infty^*$ we can associate a power series\[\mathcal{F}(T_x)=\sum_{n\geq 0}a_nT_x^n\in\Lambda_\mathcal{O}[[T_x]]\] 
such that for each $\varepsilon$ as before we have 
$\mathcal{F}_\varepsilon(T_x)=\Phi(\varepsilon)$,
where 
\[\mathcal{F}_\varepsilon(T_x)= \sum_{n\geq 0}a_n(\varepsilon(u)-1)T_x^n\]
and $u$ is a topological generator of $U^{(1)}$ 
(to explain the notation, 
the choice of $u$ fixes the isomorphism of $\mathcal{O}$-algebras 
$\Lambda_{\mathcal{O}}=\mathcal{O}[[1+p\Z_p]]\simeq  \mathcal{O}[[X]]$ that takes 
$u$ of $U^{(1)}=1+p\Z_p$ to $X+1$, so if we identify $a\in \Lambda_\mathcal{O}$ 
with a power series $a(X)\in \mathcal{O}[[X]]$, 
the value $\varepsilon(a)(X)$ is the value of $a$ at $X=\varepsilon(u)-1$). 
We therefore obtain a 
map $\mathfrak{S}_\infty^*\rightarrow \Lambda_\mathcal{O}[[T_x]]$. Note that the action of $U^{(1)}/U^{(m)}$ 
on $S_2(\Gamma_m,F)$ equips $\invlim S_2(\Gamma_m,F)$ with a canonical structure of $\Lambda_\mathcal{O}$-modules, so the aforementioned map then becomes an injective map
\begin{equation}\label{STexp}
\mathfrak{S}_\infty^*\longmono \Lambda_\mathcal{O}[[T_x]]\end{equation} of $\Lambda_\mathcal{O}$-modules (the injectivity follows immediately from the fact that modular forms are non-vanishing sections of differentials). 

{\begin{remark}\label{remwiles}
    One can define \emph{Wiles $\Lambda$-adic modular forms} to be the 
    $\mathcal{O}$-submodule of $\Lambda_\mathcal{O}[[T_x]]$ consisting of those power series $\mathcal{F}(T_x)= \sum_{n\geq 0}a_nT_x^n$ such that for all arithmetic morphisms $\kappa$ of $\Lambda_\mathcal{O}$, except possibly a finite number of them, 
    $\mathcal{F}_\kappa(T_x)= \sum_{n\geq 0}a_n(\kappa)T_x^n$ is the $T_x$-expansion of a quaternionic modular form in $S_k(\Gamma_m, \mathcal{O}_\varepsilon)$
 (where $a_n(\kappa)=\kappa(a_n)$, and $\kappa$ has weight $k$ and level $\Gamma_m$). Without a good theory of $U_p$-actions on such formal series, it seems difficult to develop a theory analogue to that of Hida--Ohta, and show a relation between 
 Wiles $\Lambda$-adic modular forms and $\mathfrak{S}_\infty^*$ similar to that in \cite[Theorem 2.3.6]{Ohta-ES}. 
\end{remark}}

\section{$p$-divisible groups of Shimura curves}
The goal of this section is to extend some results of Ohta \cite{Ohta-ES, OhtaC, OhtaMA} to the case of Shimura curves. Although we follow closely these references, some proofs are obtained by combining the original results over modular curves with the Jacquet--Langlands correspondence; for this reason, we added some of the details (sometimes overlapping with those of Ohta), while proofs which are virtually identical to the original ones are simply omitted. 

\subsection{Quotients of Jacobians of Shimura curves} 
Let $J_m$ denote as before the Jacobian variety of $X_m$ for $m\geq 0$. Recall the Shimura curves 
${X}_{m-1}^{(m)}$ for $m\geq 1$ introduced in \S\ref{sectionAL}, and let ${J}_{m-1}^{(m)}$ be its Jacobian variety. 
Then we have canonical projection maps $\pi_m\colon X_m\rightarrow X_{m-1}^{(m)}$ and $\varpi_m\colon X_{m-1}^{(m)}\rightarrow X_{m-1}$, which induce pull-backs and pushforwards on the level of the Jacobians. Note that the kernel of $\pi_m^*\colon J_{m-1}^{(m)}\rightarrow J_m$ is finite for all $m\geq 1$. We recursively define three families of quotients of $J_m$ as $m$ varies as follows (in all three items, $\cdot^0$ denotes the connected component containing the identity of the corresponding subvariety):
\begin{itemize}
\item The family $\alpha_m\colon J_m\to A_m$ starts with $A_1=J_1$ and $\alpha_1=\mathrm{id}_{J_1}$. Assuming we have $\alpha_{m-1}$, set $K_m=\ker(\alpha_{m-1}\circ\varpi_{m,*})$ and let $\alpha_m\colon J_m\rightarrow A_m=J_m/\pi_m^*(K_m)^0$ be the canonical projection;

\item The family $\beta_m\colon J_m\to B_m$ (\emph{good quotients}) starts with $B_1$ being the quotient of $J_1$ by its maximal abelian subvariety which has multiplicative reduction at $p$ and $\beta_1\colon J_1\rightarrow B_1$ the canonical projection map. Assuming we have $\beta_{m-1}$, let $H_m=\ker(\beta_{m-1}\circ\varpi_{m,*})$ and let $\beta_m\colon J_m\rightarrow 
B_m=J_m/\pi_m^*(H_m)^0$ be the canonical projection;
%The map $\beta_m\colon J_m\rightarrow B_m$ factors through $\alpha_m\colon J_m\rightarrow A_m$, and $\ker(A_m\rightarrow B_m)$ has multiplicative reduction at $p$. 

\item The family $\gamma_m\colon J_m\to C_m$ starts with $C_1$ be the quotient of $J_1/\pi_1^*(J^{(1)}_0)$ and let $\gamma_1\colon J_1\rightarrow C_1$ the canonical projection map. Assuming we have $\gamma_{m-1}$, let $T_m=\ker(\gamma_{m-1}\circ\varpi_{m,*})$ and let $\gamma_m\colon J_m\rightarrow 
C_m=J_m/\pi_m^*(T_m)^0$ be the canonical projection.
\end{itemize}

\begin{lemma}\label{lemma3.1} We have canonical maps $p_m\colon A_m\rightarrow C_m$ and $q_m\colon A_m\rightarrow B_m$ for all $m\geq 1$ satisfying
$p_m\circ\alpha_m=\gamma_m$ and $q_m\circ\alpha_m=\beta_m$. Moreover, $\ker(q_m)$ has multiplicative reduction at $p$. 
\end{lemma}
\begin{proof}
    The proof is by induction on $m$. The case $m=1$ is obvious taking $p_1=\gamma_1$ and $q_1=\beta_1$. 
 %   The recursive definition of $A_m$ and $C_m$ and $B_m$
 %   is summarized and compared for $m\geq 2$ by the following diagram in which the horizontal lines are exact and the dotted arrows are defined by the commutativity of the diagram, and we adopt the uniform notation $(Q_m,\varphi_m,r_m,S_m)$ for $(B_m,\beta_m,q_m,H_m)$ or $(C_m,\gamma_m,p_m,K_m)$: 
 %  {\footnotesize{\[\xymatrix{
%0\ar[r]&\pi_m^*(S_m)^0\ar[r] &  J_m\ar@{=}[d]\ar[rr]^{\varpi_m} && Q_m\ar@{-->>}[dd]\ar[r]&0\\
%0\ar[r]& \pi_m^*(S_m)\ar[r]&J_m\ar[r]& J_{m-1}\ar@{->>}[rd]^-{\varpi_{m-1}}\\
%0\ar[r] & S_m\ar[r]\ar[u]^-{\pi_m^*} & J_{m-1}^{(m)}\ar@{=}[d]\ar[u]^-{\pi_m^*}\ar[ur]^-{\rho_{m*}}  \ar[rr]^-
%{\varpi_{m-1}\circ\rho_{m*}} && Q_{m-1}\ar[r]&0\\
%0\ar[r] & K_m\ar[r]\ar[d]^-{\pi_m^*}\ar@{^(-->}[u] & J_{m-1}^{(m)}\ar[rd]^-{\rho_{m*}}\ar[d]^-{\pi_m^*}\ar[rr]^-
%{\alpha_{m-1}\circ\rho_{m*}} && A_{m-1}\ar@{->>}[u]^-{r_{m-1}}\ar[r]&0\\
%0\ar[r]&\pi_m^*(K_m)\ar[r]&J_m\ar@{->>}[r]\ar@{=}[d]& J_{m-1}\ar@{->>}[ru]^-{\alpha_{m-1}}\\
%0\ar[r]& \pi_m^*(K_m)^0\ar@{^(->}[u]\ar[r] & J_m\ar[rr]^{\alpha_m} && A_{m}\ar@{-->>}[uu]\ar[r]&0.
%} 
%\]}} 
The existence of $p_{m-1}$ (resp. $q_{m-1}$) shows that we have a canonical inclusion of $K_m$ into $T_m$ (resp. $H_m$), which induces a canonical inclusion of $\pi^*_m(K_m)^0$ into $\pi^*_m(T_m)^0$ (resp. $\pi^*_m(H_m)^0$) and therefore the existence and uniqueness of $p_m$ (resp. $q_m$) follows from the universal property of the quotient.
The last statement about the kernel of $q_m$ is clear. 
\end{proof}

On the level of cotangent spaces, we have a 
canonical map $\gamma_m^*\colon \Cot(C_m)\rightarrow \Cot(J_m)$ and a canonical isomorphism 
\[\delta_m\colon \Cot(J_m)\otimes_\Q\C\simeq H^0(X_m,\Omega_{X_m}^1)\otimes_\Q\C \simeq S_2(\Gamma_m,\C).\] 
Similarly one can define $\alpha_m^*$ and $\beta_m^*$. 

\begin{proposition}\label{prop3.1}
The image of the composition of $\delta_m\circ\gamma_m^*\colon \Cot_\C(C_m)\rightarrow S_2(\Gamma_m,\C)$ is the subspace $\mathcal{M}_m=\oplus_{i=1}^m\oplus_\psi S_2(\Gamma_i,\psi,\C)$ of $S_2(\Gamma_m,\C)$,
where, fixed $i=1,\dots,m$, the direct sum is over all characters $\psi\colon \Gamma_i\rightarrow\C^\times$ of conductor divisible by $p^i$. 
\end{proposition}

\begin{proof} 
The proof is similar to \cite[Proposition (4.1.8)]{OhtaC}, and is omitted. 
\end{proof}

\begin{corollary}\label{coro3.3}
The image of the composition of $\delta_m\circ\alpha_m^*\colon \Cot_\C(A_m)\rightarrow S_2(\Gamma_m,\C)$ is the direct sum 
$\mathcal{P}_m\oplus S_2(\Gamma_1,\C)$, where $\mathcal{P}_1=0$ and, for $m\geq 2$, $\mathcal{P}_m=\oplus_{i=2}^m\oplus_\psi S_2(\Gamma_i,\psi,\C)$,
and, as in Proposition $\ref{prop3.1}$, fixed $i=2,\dots,m$, the direct sum is over all characters $\psi\colon \Gamma_i\rightarrow\C^\times$ of conductor divisible by $p^i$. 
\end{corollary}

\begin{proof}
The proof is similar to \cite[Proposition (1.1.5)]{OhtaMA} and is omitted. 
\end{proof} 

\begin{comment}
\begin{remark}
     It is probably possible to \emph{directly} deduce  Proposition \ref{prop3.1} and Corollary \ref{coro3.3}
     from the fact that $J_m$ is isogenous to the $D$-new quotient of the Jacobian of the modular curve over $\GL_2$. However, 
     it seems difficult (but probably possible) to take trace of the $D$-quotient of the Hecke algebra in the operations involving the formation of $A_m$, $B_m$ and $C_m$, and therefore we preferred a direct approach following \cite{OhtaMA}, \cite{Ohta-ES}, \cite{OhtaC}. However, once these results 
     are established with a direct method as before, we can make use of the Jacquet--Langlands correspondence to deduce other informations from these results. 
\end{remark}
\end{comment}

Recall the element $v_p$ from \S\ref{stabilization}. 
\begin{corollary}\label{coro3.4}
The image of the composition of $\delta_m\circ\beta_m^*\colon \Cot(B_m)_\C\rightarrow S_2(\Gamma_m,\C)$ is the direct sum 
$\mathcal{M}_m\oplus S_2(\Gamma_0,\C)\oplus S_2(\Gamma_0,\C)^{v_p}$, where $S_2(\Gamma_0,\C)^{v_p}=\{f|v_p : f\in S_2(\Gamma_0,\C)\}$.
\end{corollary}

\begin{proof}
{
Consider the map $\Cot(q_m)\colon \Cot(B_m)_\C\hookrightarrow\Cot(A_m)_\C$ induced by the map
$q_m$ in Lemma \ref{lemma3.1}. We study the image of $\delta_m\circ\Cot(q_m)$ in $S_2(\Gamma_m,\C)$. If $m\geq 1$, the abelian variety $A_f$ associated to each $f\in \mathcal{M}_m$ by the Eichler--Shimura construction is isogenous to the abelian variety $A_{f_{\GL_2}}$ similarly constructed for the $\GL_2$ modular form $f_{\GL_2}$ which lifts through Jacquet--Langlands to $f$. 
Each of these forms have good reduction (see \cite[Chapter 3, Section 1]{MW}, \cite[\S1.1]{OhtaMA}), so the image of $\delta_m\circ\Cot(q_m)$ contains $\mathcal{M}_m$ for all $m\geq 1$. Since, from Corollary \ref{coro3.3}, 
$\Cot(A_m)_\C\simeq\mathcal{P}_m\oplus S_2(\Gamma_1,\C)$ (via the map $\delta_m$), it suffices to check that the image on $\delta_m$ in $S_2(\Gamma_1,\C)$ is $\mathcal{M}_1\oplus S_2(\Gamma_0,\C)\oplus S_2(\Gamma_0,\C)^{v_p}$.

Note that we have a direct sum decomposition 
$S_2(\Gamma_1,\C)\simeq \mathcal{M}_1\oplus S_2(J^{(1)}_0,\C)$. The $p$-new part of $J^{(1)}_0$ has toric reduction at $p$ (see for example \cite{LRV}) and therefore any $p$-new form $f$ in $S_2(\Gamma^{(1)}_0,\C)$ vanish in $\Cot(B_m)_\C$. On the other hand, the $p$-old part of $J^{(1)}_0$, which is precisely $S_2(\Gamma_0,\C)\oplus S_2(\Gamma_0,\C)^{v_p}$, has good reduction, so it belongs to the image of $\delta_m\circ\Cot(q_m)$, and the result follows.}
\end{proof}

\begin{comment}
\begin{proof}
It is enough to show that the image of the map $\Cot(q_m)\colon \Cot(B_m)_\C\hookrightarrow\Cot(A_m)_\C$ induced by the map
$q_m\colon A_m\twoheadrightarrow B_m$ in Lemma \ref{lemma3.1}, is mapped by $\delta_m$ to the direct sum of $\mathcal{M}_1$ and the image of 
$S_2(\Gamma_0,\C)\oplus S_2^*(X_0,\C)$ in $S_2(\Gamma_m,\C)$, for all $m\geq 1$. 
Recall from Corollary \ref{coro3.3} that 
$\Cot(A_m)_\C\simeq\mathcal{P}_m\oplus S_2(\Gamma_1,\C)$ (via the map $\delta_m$). If $m\geq 1$, the abelian variety $A_f$ associated to each $f\in \mathcal{M}_m$ by the Eichler--Shimura costruction is isogenous to the abelian variety $A_{f_{\GL_2}}$ associated by the Eichler--Shimura construction 
to the $\GL_2$ modular form $f_{\GL_2}$ which lifts through Jacquet--Langlands to $f$. 
Each of these forms has good reduction (see \cite[Chapter 3, Section 1]{MW}, \cite[\S1.1]{OhtaMA}), so the image of $\delta_m\circ\Cot(q_m)$ contains $\mathcal{M}_m$ for all $m\geq 1$. 
Note that we have a direct sum decomposition 
$S_2(\Gamma_1,\C)\simeq \mathcal{M}_1\oplus S_2(J^{(1)}_0,\C)$; the $p$-new part of $J^{(1)}_0$ has toric reduction at $p$ (see for example \cite{LRV}) and therefore any $p$-new form $f$ in $S_2(\Gamma^{(1)}_0,\C)$ does not belong to the image of $\delta_m\circ\Cot(q_m)$. On the other hand, the $p$-old part of $J^{(1)}_0$ has good reduction, so it belongs to the image of $\delta_m\circ\Cot(q_m)$, and the result follows.
\end{proof}
\end{comment}

In the following, when the context is clear, we often omit the map $\delta_m$. Using the results collected so far, we now compare $A_m$, $B_m$ and $C_m$ with the analogous objects for the $\GL_2$-case, extensively studied by Mazur--Wiles \cite{MW} and Ohta \cite{OhtaJ, Ohta-ES, OhtaC, OhtaMA} (note that this approach, based on the Jacquet--Langlands correspondence, is already used in the proof of Corollary \ref{coro3.4}). For this, we need to introduce the analogue of $A_m$ and $B_m$ in the case of $\GL_2$ considered in \cite[\S1.1]{OhtaMA} and
\cite[\S4.1]{OhtaC}. We introduce the following notation, where $r\geq 1$ is an integer: 
\begin{itemize}
    \item We denote by $A_r^\Oh$ and $Q_r^\Oh$ the abelian varieties denoted $\mathcal{A}_r$ and $\mathcal{Q}_r$ in \cite[\S1.1]{OhtaMA}; in the case of Shimura curves considered in this paper (\emph{i.e.} for $D>1$) the absence of cusps shows that both constructions of $A_r^\Oh $ and $Q_r^\Oh $ reduce to the construction of the variety we denoted $A_m$ before. 
    \item We denote by $B_r^\Oh $ the good quotient of the modular Jacobian introduced in the last paragraph of \cite[\S1.1]{OhtaMA} and denoted $A_r$ in \emph{loc. cit.} In the case of Shimura curves considered in this paper (\emph{i.e.} for $D>1$) this corresponds to our varieties $B_m$.  
    \item We finally denoted $C_r^\Oh$ the abelian variery denoted by $Q_r$ in \cite[\S4.1]{OhtaC}; In the case of Shimura curves considered in this paper (\emph{i.e.} for $D>1$) this corresponds to our varieties $C_m$.  
\end{itemize} 
\begin{remark} We apologize to the reader for the conflict of notation regarding the symbol $A_r$, used in this paper and in \cite[\S1.1]{OhtaMA} with a different meaning; however, note that in \cite[\S4.1]{OhtaC} the symbol $B_r$ is used for good quotients, as in this paper. 
In any case, the meaning of the symbols should be entirely clarified in the lines preceding this remark.
\end{remark}

The abelian varieties $A_m^\Oh$, $B_m^\Oh$ and $C_m^\Oh$ are all equipped with the action of the full Hecke algebra $\tilde h_m$ for $\GL_2$ generated by standard Hecke operators $T_\ell$ for $\ell\nmid MDp^m$ and $U_\ell$ for $\ell\mid MDp^m$;  
if we denote $\tilde h_m^\new$ the $D$-new quotient of $\tilde h_m$, we may form the $D$-new quotients $ A_m^\Oh\otimes_{\tilde h_m}\tilde h_m^\new$,  
$B_m^\Oh\otimes_{\tilde h_m}\tilde h_m^\new$
and 
$C_m^\Oh\otimes_{\tilde h_m}\tilde h_m^\new$ of $A_m^\Oh$, $B_m^\Oh$ and $C_m^\Oh$ and take abelian subvarieties $A_m^{\Oh,\new}$, $B_m^{\Oh,\new}$ and $C_m^{\Oh,\new}$ of $A_m^\Oh$, $B_m^\Oh$ 
and $C_m^\Oh$ which are isogenous to $ A_m^\Oh\otimes_{\tilde h_m}\tilde h_m^\new$, $B_m^\Oh\otimes_{\tilde h_m}\tilde h_m^\new$  and 
$C_m^\Oh\otimes_{\tilde h_m}\tilde h_m^\new$, respectively. 
Let $X_1(MDp^m)$ be the compact modular curve of 
    level $\Gamma_1(MDp^m)$ and let $J_1(MDp^m)=\mathrm{Jac}(X_1(MDp^m))$ be its Jacobian; 
    taking the $D$-new quotient $J_1^\new(MDp^m)=J_1(MDp^m)\otimes_{\tilde h_m} \tilde{h}_m^\new$ 
     gives an abelian
    variety which is isogenous to $J_m$; we also denote with the same symbol $J_1^\new(MDp^m)$ an abelian subvariety of $J_1(MDp^m)$ which is isogenous to the $D$-new quotient.  

\begin{proposition}\label{isogenyprop} 
    There are $\Q$-isogenies $A_m\sim A_m^{\Oh,\new}$, $B_m\sim B_m^{\Oh,\new}$ and $C_m\sim C_m^{\Oh,\new}$.
\end{proposition}
\begin{proof} This follows from the Jacquet--Langlands correspondence; the reader is referred 
to \cite[\S1.6]{BD} or \cite[\S2.2]{LRdVP} for the properties of the Jacquet--Langlands correspondence that are need in this proof and in other similar proofs.
For any $f\in S_2(\Gamma_m,\C)$ which is an eigenform for all Hecke operators, choose a form $f_{\GL_2}\in S_2(\Gamma_0(MD)\cap\Gamma_1(p^m))$ 
which has the same eigenvalues of $f$ under the action of the Hecke operators $T_\ell$ for $\ell\nmid MDp$ and $U_\ell$ for $\ell\mid Mp$; the lift $f_\mathrm{GL_2}\mapsto f$ is not canonical, but the lines spanned by $f$ and $f_\mathrm{GL_2}$ are, as well as the abelian varieties $A_f$ and $A_{f_\mathrm{GL_2}}$ attached to $f$ and $f_\mathrm{GL_2}$ by the Eichler--Shimura construction. The two abelian varieties $A_f$ and $A_{f_\mathrm{GL_2}}$ are defined over $\Q$ and are isogenous. Combining Proposition \ref{prop3.1} and \cite[Proposition (4.1.8)]{OhtaC} for $C_m$, 
Corollary \ref{coro3.3} and \cite[Proposition (1.1.5)]{OhtaMA} for $A_m$, Corollary \ref{coro3.4} and the last paragraph of \cite[\S1.1]{OhtaMA} for $B_m$, 
we see that the factors appearing on the two abelian varieties that we want to compare are the same, concluding the proof. 
\end{proof} 

\begin{remark} With an abuse of notation, we write $f_\mathrm{GL_2}\mapsto f$ the Jacquet--Langlands correspondence in the proof of Proposition \ref{isogenyprop}; however, note that we canonically 
only have a map which takes the subspace $\langle f_\mathrm{GL_2}\rangle $ of the space of weight $2$ modular forms on $\GL_2$ of level $\Gamma_1(MDp^m)$ to the subspace $\langle f\rangle$ of $S_2(\Gamma_m,\C)$.
\end{remark}

\begin{lemma} 
The abelian variety $A_m$ has semistable reduction over $\Q_p(\zeta_{p^m})$.
\end{lemma}
\begin{proof} 
By Corollary \ref{coro3.3}, we know that $A_m$ is isogenous to the direct sum of  
$J_1$ and the abelian varieties corresponding to newforms in $\mathcal{P}_m$. If $f\in \mathcal{P}_m$, then using the Jacquet--Langlands correspondence as in the proof of Proposition \ref{isogenyprop}, there is a unique newform $f_\mathrm{GL_2}$ of level $\Gamma_0(MD)\cap\Gamma_1(p^i)$ and primitive character of conductor $p^i$ for some $2\leq i\leq m$ such that the abelian varieties $A_f$ and $A_{f_\mathrm{GL_2}}$, associated by Eichler--Shimura construction 
to $f$ and $f_\mathrm{GL_2}$ respectively, are isogenous. Now $A_{f_\mathrm{GL_2}}$ has semistable reduction over $\Q_p(\zeta_{p^m})$ (because a toric quotient of it has good reduction, see  
\cite[Chapter 3, Section 2, Proposition 1]{MW} and \cite[page 562]{OhtaMA} for details) and therefore the same is true for $A_f$. The conclusion follows because $J_1$ has semistable reduction (already over $\Q$) by \cite[Theorem 4.10]{Buz}.  
\end{proof}

Let $\mathcal{S}$ denote the $Mp$-old subspace of $S_2(\Gamma_1,\C)\subseteq S_2(\Gamma_m,\C)$,  
spanned by all $f|v_d$ for all divisors $d\mid Mp$ and newforms $f$ of the relevant level (here $v_d$ is defined 
as in \S\ref{stabilization} via the matrices $\smallmat d001$, see \cite[\S2.1]{Hijikata} for details). 
Then there is a unique quotient $F_m$ of $J_m$ by a $\Q$-rational abelian subvariety such that the pull-back of $\Cot(F_m)_\C$ via the cotangent of the canonical projection corresponds to $\mathcal{S}$; this abelian variety $F_m$ is isogenous to the product of abelian subvarieties (possibly with multiplicity) of the forms in $\mathcal{S}$. By Lemma \ref{lemma3.1}, the canonical homomorphism $J_m\rightarrow F_m\times B_m$ factors through $\alpha_m$. By comparing Corollary \ref{coro3.3} and Corollary \ref{coro3.4}, we see that the cokernel of the map $\Cot(q_m)\colon 
\Cot(B_m)_\C\rightarrow\Cot(A_m)_\C$ is isomorphic to $\mathcal{S}$, and therefore the canonical homomorphism $J_m\rightarrow F_m\times B_m$
induces an isogeny \begin{equation}\label{isogeny}
A_m\longrightarrow F_m\times B_m.\end{equation} 

\subsection{$p$-divisible groups}
For an abelian variety $A$, denote $A[p^\infty]$ its $p$-divisible group and $\Ta_p(A)$ its $p$-adic Tate module. {Let $e^\ord=\lim_{n\to\infty} U_p^{n!}$ denote the Hida ordinary projector.}

\begin{proposition}\label{isopdiv1}
The map $\alpha_m$ induces an isomorphism $e^\ord J_m[p^\infty]\simeq e^\ord A_m[p^\infty]$ of $p$-divisible groups, $\forall m\geq 1$.  
\end{proposition}

\begin{proof} Write $\mathcal{K}_m=\pi^*_m(K_m)^0$ for the kernel of $\alpha_m\colon J_m\rightarrow A_m$ to simplify the notation. Since $\mathcal{K}_m$ is connected, we have an exact sequence of $p$-divisible groups 
\[0\longrightarrow \mathcal{K}_m[p^\infty]\longrightarrow J_m[p^\infty]\longrightarrow A_m[p^\infty]\longrightarrow 0.\]
It is enough to show that $U_p$ does not have $p$-adic unit eigenvalues on \[\Ta_p(\mathcal{K}_m)\otimes_{\Z_p}\Q_p\simeq H_1(\mathcal{K}_m(\C),\Q)\otimes_\Q\Q_p\] and on \[\Cot(\mathcal{K}_m)_\C\oplus\overline{\Cot(\mathcal{K}_m)}_\C\simeq H^1(\mathcal{K}_m(\C),\C).\]
Thanks to Corollary \ref{coro3.3}, this follows from Jacquet--Langlands and the analogue result \cite[Proposition 4.1]{hida-measure} of Hida for $\GL_2$. 
\end{proof}

\begin{remark} Alternatively, since the Hecke operator $U_p$ is equivariant with respect to the Jacquet--Langlands correspondence, Proposition \ref{isopdiv1}
can also be obtained using a slightly different argument. Take the $D$-new quotient $J_1^\new(MDp^m)$ 
of the Jacobian $J_1(MDp^m)$, which is isogenous to $J_m$; since $e^\ord$ commutes with taking the $D$-new quotient, so from \cite[Proposition (1.2.1)]{OhtaMA} we obtain an isogeny between $e^\ord J_m[p^\infty]$ and $e^\ord A^{\Oh,\new}_m[p^\infty]$ . The result follows then from 
Proposition \ref{isogenyprop}.
\end{remark}
 
Let $\mathcal{A}_m$ be a N\'eron model of $A_m$ over $\Z[\zeta_{p^m}]$. 
We denote $\mathcal{A}_m^0$ the connected component of $\mathcal{A}_m$ containing the identity (\cite[\S1b)]{Mazur-Rational}). 
For each $n\geq 1$ let $\mathcal{A}^0_m[p^n]$ be the connected part of the $p^n$-torsion subgroup scheme $\mathcal{A}_m[p^n]$ of $\mathcal{A}_m$, let $(\mathcal{A}_m^0[p^n])^\mathrm{f}$ be the Grothendieck \emph{fixed part} (for the definition, see \cite[IX (2.2.2)]{Groth} -- more generally for any $p$-divisible group $G$ we denote $G^\mathrm{f}$ its fixed part) and define the $p$-divisible group 
\[\mathcal{A}^0_m[p^\infty]^\mathrm{f}\defeq\dirlim_n(\mathcal{A}_m^0[p^n])^\mathrm{f}.\]
Finally, define 
\[\mathcal{G}_m\defeq e^\ord\mathcal{A}^0_m[p^\infty]^\mathrm{f}\]
and let $\mathcal{G}_m^0$ be the connected part of $\mathcal{G}_m$.

\begin{proposition}\label{multiplicativetype}
$\mathcal{G}_m^0$ is of multiplicative type and its height is half of the height of $e^\ord J_m[p^\infty]$. 
\end{proposition}

\begin{proof} The proof is similar to \cite[Proposition (1.2.4)]{OhtaMA}. Let $\mathcal{B}_m$ and $\mathcal{F}_m$ be Néron models of $B_m$ and $F_m$, respectively, over $\Z[\zeta_{p^m}]$. Recall the isogeny $A_m\rightarrow F_m\times B_m$ introduced in \eqref{isogeny}, from which we obtain an isogeny of $p$-divisible groups
\[\mathcal{G}_m\longrightarrow (e^\ord \mathcal{F}_m^0[p^\infty]^\mathrm{f})\times_{\Z[\zeta_{p^m}]}(e^\ord \mathcal{B}^0_m[p^\infty]^\mathrm{f}).\] It is therefore enough to prove the result for the two $p$-divisible groups 
$e^\ord \mathcal{F}_m^0[p^\infty]^\mathrm{f}$ and $e^\ord \mathcal{B}^0_m[p^\infty]^\mathrm{f}$. The argument for both is similar. 

For $e^\ord\mathcal{F}_m^0[p^\infty]^\mathrm{f}$, the result follows
from the Jacquet--Langlands correspondence, which is compatible with the action of $U_p$ on both sides, and the analogue result \cite[Proposition (1.2.4)]{OhtaMA} for the abelian varieties associated with forms on $\GL_2$; more explicitly, $\mathcal{F}_m^0[p^\infty]^\mathrm{f}\otimes\mathbb{F}_p$ is the $p$-divisible group associated with a torus, so it is of multiplicative type and its height is equal to half of the dimension of $F_m$; furthermore, $U_p$ acts invertibly on it by \cite[Lemma 3.2]{hida-measure}, so $\mathcal{F}_m^0[p^\infty]^\mathrm{f}=e^\ord \mathcal{F}_m^0[p^\infty]^\mathrm{f}$.  

By Proposition \ref{isogenyprop}, $B_m$ is isogenous to the abelian subvariety $B_m^{\Oh,\new}$ of $B_m^\Oh$; let $\mathcal{B}_m^{\Oh,\new}$ and $\mathcal{B}_m^\Oh$ be Néron models of $B_m^{\Oh,\new}$ and $B_m^\Oh$, respectively, over $\Z[\zeta_{p^m}]$. Now 
$e^\ord(\mathcal{B}_m^{\Oh})^0[p^\infty]^\mathrm{f}$ is equal to $e^\ord \mathcal{B}_m^{\Oh}[p^\infty]$ and is of multiplicative type (see the proof of \cite[Proposition (1.2.4)]{OhtaMA}), so the same is true $e^\ord \mathcal{B}^0_m[p^\infty]^\mathrm{f}$. 
Now, the height of $e^\ord \mathcal{B}_m^\Oh[p^\infty]$ is half of the dimension of $B_m^{\Oh}$ (this follows from the proof of \cite[Proposition (1.2.4)]{OhtaMA}). The heights of the $p$-divisible groups $e^\ord(\mathcal{B}_m^{\Oh,\new})[p^\infty]$ and $e^\ord(\mathcal{B}_m^{\Oh}/\mathcal{B}_m^{\Oh,\new})[p^\infty]$ cannot be smaller than the dimensions of $B_m^{\Oh,\new}$ and $B_m^{\Oh}/B_m^{\Oh,\new}$, respectively; so both must be equal to the respective dimensions, and in particular  
the height of $e^\ord(\mathcal{B}_m^{\Oh,\new})[p^\infty]$ is equal to the dimension of $B_m^{\Oh,\new}$. The result follows from Proposition \ref{isogenyprop}.
\end{proof}

\begin{remark}
Alternatively, one can prove the statement regarding $e^\ord \mathcal{B}^0_m[p^\infty]^\mathrm{f}$ 
in the proof of the previous proposition by the same argument of \cite[Proposition (1.2.4)]{OhtaMA} which uses Igusa towers. Of course, this approach has the advantage of being adapted to more general Shimura varieties. 
\end{remark}

\subsection{Twisted Tate pairings} 
Recall that $\tau_m$ denotes the Atkin--Lehner involution of $S_2(\Gamma_m,\C)$ and let $\zeta_{MDp^m}$ be a primitive $MDp^m$-root of unity.

\begin{lemma}\label{lemmmaAL}
   % $\tau_m$ induces an involution of $J_m$ which is defined over $\Q(\zeta_{MDp^m})$ and for each $\sigma\in \Gal(\Q(\zeta_{MDp^m})/\Q)$ we have $\tau_m^\sigma=\sigma\circ\langle a\rangle$ if $\zeta_{MDp^m}^\sigma=\zeta_{MDp^m}^a$. 
   The Atkin--Lehner involution $\tau_m$ induces an involution of $J_m$ which is defined over $\Q(\zeta_{MDp^m})$ and for each $\sigma\in \Gal(\Q(\zeta_{MDp^m})/\Q)$ we have $\tau_m^\sigma=\sigma\circ\langle \chi_\cyc(\sigma)\rangle$, where $\chi_\cyc$ is the cyclotomic character.
\end{lemma}
\begin{proof}
   These facts are known for Jacobians $J_1(MDp^m)$ of modular curves, 
thanks to \cite[Proposition 7.8]{Shimura78} or \cite[\S3.1]{OhtaC}, and one can deduce the corresponding assertions for Jacobians of Shimura curves by taking the $D$-new quotient, and noticing that 
the Atkin--Lehner involution on the modular Jacobians takes old forms to old forms and is equivariant with respect to the Jacquet--Langlands correspondence $f_{\mathrm{GL}_2}\mapsto f$. See \cite[Theorem 1.2]{BD96} for more details, recalling that $D$ is a square-free product of an even number of prime factors to fix the sign change in \emph{loc. cit.}  
\end{proof}

Let $(\cdot,\cdot)_m$ be the usual Weil pairing on $\Ta_p(J_m)$; this is the restriction to the 
$D$-new quotient of the usual Weil pairing on the modular Jacobian, denoted with the same symbol $(\cdot,\cdot)_m$. 
Define the \emph{twisted Tate pairing} by \[[x,y]_m\defeq (x,\tau_my)_m.\] 

\begin{lemma}\label{lemmapairing}\begin{enumerate}
\item 
%Let $\sigma\in\Gal(\overline{\Q}/\Q)$ satisfy $\zeta_{MDp^m}^\sigma=\zeta_{MDp^m}^a$ for 
%$a\in (\Z/MDp^m\Z)^\times$. Then  
$[x^\sigma,\langle a\rangle y^\sigma]_m=[x,y]_m^{\chi_\cyc(\sigma)}$.
\item $[\cdot,\cdot]_m$ induces a perfect pairing on $e^\ord\Ta_p(J_m)$. \end{enumerate}
\end{lemma}
\begin{proof}
    (1) The stated formula holds if we replace $J_m$ by the modular Jacobian $J_1(MDp^m)$ (\cite[(1.2.7)]{OhtaMA}, and the lines preceding it). Taking the $D$-new quotient, we see that the same formula also holds for $e^\ord\Ta_p(J_m)$.
    
    (2) We know that $[Tx,y]_m=[x,Ty]_m$ on the modular jacobian  for any Hecke operator $T$ 
     (\cite[Lemma 4.1]{Til}), so taking $D$-new quotients we see that the same formula holds for $\Ta_p(J_m)$ and this implies that the pairing on $\Ta_p(J_m)$ (which is perfect by \emph{loc. cit.}) induces a perfect pairing on $e^\ord\Ta_p(J_m)$. 
\end{proof}

Let $Q_m$ be defined by the exactness of the following sequence: 
\begin{equation}\label{defQm}
0\longrightarrow \Ta_p(\mathcal{G}_m^0)\longrightarrow e^\ord\Ta_p(J_m)\longrightarrow Q_m\longrightarrow0,\end{equation}
where the first map is induced by Proposition \ref{isopdiv1}.

\begin{proposition}\label{prop1.2.8}
\begin{enumerate}
    \item $Q_m$ is a free $\Z_p$-module of finite rank. 
    \item $\Ta_p(\mathcal{G}_m^0)$ is isotropic with respect to the twisted Weil pairing 
    $[\cdot,\cdot]_m$, which induces a perfect pairing between $\Ta_p(\mathcal{G}_m^0)$ and $Q_m$. 
    \item $\Ta_p(\mathcal{G}_m^0)$ is stable under the action of $G_{\Q_p}=\Gal(\overline{\Q}_p/\Q_p)$. Any $\sigma$ in the inertia subgroup $I_{\Q_p}$ of $G_{\Q_p}$ acts on 
    $\Ta_p(\mathcal{G}_m^0)$ by $\chi_\cyc(\sigma)$ and acts on $Q_m$ by $\langle a\rangle$ for 
    $a=\chi_\cyc^{-1}(\sigma)$.
    %any $a\in (\Z/MDp^m\Z)^\times$ such that $\zeta_{MDp^r}^{\sigma^{-1}}=\zeta_{MDp^r}^a$.  
\end{enumerate}
\end{proposition}

\begin{proof}
The proof follows closely \cite[Proposition (1.2.8)]{OhtaMA}. 

(1) By definition, $\mathcal{G}^0_m\otimes_{\Z_p[\zeta_{p^m}]}\Q_p(\zeta_{p^m})$ is a $p$-divisible subgroup of $e^\ord A_m\otimes_\Q\Q_p(\zeta_{p^m})[p^\infty]$, so the quotient is a $p$-divisible group whose Tate module is identified with $Q_m$, which is therefore $\Z_p$-free. 

(2) Take $\sigma\in I_{\Q_p}$, $\sigma\neq1$, which acts as the identity on $\Q(\zeta_{p^m})$. By Proposition 
\ref{multiplicativetype}, $\mathcal{G}_m^0$ is of multiplicative type, so $\sigma$ acts on $\Ta_p(\mathcal{G}_m^0)$ by $\chi_\cyc(\sigma)$, and therefore from (1) in Lemma \ref{lemmapairing} we conclude that $[x,y]_m=0$ for all $x,y\in \Ta_p(\mathcal{G}_m^0)$. Since $[\cdot,\cdot]_m$ is a perfect pairing on $e^\ord\Ta_p(J_m)$ by (2) in Lemma \ref{lemmapairing}, we see that $[\cdot,\cdot]_m$ induces a surjective homomorphism 
\begin{equation}\label{pairing2}
Q_m\longepi \Hom(\Ta_p(\mathcal{G}_m^0),\Z_p(1)).\end{equation}
The source and the target of \eqref{pairing2} are finitely generated free $\Z_p$-modules, which have the same $\Z_p$-rank by Proposition \ref{multiplicativetype} and the definition of $Q_m$ in \eqref{defQm}; so \eqref{pairing2} is an isomorphism of $\Z_p$-modules. 

(3) We first prove that $\Ta_p(\mathcal{G}_m^0)$ is also stable under the action of $\sigma$. Since $A_m$ is defined over $\Q$, the action of $G_{\Q_p}$ on its N\'eron model $\mathcal{A}_m$ over  $\Z_q(\zeta_{p^m})$ is via its action on $\Spec(\Z_q(\zeta_{p^m}))$; the universal property of N\'eron models shows that for each $\sigma\in G_{\Q_p}$ we have a map $i_\sigma$ making the following diagram commutative:
\[\xymatrix{\mathcal{A}_m\ar[r]^-{i_\sigma}\ar[d]& \mathcal{A}_m\ar[d]\\
\Spec(\Z_q(\zeta_{p^m}))\ar[r]^-{\Spec(\sigma)}&\Spec(\Z_q(\zeta_{p^m})).
}\] One then deduces that $i_\sigma$ induces an automorphism of $(\mathcal{A}[p^n]^\mathrm{f})^0$ and therefore $\Ta_p(\mathcal{G}_m^0)$ is also stable under the action of $\sigma$. 
To complete the proof, in light of (1) in Lemma \ref{lemmapairing}, it is enough to show the stated action for $Q_m$ only. We show it for $Q_m\otimes_{\Z_p}\Q_p$, which implies the assertion for $Q_m$. 
Recall the isogeny $A_m\rightarrow F_m\times B_m$ in \eqref{isogeny}; this isogeny commutes with the action of $\langle a\rangle$, so it is enough to show the relation for $F_m$ and $B_m$. The proof 
of \cite[Proposition (1.2.8)]{OhtaMA} shows that the relation hold for $B_m^{\Oh}$, and therefore for its quotient $B_m^{\Oh,\new}$, hence also for $B_m$ by Proposition \ref{isogenyprop}.  
Since $F_m$ has multiplicative reduction over $\Q_p$, we have that its N\'eron model over $\Z_p[\zeta_{p^m}]$ is the scalar extension of its N\'eron model $\mathcal{F}_m$ over $\Z_p$ (\cite[\S7.4, Corollary 4]{BLR-Neron}; by \cite[IX Propotition 5.6]{Groth}, 
$(F_m\otimes_\Q\Q_p)[p^\infty]/(\mathcal{F}_m^0[p^\infty])^\mathrm{f}\otimes_{\Z_p}\Q_p$
extends to an \'etale $p$-divisible group over $\Q_p$. It follows that $I_{\Q_p}$ acts trivially 
on $\Ta_p(F_m)/\Ta_p((\mathcal{F}_m^0)^\mathrm{f})$. On the other hand, for $\sigma\in I_{\Q_p}$ we have $a\equiv 1\pmod{MD}$, and, comparing with the definition of $\mathcal{S}$ (see the paragraph before \eqref{isogeny}), we see that $\langle a\rangle$ acts trivially also on $F_m$, and the result follows. 
\end{proof}

\subsection{Twisting by the Atkin--Lehner involution}
Define ${A}_m^*=J_m/\tau_m(\ker(\alpha_m))$. 
Note that the Hecke action on $A_m^*$ is via the Hecke operators $T_\ell^*$,  $U_\ell^*$ and $\langle a\rangle^*$, 
because $T\circ\tau_m=\tau_m\circ T^*$ for $T$ any such operator. 
As before let $\mathcal{A}_m^*$ be a N\'eron model of $A_m^*$ over $\Z_p[\zeta_{p^m}]$, $\mathcal{A}_m^{*,0}$ be
its connected component containing the identity, and set 
\[\mathcal{A}^{*,0}_m[p^\infty]^\mathrm{f}=\dirlim_n(\mathcal{A}_m^{*,0}[p^n])^\mathrm{f}.\]
{Denoting by $e^{\ord,*}$ the Hida ordinary projector associated with $U_p^*$,} define the $p$-divisible group 
\[\mathcal{G}_m^*\defeq e^{\ord,*}\mathcal{A}^{*,0}_m[p^\infty]^\mathrm{f}.\]
and, as before, define $Q_m^*$ by the exactness of the sequence
\begin{equation}\label{defQm*} 
0\longrightarrow \Ta_p(\mathcal{G}_m^{*,0})\longrightarrow e^{\ord,*}\Ta_p(J_m)
\longrightarrow Q_m^*\longrightarrow 0\end{equation}
where $\mathcal{G}_m^{*,0}$ is the connected component of the identity of $\mathcal{G}_m^*$. 
Note that $\tau_m$ induces an {automorphism} of $\mathcal{G}_m^*$ over $\Z[\zeta_{MDp^m}]$,
which is an \'etale extension of $\Z[\zeta_{p^m}]$; as a consequence of Proposition \ref{multiplicativetype}, we see that $\mathcal{G}_m^{*,0}$ is also ordinary. 

\begin{proposition}\label{isoprop1}
    For each integer $m\geq 0$ we have canonical isomorphisms
    \[e^{\ord,*}H^1_{\et}(\overline{X}_m,\Z_p)^{I_{\Q_p}}\simeq \Hom(Q_m^*,\Z_p),\]
    \[e^{\ord,*}H^1_{\et}(\overline{X}_m,\Z_p)/e^{\ord,* }H^1_{\et}(\overline{X}_m,\Z_p)^{I_{\Q_p}}\simeq \Hom(\Ta_p(\mathcal{G}_m^{*,0}),\Z_p),\]
    of free $\Z_p$-modules of finite rank, where we denote $\overline{X}_m\defeq X_m\otimes_{\Q}\overline{\Q}$. The action of $\sigma\in I_{{\Q_p}}$ on  
   $\Hom(\Ta_p(\mathcal{G}_m^{*,0}),\Z_p)$ is given by $\chi_\cyc^{-1}(\sigma)\langle \chi_\cyc(\sigma)^{-1}\rangle^*$.  
\end{proposition}

\begin{proof}
    The proof is similar to \cite[Proposition--Definition (1.3.6)]{OhtaMA}.
    By Proposition \ref{isopdiv1}, we have a canonical isomorphisms
    \[H^1_{\et}(\overline{X}_m,\Z_p)\simeq \Hom(e^{\ord,* }\Ta_p(J_m),\Z_p)\simeq \Hom(e^{\ord,* }\Ta_p(A_m^*),\Z_p).\] By Lemma \ref{lemmmaAL} and Proposition \ref{prop1.2.8}, we see that $\sigma\in I_{\Q_p}$ acts trivially on $Q_m^*$ and as $\chi_\cyc(\sigma)\iota(\chi_\cyc(\sigma))$ on $\Ta_p(\mathcal{G}_m^{*,0})$, where $\iota\colon \mathcal{Z}_{MD}\hookrightarrow \widetilde{\Lambda}_{\Z_p}$ is the inclusion of group-like elements of 
    \begin{equation}\label{ZMD}
    \mathcal{Z}_{MD}=\invlim_m(\Z/MDp^m\Z)^\times\simeq (\Z/MD\Z)^\times\times\Z_p^\times.\end{equation} 
    The statement on the action of $\sigma\in I_{\Q_p}$ on $\Hom(\Ta_p(\mathcal{G}_m^{*0}),\Z_p)$ follows.
    It also follows that $\Ta_p(\mathcal{G}_m^{*,0})^{I_{\Q_p}}=0$ and $(Q_m^*)^{I_{\Q_p}}=Q_m^*$. 
    Since all modules are $\Z_p$-free, taking $\Z_p$-duals in \eqref{defQm*} gives an exact sequence 
    \begin{equation}\label{Homexseq}
    0\longrightarrow \Hom(Q_m^*,\Z_p)\longrightarrow \Hom(e^{\ord,*}\Ta_p(J_m),\Z_p)\longrightarrow 
    \Hom(\Ta_p(\mathcal{G}_m^{*,0}),\Z_p)\longrightarrow 0\end{equation}
Taking $I_{\Q_p}$-invariants in \eqref{Homexseq} gives the first isomorphism; replacing the first isomorphism in \eqref{Homexseq} gives also the second isomorphism.
\end{proof}

 \section{Eichler--Shimura isomorphism}\label{sec.ESisom}
In this section we introduce Eichler--Shimura cohomology groups and study their Hodge--Tate exact sequences
using the results of the previous section. We then combine with the results of Section \ref{secmodforms} 
to produce a power series out of a generator of the ramified quotient of the inverse limit of the 
Eichler--Shimura cohomology groups. 
 
\subsection{Projective limits of cohomology groups}\label{secProjlim}
Define the \emph{$p$-adic Eichler--Shimura cohomology group} to be the $G_\Q=\Gal(\overline{\Q}/\Q)$-representation  
\[\ES_{\Z_p}=\invlim_m H^1_\text{\'et}(\overline{X}_m,\Z_p)\]
where as before we denote $\overline{X}_m=X_m\otimes_\Q\overline{\Q}$ and the inverse limit is taken with respect to the canonical projection maps $X_{m+1}\rightarrow X_{m}$ for $m\geq 1$.

Following \cite[\S1]{Ohta-ES} the inverse limit of \'etale cohomology groups admit a description in terms of group cohomology. First, we have a canonical isomorphism (with $\Z_p$ equipped with the trivial $\Gamma_m$-action)
\[H^1_{\et}(\overline{X}_m,\Z_p)\simeq H^1(\Gamma_m,\Z_p).\]
Then, let $\mathfrak{G}_m$ be the $p$-adic completion of $\Gamma_m$ in $\M_2(\Z_p)$ (here, we understand the fixed isomorphism $B_p\simeq\M_2(\Q_p)$). We have 
$\mathfrak{G}_1/\mathfrak{G}_m\simeq\Gamma_1/\Gamma_m$ and \[\invlim_{m'\geq m}\Gamma_m/\Gamma_{m'}\simeq Z_m\defeq\mathfrak{G}_m/\mathfrak{U},\] where $\mathfrak{U}$ is the subgroup of $\M_2(\Z_p)$ consisting of matrices of the form $\smallmat 1*01$. Note that the map $\smallmat abcd\mapsto \binom ac$ defines an isomorphism of $\Z_p$-modules of $Z_m$ with the subset of $\Z_p^2$ consisting 
of elements $\binom ac$ such that $\binom ac\equiv \binom 10$ modulo $p^m$, for all $m\geq 1$. Shapiro's lemma induces an isomorphism 
\begin{equation}\label{iso-et-gr}
\invlim _{m'\geq m}H^1_{\et}(\overline{X}_{m'},\Z_p)\simeq H^1(\Gamma_m,\Z_p[[Z_m]])\end{equation}
for all $m\geq 1$. In particular, $\ES_{\Z_p}\simeq H^1(\Gamma_1,\Lambda)$, where $\Lambda=\Z_p[[1+p\Z_p]]$.  

\begin{remark}\label{rem4.1}
This result follows by combining \cite[Theorem (2.4.4)]{OhtaJ} and \cite[Proposition (4.3.1)]{OhtaJ}, and the proof makes use of the completion $\mathfrak{F}_m$ of $\Gamma_m$ with respect to all subgroups of $\Gamma_m\cap\Gamma(p)$, where $\Gamma(p)$ is the subgroup of norm $1$ elements in $\mathcal{O}_B$ whose image in $\M_2(\Z_p)$ is congruent to the identity modulo $p$. 
\end{remark}

Let $J_m$ be the Jacobian variety of $X_m$ and $\Ta_m\defeq\Ta_p(J_m)$ be its $p$-adic Tate module. We have an isomorphism of $G_\Q$-modules  
\begin{equation}\label{eqES1}H^1_\text{\'et}(\overline{X}_m,\Z_p)(1)\simeq \Ta_m\end{equation}
(where for a $G_\Q$-module $M$ we denote $M(1)$ its Tate twist). Define the inverse limit 
\[\mathbf{T}=\invlim_{m\geq 1} \Ta_m\] taken with respect to the canonical projection maps $J_{m+1}\rightarrow J_m$ for $m\geq 1$; we then have a canonical isomorphism of $\Z_p[G_\Q]$-modules 
\[\ES_{\Z_p}(1)\simeq \mathbf{T}.\] 

The $\Z_p$-modules introduced before are equipped with actions of the Hecke and diamond operators from \S\ref{secHeckeop}. Let $T_\ell$ for $\ell\nmid MDp$ and $U_\ell$ for $\ell\mid Mp$ be the standard Hecke operators
acting as correspondences on $X_m$ (\cite[\S3.1, \S7.2]{shimura}); we also denote $\langle d\rangle$ the diamond operator for integers $d$ with $d\in (\Z/MDp^m\Z)^\times$; these operators induce a (covariant) Hecke action on singular 
cohomology groups $H^i(X_m,\Z)$ (\cite[\S8.2]{shimura}) which induces (using the singular-\'etale isomorphism) an action on 
$H^i_\text{\'et}(X_m,\Z_p)$ of Hecke operators, denoted $T_\ell^*$ and $U_\ell^*$, and diamond operators, denoted $\langle d\rangle^*$ (see more details in the modular curve case in 
\cite[\S7.3]{OhtaJ}; in the quaternionic case see also \cite[\S2.1]{LV-IJNT}). 
The Hecke and diamond operators acting as correspondences on $X_m$ induce by contravariant functoriality an action 
on $J_m$ and $\Ta_m$, denoted $T_\ell$, $U_\ell$ and $\langle d\rangle$; under the isomorphism \eqref{eqES1} 
the action of the Hecke and diamond operators $T^*$ on $H^1_\text{\'et}(X_m,\Z_p)(1)$ corresponds to the action of $T$ on $J_m$ (see \cite[\S8]{hida-galois}).

\subsection{Primitive branches of Hida families}\label{PBHF}
The Hecke and diamond operators acting as correspondences on $X_m$ also induce by contravariant functoriality an action 
on $J_m$ and $\Ta_m$, denoted $T_\ell$, $U_\ell$ and $\langle d\rangle$ as before. 
 One can define $\mathfrak{h}_m$ to be the $\mathcal{O}$-subalgebra of 
$\End_\mathcal{O}(\Ta_m\otimes_{\Z_p}\mathcal{O})$ generated by these operators (\cite[Section 6]{LV-MM}).
Put $\mathfrak{h}_\infty=\invlim\mathfrak{h}_m$, which acts on \[\mathbf{T}_\mathcal{O}=\invlim_m (\Ta_m\otimes_{\Z_p}\mathcal{O}).\] Recalling the ordinary idempotent $e^\ord$ associated with $U_p$, define $\mathfrak{h}_\infty^\ord=e^\ord\mathfrak{h}_\infty$ 
and \[\mathbf{T}^\ord_\mathcal{O}=e^\ord\mathbf{T}_\mathcal{O}.\]

We fix from now on a \emph{primitive} component $\I$ of the localization $\mathfrak{h}_\mathfrak{m}^\ord$  of 
$\mathfrak{h}_\infty^\ord$ at a maximal ideal $\mathfrak{m}$: this means that    
$\mathbb{I}$ is a complete local noetherian domain which is finitely generated as $\Lambda_\mathcal{O}$-module, 
and is obtained as the integral closure of the quotient ring $\mathcal{R}=\mathfrak{h}_\mathfrak{m}^\ord/\mathfrak{a}$ 
for a unique minimal prime ideal $\mathfrak{a}\subseteq\mathfrak{m}$. In particular, $\mathbb{I}$ is a finite flat extension of $\mathcal{O}[[1+p\Z_p]]$ where $\mathcal{O}$ is the valuation ring of a finite unramified extension of $\Q_p$.
See \cite[\S1.5]{NP} and \cite[Sections 5-6]{LV-MM} for details, where $\I$ is also called a \emph{branch} 
of a Hida family, terminology  that we adopt in this paper too. 
We denote $e_\mathcal{R}$ the projector associated with the local component $\mathcal{R}$. 

A point $\kappa\in\Spec(\I)$ is said to be \emph{arithmetic} if its restriction 
to $\Lambda$ is of the form $\kappa(u)=\varepsilon(u)u^{k-2}$, where $u$ is a topological generator of $U^{(1)}=1+p\Z_p$, 
$k\geq 2$ is an even integer and $\varepsilon$ is a character factoring through $U^{(m)}$ for some $m\geq 1$.  
We call $k$ the \emph{weight} of $\kappa$ and $\Gamma_m$ its \emph{level}
if $\varepsilon$ factors though $U^{(m)}$ and not through $U^{(m')}$ for any $m'\leq m$ (with this terminology, the trivial character has level $\Gamma_1$). We call $(k,\varepsilon)$ the \emph{signature} of $\kappa$.
Each arithmetic point $\kappa\in\Spec(\I)$ of weight $k$ corresponds to a normalized eigenform $f_{\GL_2,\kappa}\in S_k(\Gamma_0(Np^s),\psi)$ for a suitable integer $s\geq 1$ and a Dirichlet character $\psi$ modulo $Np^s$ 
(see \cite[Theorem III]{Hida-Annals}, where arithmetic points are called \emph{algebraic}, or \cite[Definition 2.1.1]{Howard-Inv} for details). In this case, we say that $\I$ \emph{passes through} $f_{\GL_2,\kappa}$, or that $f_{\GL_2,\kappa}$ is the \emph{specialization} at $\kappa$ of $\I$. 

We assume that $\mathbb{I}$ satisfies the following conditions: 
\begin{itemize}
    \item $\I$ passes through a $p$-stabilized newform $f_{\GL_2}\in S_{k_0}(\Gamma_0(Np))$ of trivial character and even weight $k_0\equiv 2\mod{2(p-1)}$;
    \item The residual Galois $p$-adic representation $\bar{\rho}_{f_{\GL_2}}\colon\Gal(\overline{\Q}/\Q)\rightarrow\GL_2(\mathbb{F})$ attached to $f_{\GL_2}$ (where $\mathbb{F}$ is a finite field of characteristic $p$) is irreducible and $p$-distinguished (the last condition means that the semisemplification of its restriction to a decomposition group at $p$ is the direct sum of two distinct characters).
    \item $\mathfrak{h}_\mathfrak{m}^\ord$ is Gorenstein (\emph{cf.} \cite[Assumption 6.2]{LV-MM}).
\end{itemize}

\subsection{Control Theorem}\label{sec.controltheorem} Let $F\subseteq\C_p$ be a complete subfield and let $\mathcal{O}$ be its valuation ring. We may then consider the cohomology groups 
$H^1(\Gamma_m,\mathcal{O})$ and form their projective limit  
\[\ES_\mathcal{O}=\invlim_mH^1(\Gamma_m,\mathcal{O})\] as before. 
The $\mathcal{O}$-modules $H^1(\Gamma_m,\mathcal{O})$ are equipped with an action of 
Hecke and diamond operators $T^*_\ell$ for $\ell\nmid MDp$ and $U^*_\ell$ for $\ell\mid Mp$ and $\langle d\rangle^*$ for $d\in (\Z/MDp^m\Z)^\times$, 
and we denote $\mathfrak{h}_m^*$ the $\mathcal{O}$-subalgebra of 
$\End_\mathcal{O}(H^1(\Gamma_m,\mathcal{O}))$ generated by these operators
(see \cite[\S2.1]{LV-IJNT}). 
These actions are compatible with respect to the maps induced by the canonical projection maps $X_{m+1}\rightarrow X_m$ for $m\geq 1$, and we may define 
the inverse limit ${\mathfrak{h}}_\infty^*=\invlim{\mathfrak{h}}_m^*$ which equips $\ES_\mathcal{O}$ with a canonical structure 
of $\mathfrak{h}^*_\infty$-module. If $e^{\ord,*}$ is the ordinary idempotent associated with $U_p^*$ as before, then we set $\mathfrak{h}_\infty^{\ord,*}=e^{\ord,*}\mathfrak{h}_\infty^*$ and 
define the $\mathfrak{h}_\infty^{\ord,*}$-module 
\[\ES^\ord_\mathcal{O}=e^{\ord,*}\ES_{\mathcal{O}}^\ord.\] 
Note that the canonical structure of $\Lambda_\mathcal{O}$-module of $\ES^\ord_\mathcal{O}$ is via diamond operators $\langle d\rangle^*$. Moreover, $\ES_\mathcal{O}^\ord\simeq\ES^\ord_{\Z_p}\hat\otimes_{\Z_p}\mathcal{O}$ by the arguments in \cite[Lemma (1.2.11)]{Ohta-ES} and \cite[Lemma (1.2.12)]{Ohta-ES} (these results use the same group $\mathfrak{F}_m$ used to justify \eqref{iso-et-gr}, \emph{cf.} Remark \ref{rem4.1}). 
Under the isomorphism \eqref{eqES1} 
the action of the Hecke and diamond operators $T^*$ on $H^1_\text{\'et}(X_m,\Z_p)(1)$ corresponds to the action of $T$ on $J_m$ (see \cite[\S8]{hida-galois}), and  
the map $T\mapsto T^*$ induces an isomorphism of $\mathcal{O}$-algebras $\mathfrak{h}_m\simeq \mathfrak{h}_m^*$
and $\mathfrak{h}_\infty^\ord\simeq \mathfrak{h}_\infty^{\ord,*}$.

Fix $\I$ as in \S\ref{PBHF}. Then using the isomorphism $\mathfrak{h}_\infty^\ord\simeq \mathfrak{h}_\infty^{\ord,*}$, we have a corresponding maximal ideal $\mathfrak{m}^*$ and a primitive Hida family $\I^*$ for $\mathfrak{h}_\infty^{\ord,*}$; as before, this is the integral closure of $\mathcal{R}^*=\mathfrak{h}_\mathfrak{m}^{\ord,*}/\mathfrak{a}^*$ for a unique minimal 
prime ideal $\mathfrak{a}^*\subseteq\mathfrak{m}^*$, with associated projector $e_{\mathcal{R}^*}$. 
Define \[\ES_\I=\ES_\mathcal{O}^\ord\otimes_{\mathfrak{h}_\infty^{\ord,*}}\I^*.\] 
Note that, using the projector $e_\mathcal{R}$, we have 
$\ES_\I=(e_\mathcal{R}\ES_\mathcal{O}^\ord)\otimes_{\mathcal{R}^*}\I^*$.  
Let $\widetilde{\Lambda}_\mathcal{O}\defeq\mathcal{O}[[\mathcal{Z}_{MD}]]$, 
where recall that $\mathcal{Z}_{MD}$  is defined in \eqref{ZMD}.

\begin{theorem}\label{thm1.2}
$\ES_\mathcal{O}^\ord$ is a $\widetilde{\Lambda}_\mathcal{O}$-module of finite rank. 
Moreover, $\ES_\I$ is a free $\I^*$-module of rank $2$.
\end{theorem}

\begin{proof} 
From \eqref{eqES1} we have an isomorphism of $\widetilde{\Lambda}_\mathcal{O}$-modules 
$\ES_\mathcal{O}^\ord\simeq\mathbf{T}_\mathcal{O}^\ord$ which intertwines the two Hecke actions (and diamond operators) described before.
Now $\mathbf{T}_\mathcal{O}^\ord$ is isomorphic to the $D$-new quotient $\widetilde T_\infty^\ord$ 
of the Hida big Galois representation $T_\infty^\ord$ for $\GL_2$ of tame level $\Gamma_0(MD)$; this is an isomorphism as Hecke modules, where we identify $\tilde{\mathfrak{h}}_\infty^\ord$ 
with the $D$-new quotient of Hida's Hecke algebra $\tilde h_\infty^\ord$ for $\GL_2$ of tame level $\Gamma_0(MD)$. 
Since the Hida big Galois representation $T_\infty^\ord$ is finitely generated over $\Lambda_\mathcal{O}$, the first part follows. For the second part, recall that $\I$ is isomorphic to a primitive component $I$ of Hida's Hecke algebra $\tilde h_\infty^\ord$, and therefore $\mathbf{T}_\infty^\ord\otimes_{\tilde{\mathfrak{h}}_\infty^\ord}\I\simeq T_\infty^\ord\otimes_{\tilde{h}_\infty^\ord}I$; the result then follows from the fact that
$T_\infty^\ord\otimes_{\tilde{h}_\infty^\ord}I$ is free of rank $2$ over $I$. 
The reader is referred to \cite[Proposition 6.4]{LV-MM} for details. 
\end{proof}

Let $\Sym^{k-2}(R)$ be the standard symmetric left representation of $\GL_2(R)$, for a ring $R$, defined for non-negative integers $i$ and $j$ such that $i+j=k-2$ by  \[\gamma (x^iy^j)=(ax+by)^i(cx+dy)^j\] 
for $\smallmat abcd\in\GL_2(R)$, where $\{x^iy^j|i+j=k-2\}$ is a basis of $\Sym^{k-2}(R)$.
We then have a specialization map 
\begin{equation}\label{specmap}
\mathrm{sp}_{k,m}\colon \ES_\mathcal{O}^\ord\longrightarrow H^1(\Gamma_m,\Sym^{k-2}(\mathcal{O}))\end{equation} 
for all $m\geq 1$ obtained by integrating measures (\cite[\S2.5]{LV-IJNT}). 
Let $\kappa\colon \I\rightarrow\overline{\Q}_p$ be an arithmetic point and 
denote $P_\kappa$ the kernel of $\kappa$, $F_\kappa$ the fraction field of $\I/P_\kappa\I$ and $\mathcal{O}_\kappa$ its valuation ring. Each arithmetic point $\kappa\colon \I\rightarrow \mathcal{O}_\kappa$ corresponds to a unique 
$\kappa^*\colon \I^*\rightarrow\mathcal{O}_\kappa$, with kernel $P_{\kappa^*}$, and we 
denote $H^1(\Gamma_m,\Sym^{k-2}(\mathcal{O}_\kappa))[P_{{\kappa}^*}]$ the submodule of the $\mathcal{O}_\kappa$-module $H^1(\Gamma_m,\Sym^{k-2}(\mathcal{O}_\kappa))$ consisting of elements
$x$ such that $P_{{\kappa^*}}x=0$. 

\begin{theorem}[Control Theorem]\label{CT}
    Let $\kappa\in \Spec(\I)$ be an arithmetic point of signature $(k,\varepsilon)$ and level $\Gamma_m$. Then we have an isomorphism of $2$-dimensional $\mathcal{O}_\kappa$-modules 
    \[\ES_\I/P_{\kappa^*}\ES_\I\simeq H^1(\Gamma_m,\Sym^{k-2}(\mathcal{O}_\kappa))[P_{{\kappa^*}}].\]
\end{theorem}
\begin{proof}
    This is a reformulation of \cite[Theorem 2.19]{LV-IJNT}, taking into account that $\Gamma_m$ is torsion-free and using the 
    argument in \cite[Supplement 5.3]{AS} to control the surjectivity of the injective map in \cite[Lemma 2.15]{LV-IJNT}. 
\end{proof}

\subsection{Hodge--Tate exact sequence} \label{HT}
Denote 
\[\mathfrak{A}_m^*=e^{\ord,*}H^1_{\et}(\overline{X}_m,\Z_p)^{I_{\Q_p}},\]
\begin{equation}\label{def:Bm*}
\mathfrak{B}_m^*=e^{\ord,*}H^1_{\et}(\overline{X}_m,\Z_p)/e^{\ord,*}H^1_{\et}(\overline{X}_m,\Z_p)^{I_{\Q_p}}\end{equation}
the Hecke modules appearing in Proposition \ref{isoprop1}. 

Define 
$\mathfrak{A}_\infty^*=\invlim\mathfrak{A}_m^*$, 
$\mathfrak{B}_\infty^*=\invlim\mathfrak{B}_m^*$. 
If we let $\Lambda=\Z_p[[\Z_p^\times]]$, $\widetilde{\Lambda}=\Z_p[[\mathcal{Z}_{MD}]]$ and, as in the proof of Proposition \ref{isoprop1} we denote $\iota\colon \mathcal{Z}_{MD}\hookrightarrow \widetilde{\Lambda}$ the inclusion of group-like elements of $\mathcal{Z}_{MD}$ (defined in \eqref{ZMD}) in $\widetilde{\Lambda}$, then $\mathfrak{A}_\infty^*$ and $\mathfrak{B}_\infty^*$ are equipped with canonical structure of $\widetilde\Lambda$-modules such that 
$\iota(a)$ acts via $\langle a\rangle^*$.  
Taking inverse limits in Proposition \ref{isoprop1} (and using that these are all free $\Z_p$-modules of finite rank) we obtain an exact sequence of $\widetilde\Lambda$-modules (which we call \emph{Tate exact sequence} in analogy with the case of $p$-divisible groups): 
\[0\longrightarrow \mathfrak{A}_\infty^*\longrightarrow  \ES_{\Z_p}^\ord\longrightarrow \mathfrak{B}_\infty^*\longrightarrow 0.\] It follows from Proposition \ref{isoprop1} that 
$\mathfrak{A}_\infty^* = (\ES_{\Z_p}^\ord)^{I_{\Q_p}},$ 
and that the action of $\sigma\in I_{\Q_p}$ on $\mathfrak{B}_\infty^*$ is via the character $\chi_\cyc(\sigma)^{-1}\iota(\chi_\cyc(\sigma))^{-1}.$

Theorem \ref{CT} in the case $\mathcal{O}=\Z_p$
give an isomorphism of $\Lambda$-modules
\[\ES^\ord_{\Z_p}/({\omega_m}\footnote{E: why is it $\omega_m$ in here?})\overset{\sim}\longrightarrow e^{\ord,*}H^1_{\et}(X_m,\Z_p)\]
where $(\omega_m)=\prod_\kappa P_{\kappa^*}$, the product ranges over all arithmetic points $\kappa$ of weight $2$ and conductor $p^m$ and $P_{\kappa^*}$ is the kernel of $\kappa^*$ (compare with \cite[(2.1.2)]{OhtaMA}; alternatively, one may deduce the statement directly from Theorem \ref{CT} 
over $\mathcal{O}$ and using that $\Lambda_\mathcal{O}$ is faithfully flat over $\Lambda_{\Z_p}$ by 
\cite[Lemma 2.1.1]{OhtaMA}). We thus have a diagram with exact horizontal lines:
\begin{equation}\label{dia1}
\xymatrix{&\mathfrak{A}^*_\infty/(\omega_m)\ar[r]\ar[d] & \ES^\ord_{\Z_p}/(\omega_m)\ar[r] \ar[d]& \mathfrak{B}^*_\infty/(\omega_m)\ar[r] \ar[d]& 0\\
0\ar[r] & \mathfrak{A}^*_m \ar[r] &e^{\ord,*}H^1_{\et}(\overline{X}_m,\Z_p)\ar[r] & \mathfrak{B}^*_m\ar[r] & 0.}
\end{equation} 

Let $\mathfrak{O}$ be the valuation ring of a complete field extension $F/\Q_p$ 
which contains all the roots of unity and the maximal unramified extension of $\Q_p$, and denote $\Lambda_\mathfrak{O}\defeq \Lambda\otimes_{\Z_p}\mathfrak{O}$.
We obtain from \eqref{dia1} surjective homomorphisms 
\begin{equation}\label{O2.1.4}
    \mathfrak{B}_\infty^*\otimes_{\Lambda}\Lambda_\mathfrak{O}\longepi \mathfrak{B}_m^*\otimes_{\Lambda}\Lambda_\mathfrak{O}=\mathfrak{B}_m^*\otimes_{\Z_p}\mathfrak{O}.
\end{equation}
Taking inverse limits over $m$ we thus obtain homomorphisms: 
\begin{equation}\label{O2.1.5}
\mathfrak{B}_\infty^*\otimes_{\Lambda}\Lambda_\mathfrak{O}\longrightarrow \invlim_m\left(\mathfrak{B}^*_m\otimes_{\Z_p}\mathfrak{O}\right).\end{equation}

\begin{lemma}\label{O2.1.5lemma}
The homomorphism \eqref{O2.1.5} is  injective. 
\end{lemma}

\begin{proof} The proof is taken from \cite[Lemma 2.1.6]{OhtaMA}.
    Let $\mathbf{B}_m$ be the kernel of $\mathfrak{B}_\infty^*\rightarrow \mathfrak{B}_m^*$. 
    By \cite[Lemma (2.1.1)]{OhtaMA}, $\Lambda_\mathfrak{O}$ is faithfully flat as $\Lambda$-algebra, and therefore it is enough to prove 
    that $\bigcap_m(\mathbf{B}_m\otimes_{\Lambda}\Lambda_\mathfrak{O})=0$ in $\mathfrak{B}_\infty^*\otimes_\Lambda\Lambda_\mathfrak{O}$. For this, 
    since each $\mathfrak{B}_m^*$ is $\Z_p$-free by Proposition \ref{isoprop1}, $\mathfrak{B}_\infty^*$ has no non-zero finite $\Lambda$-submodule, 
    so the structure theorem of $\Lambda$-modules shows that there is an exact sequence of $\Lambda$-modules 
    \[0\longrightarrow \mathfrak{B}_\infty^*\longrightarrow E\longrightarrow G\longrightarrow0 \]
where $G$ is a finite $\Lambda$-module, and $E$ is the direct sum of copies of $\Lambda$ and torsion $\Lambda$-modules of the form $\Lambda/(f^n)$ 
for a distinguished polynomial $f$, or $\Lambda/(p^n)$ for some integer $n\geq 1$. Now $\bigcap_m\mathbf{B}_m=0$, and therefore by \cite[IV, \S2, Cor.4]{Bourbaki-CA}, for any $n\ge 1$ we see that 
$\mathbf{B}_m\subseteq (p,T)^n\mathfrak{B}_\infty^*$ for $m$ sufficiently large. Hence 
\[\bigcap_{m\geq 1}\left(\mathbf{B}_m\otimes_{\Lambda}\Lambda_\mathfrak{O}\right)\subseteq \bigcap_{n\geq 1}(p,T)^n (\mathfrak{B}_\infty^*\otimes_\Lambda\Lambda_\mathfrak{O})
\subseteq \bigcap_{n\geq 1}(p,T)^n\left(E\otimes_\Lambda\Lambda_\mathfrak{O}\right).\]
The result follows.
\end{proof}

Recall that $\mathfrak{B}_m^*$ 
is isomorphic to $\Hom(\Ta_p(\mathcal{G}_m^0),\Z_p)$ (by Proposition \ref{isoprop1} and \eqref{def:Bm*})
and that $\mathcal{G}_m$ is ordinary (Proposition \ref{multiplicativetype}). 
We then have a map 
\begin{equation}\label{Ohta2.1.7}
\mathfrak{B}_m^*\otimes_{\Z_p}\mathfrak{O}\simeq\Cot(\mathcal{G}^{*,0}_{m}\otimes\mathfrak{O})(-1)\simeq
e^{\ord,*}\Cot(\mathcal{A}_m^{*,0}\otimes\mathfrak{O})(-1)\longmono S_2(\Gamma_m,F)(-1)\end{equation}
where $\cdot\otimes\mathfrak{O}$ actually means $\cdot\otimes_{\Z_p[\zeta_{p^m}]}\otimes\mathfrak{O}$. The first isomorphism 
is due to Tate \cite[Section 4]{Tate-pdiv}. 
The second isomorphism can be justified as in \cite[(2.1.8), (2.1.9)]{OhtaMA} as follows. From the standard exact sequence: 
\[0\longrightarrow \mathcal{A}_m^{*,0}[p^n]\longrightarrow \mathcal{A}_m^{*,0}\overset{p^n}\longrightarrow \mathcal{A}_m^{*,0}\longrightarrow 0\] we extract for each $n\geq0$ the isomorphism 
$\Cot(\mathcal{A}_m^{*,0})/(p^n)\simeq\Cot(\mathcal{A}_m^{*,0}[p^n])$
and by definition of the fixed part we have an isomorphism of the cotangent spaces of 
$\mathcal{A}_m^{*,0}[p^n]$ and $ (\mathcal{A}_m^{*,0}[p^n])^\mathrm{f}$ along the unit section. 
We thus obtain isomorphisms $e^{\ord,*}\Cot(\mathcal{A}_m^{*,0})/(p^n)\simeq \Cot(\mathcal{G}_m^{*,0}[p^n])$.
Taking the projective limit in $n$ and tensoring by $\mathfrak{O}$ over $\Z_p[\zeta_{p^m}]$ gives the isomorphism stated before. 
The last map is the composition 
\begin{equation}\label{Ohta3.3.4.1}
\Cot(\mathcal{A}_m^{*,0}\otimes\mathfrak{O})\longmono 
\Cot(\mathcal{J}_m^{*,0}\otimes\mathfrak{O})\longmono 
\Cot(J_m^*\otimes_\Q F)\longmono 
S_2(\Gamma_m,F),\end{equation}
where $\mathcal{J}_m$ is the N\'eron model of $J_m$ over $\Z_p[\zeta_{p^m}]$, $\mathcal{J}_m^0$ its connected component (\cite[\S1b)]{Mazur-Rational}) and $\cdot\otimes\mathfrak{O}$ once again means $\cdot\otimes_{\Z_p[\zeta_{p^m}]}\otimes\mathfrak{O}$.

{
\begin{lemma}\label{lemmaintegral} 
\begin{enumerate}
\item The image of $e^{\ord,*}\Cot(\mathcal{A}_m^{*,0}\otimes\mathfrak{O})$ via the map \eqref{Ohta3.3.4.1} is contained in $\mathfrak{S}_m^*$. 
\item The image of  
\eqref{Ohta2.1.7} is contained in $\mathfrak{S}_m^*(-1)$. 
\end{enumerate}
\end{lemma}

\begin{proof}
(2) follows from (1), so we only need to show the first statement, which follows  
from Proposition \ref{propinteg} using that the $\mathfrak{O}$-module 
$e^{\ord,*}\Cot(\mathcal{A}_m^{*,0}\otimes\mathfrak{O})$ 
is $U_p^*$-stable, and the following fact: 
the $\mathfrak{O}$-module 
$e^{\ord,*}\Cot(\mathcal{A}_m^{*,0}\otimes\mathfrak{O})$ is contained in $S_2^*(\Gamma_m,\mathfrak{O})$. The argument to prove the last statement is similar to the proof of Lemma \ref{lemma_Cot-MF}
(and \cite[Proposition (3.3.6)]{Ohta-ES}). 
With notation as in  Lemma \ref{lemma_Cot-MF}, 
  the $T_x$-expansion of the map \eqref{Ohta3.3.4.1} is the scalar extension to $\mathfrak{O}$ of the pull-back of differential forms via the map 
\[\Spec(\mathcal{R}_x^\mathrm{univ})\simeq\Spec(\widehat{\mathcal{O}}_{\mathcal{X}_{m},x_m})\longrightarrow
\mathcal{X}_m^\mathrm{sm}\longrightarrow\mathcal{J}_m\rightarrow\mathcal{A}_m^*.\] 
Since the action of $\tau_m$ preserves regular differentials, the claim follows. 
\end{proof}

Taking projective limits, and combining Lemma \ref{O2.1.5lemma} with Lemma \ref{lemmaintegral}, 
{we obtain an injective homomorphism of $\Lambda_\mathfrak{O}$-modules} 
\begin{equation}\label{mapmainthm}
(\mathfrak{B}_\infty^*\otimes_{\Lambda}\Lambda_\mathfrak{O})(1)\longmono \mathfrak{S}_\infty^*. 
\end{equation}

\begin{remark}\label{remiso}
    We conjecture that \eqref{mapmainthm} is an isomorphism, as in the $\GL_2$-case. The proof could be probably obtained by means of a finer control theorem for the inverse limit of modular forms. 
\end{remark}

We now discuss the case of primitive Hida families. 
Let $\mathcal{O}$ be the valuation ring of a finite unramified extension of $\Q_p$, and let $\mathbb{I}$ be a primitive component of the Hida Hecke algebra $\mathfrak{h}_\infty^\ord$ acting on the \'etale cohomology of $X_m$ with coefficients in $\mathcal{O}$; 
therefore, $\mathbb{I}$ is an integral finite flat extension of $\Lambda_\mathcal{O}=\mathcal{O}[[1+p\Z_p]]$. 
Since we need to work over the bigger ring $\mathfrak{O}$, we set $\mathbb{J}=\mathbb{I}\otimes_{\mathcal{O}}\mathfrak{O}$. 
We also define 
$\mathbb{J}^*=\mathbb{I}^*\otimes_{\mathcal{O}}\mathfrak{O}$. 
Define 
\[\ES_{\J}=\ES_{\I}\otimes_{\I^*}\J^*= \mathrm{ES}_\mathcal{O}^\ord\otimes\J^*,\] 
and similarly $\mathfrak{B}_{\J}^*=\mathfrak{B}_\infty^*\otimes\J^*$ and {$\mathfrak{A}_\J^*=\mathfrak{A}_\infty^*\otimes\J^*$}, where the unadorned tensor products are taken 
over the Hida Hecke algebra $\mathfrak{h}_\infty^{\ord,*}$.

\begin{theorem}
    We have an exact sequence of $\J^*$-modules 
    \[0\longrightarrow \mathfrak{A}_{\J}^{*}\longrightarrow \ES_{\J}\longrightarrow 
\mathfrak{B}_{\J}^{*}\longrightarrow 0\] which splits as 
$\ES_{\J}\simeq \mathfrak{A}_{\J}^{*}\oplus \mathfrak{B}_{\J}^{*}$
and each summand is a free $\J^*$-module of rank $1$. 
\end{theorem}
\begin{proof}
We consider the same objects 
as before defined by over $\GL_2$ as in \cite{Ohta-ES, OhtaMA}, 
which we denote $\mathfrak{B}_{\J}^{*,\Oh}$ and $\mathfrak{A}_{\J}^{*,\Oh}$; we also consider 
$\ES_{\J}^{\Oh}$. Then we have an isomorphism of $\J^*[G_\Q]$-modules 
$\ES_{\J}^{\Oh}\simeq \ES_{\J}$. Moreover, we have an exact sequence 
\[0\longrightarrow \mathfrak{A}_{\J}^{*,\Oh}\longrightarrow \ES_{\J}^{\Oh}\longrightarrow 
\mathfrak{B}_{\J}^{*,\Oh}\longrightarrow 0\] which splits as 
$\ES_{\J}^{\Oh}\simeq \mathfrak{A}_{\J}^{*,\Oh}\oplus \mathfrak{B}_{\J}^{*,\Oh}$
and each summand is a free $\J^*$-module of rank $1$. Furthermore, we have 
$\mathfrak{A}_{\J}^{*, \Oh} = (\ES_{\J}^{\Oh})^{I_{\Q_p}}$ 
and $\sigma\in I_{\Q_p}$ acts on $\mathfrak{B}_{\J}^{*,\Oh}$ as $\chi_\cyc(\sigma)^{-1}\iota(\chi_\cyc(\sigma))^{-1}.$ It follows from the 
isomorphism $\ES_{\J}^{\Oh}\simeq \ES_{\J}$ that 
that 
$\mathfrak{A}_\J^*\simeq \mathfrak{A}_{\J}^{*,\Oh}$ is also $\J^*$-free of rank $1$, 
and similarly $\mathfrak{B}_\J^*\simeq \mathfrak{B}_{\J}^{*,\Oh}$ is also $\J^*$-free of rank $1$ 
and we have a decomposition 
$\ES_{\J}\simeq \mathfrak{A}_{\J}^{*}\oplus \mathfrak{B}_\J^*$. The result follows.\end{proof} 

Define 
\begin{equation}\label{defT}
\mathbf{T}=\mathbf{T}^\ord_\mathcal{O}\otimes_{\mathfrak{h}_\infty^\ord}\I,\end{equation} which is a free $\I$-module 
of rank $2$ equipped with a continuous $G_{\Q}$-action. 
As $G_{\Q}$-representations, 
$\mathbf{T}$ is the dual representation of $\ES_\I$; also,
$\mathbf{T}=(e_\mathcal{R}\mathbf{T}^\ord_\mathcal{O})\otimes_\mathcal{R}\I$.
As $G_{\Q_p}=\Gal(\overline{\Q}_p/\Q_p)$-representations, 
we have an exact sequence \begin{equation}\label{filtration}
0\longrightarrow \mathbf{T}^+\longrightarrow  \mathbf{T}\longrightarrow\mathbf{T}^-\longrightarrow 0\end{equation}
such that both $\mathbf{T}^+$ and $\mathbf{T}^-$ are free $\I$-modules of rank $1$, and $\mathbf{T}^-$ is unramified. More precisely, assume that $\I$ contains an arithmetic point $\kappa$ corresponding to 
a $p$-stabilized eigenform in $S_k(\Gamma_0(MDp))$; 
let $\Theta\colon G_\Q\rightarrow\Lambda^\times$ be
the choice of a critical character defined in \cite[Definition 2.1.3]{Howard-Inv} by
$\Theta(\sigma)=\omega^\frac{k-2}{2}[\langle\sigma\rangle^{1/2}]$ where 
$x\mapsto x^{1/2}$ is the unique square root of $x\in 1+p\Z_p$ and $\omega$ is the Teichmüller character. 
Then $G_{\Q_p}$ acts on $ \mathbf{T}^+$ via 
{$\eta^{-1}\chi_\cyc\Theta^2$} and acts on the unramified quotient 
$\mathbf{T}^-$ via $\eta$, where $\eta$ is an unramified character 
(see \cite[\S5.5]{LV-MM} and \cite[Corollary 6.5]{LV-MM} for details). 
We then define $\mathbf{T}_\J^\bullet=\mathbf{T}^\bullet\otimes_\I\J$ for $\bullet$ being +, $-$ or 
no symbol. Then we have 
$\mathfrak{B}_\J^*(1)\simeq \mathbf{T}_\J^+$. 

\subsection{Serre--Tate expansions}\label{subsec.STexp}
Since $\Lambda\rightarrow\J$ is a flat extension of local rings, by \cite[Theorem 7.2]{Matsumura} 
it is also faithfully flat, and we obtain a chain of injective maps 
\begin{equation}\label{STR}
\ST_\mathcal{R}\colon e_{\mathcal{R}^*}\mathfrak{B}_\infty^* \longmono \mathfrak{B}_\infty^*\longmono
\mathfrak{B}_\infty^*\otimes_{\Lambda_\mathfrak{O}}\J^*\overset{\eqref{mapmainthm}}\longmono \mathfrak{S}_\infty^*\otimes_{\Lambda_\mathcal{O}}\J^*\overset{\eqref{STexp}}\longmono \Lambda_\mathfrak{O}[[T_x]]\otimes_{\Lambda_\mathfrak{O}}\J\simeq\J[[T_x]].\end{equation}
Extending by $\J$-linearity via $x\otimes \lambda\mapsto \lambda\ST(x)$, we get a map of $\J$-modules 
(denoted with the same symbol) 
\begin{equation}\label{STJ}
\ST\colon \mathfrak{B}_\J^*\longrightarrow \J[[T_x]].\end{equation}
\begin{lemma} The map $\ST$ in $\eqref{STJ}$ is injective.\end{lemma}
\begin{proof} Since the 
$\J$-module $\mathfrak{B}_\J^*$ is  free of rank $1$, and the $\J$-module $\J[[T_x]]$ is $\J$-torsion free,
the image of $\ST$ in \eqref{STJ} is either a cyclic module of rank $1$ or trivial; since 
$\ST_\mathcal{R}$ in \eqref{STR} is injective, the second possibility does not occur, so 
$\ST$ must be injective as well. 
\end{proof}

{Define \[\mathbf{D}(\mathbf{T}^+)=(\mathbf{T}^+\widehat{\otimes}_{\Z_p}\Z_p^\unr)^{G_{\Q_p}}.\] 
By \cite[Lemma 3.3]{Ochiai-Coleman}, since $\mathbf{T}^+$ is unramified, $\mathbf{D}(\mathbf{T}^+)$ is a free $\I$-module of rank $1$. 
Since $\mathbf{T}_\J^+\simeq (\mathbf{T}^+\hat{\otimes}_{\Z_p}\Z_p^\unr)\hat\otimes_{\Z_p^\unr}\mathfrak{O}$ and $\mathbf{T}^+$ is unramified, we have that $\mathbf{D}(\mathbf{T}^+)\hat\otimes_{\Z_p^\unr}\mathfrak{O}$ is also equal to $(\mathbf{T}_\J^+)^{G_{\Q_p}}$. 
Fixing 
an $\I$-basis $\eta_\I$ of $\mathbf{D}(\mathbf{T}^+)$ then gives an $\I$-basis of $(\mathbf{T}_\J^+)^{G_{\Q_p}}$, 
which we denote with the same symbol $\eta_\I$. The canonical inclusion $(\mathbf{T}_\J^+)^{G_{\Q_p}}\subseteq \mathbf{T}_\J^+$ combined with the isomorphism $\mathbf{T}_\J^+\simeq\mathfrak{B}_\J^*(1)$ then gives 
by restriction an injective map (denoted with the same symbol)
\begin{equation}\label{STD}
\ST\colon \mathbf{D}(\mathbf{T}^+)\longrightarrow (\mathfrak{S}_\infty^*\otimes_{\Lambda_\mathcal{O}}\J^*)^{G_{\Q_p}}\longrightarrow \J[[T_x]]\end{equation}

\begin{lemma}\label{lemmaI}
The image $\ST(\eta_\I)$ belongs to $\widetilde{\I}[[T_x]]$, where $\widetilde{\I}\defeq\I\otimes_{\Z_p}\Z_p^\unr.$
\end{lemma}

\begin{proof}
  Since we are taking $G_{\Q_p}$-invariants, the target $(\mathfrak{S}_\infty^*\otimes_{\Lambda_\mathcal{O}}\J^*)^{G_{\Q_p}}$ of the first map 
  in \eqref{STD} is the tensor product (over $\Lambda$) with $\I$ of the 
  $\Z_p$-module $\mathfrak{S}_\infty^*$ as in Definition \ref{DEFINT} with $\mathcal{O}=\Z_p$, 
  \emph{i.e.} the $\Z_p$-submodule of
    $\invlim S_2(\Gamma_m,\Q_p)$ consisting of those $(f_m)_{m\geq 1}$ with $f_m$ 
    satisfying the condition $f_m|\tau_m| U_p^m\in S_2^\ST(\Gamma_m,\Z_p)$. Any such $f_m$ has Serre--Tate expansion in $\Z_p^\unr[[T_x]]$, and the result follows.  
\end{proof}

Thanks to Lemma \ref{lemmaI} we obtain an injective map of $\I$-modules   
\[\ST\colon\mathbf{D}(\mathbf{T}^+)\longmono \widetilde{\I}[[T_x]]\] 
and we may define the power series 
\begin{equation}\label{FIT}\mathcal{F}_\I(T_x)\defeq\ST(\eta_\I)\end{equation} in $\widetilde{\I}[[T_x]]$.}
We write this power series as $\mathcal{F}_\I(T_x)=\sum_{n\geq 0}a_nT_x^n$, and observe that it is well defined up to a unit in $\I^\times$ depending on our choice of $\eta_\I$. 
By specialization we obtain a power series $\mathcal{F}_\kappa(T_x)= \sum_{n\geq 0}a_n(\kappa)T_x^n$ in $\mathcal{O}_\kappa^\unr[[T_x]]$, where  $a_n(\kappa)=\kappa(a_n)$ and $\mathcal{O}_\kappa^\unr$ is the maximal unramified extension of $\mathcal{O}_\kappa$. 
Then $\mathcal{F}_\kappa(T_x)$ is the $T_x$-expansion of a unique modular form $f_\kappa\in S_k(\Gamma_m,F_\kappa)$ 
(so in fact $f_\kappa\in S^\ST_k(\Gamma_m,\mathcal{O}^\unr_\kappa)$). 
This fixes a choice of Jacquet--Langlands lift of the elliptic modular form corresponding to $\kappa$; as remarked above, this choice depends on the choice of $\eta_\I$ only.

\section{The $p$-adic $L$-function} \label{sec.padicLfunction}
We apply the machinery developed so far to interpolate in $p$-adic analytic families the $p$-adic $L$-functions for quaternionic modular forms twisted by Hecke characters constructed in \cite{Magrone, Mori1, Brooks}.
Fix $\I$ as in \S\ref{PBHF}. 

\subsection{CM points on Shimura curves} \label{CM section}
Following \cite{ChHs2}, \cite{ChHs1}, \cite{CH}, \cite{BCK},
\cite{CL},
we introduce a collection of Heegner points of $p$-power conductor. 

We first make the choice of the isomorphisms $i_\ell\colon B_\ell\simeq\M_2(\Q_\ell)$ more 
precise, following mainly \cite{ChHs1}.  
Let $K=\Q(\sqrt{-D_K})$ be an imaginary quadratic field, of discriminant $-D_K$ prime to $MDp$, such that all primes dividing $Mp$ are split in $K$ and all primes dividing $D$  are inert in $K$. 
Then $K$ splits $B$, and we may fix an embedding of $\Q$-algebras $\iota_K\colon K\hookrightarrow B$. 
We often identify $K$ with the subfield $\iota_K(K)\subseteq B$
via $\iota_K$, 
and therefore simply write $K\subseteq B$ for $\iota_K(K)\subseteq B$. 
Define $\delta=\sqrt{-D_K}$ and $\theta=\frac{D'+\delta}{2}$, where $D'=D_K$ if $2\nmid D_K$ and $D'=D_K/2$ if $2\mid D_K$. 
As in \cite[\S2.1]{ChHs1}, choose a $K$-basis of $B=K\oplus KJ$
so that $Jx=\bar{x}J$ for all $x\in \iota_K(K)$ 
and $J^2=\beta\in\Q^\times$ such that $\beta\in(\Z_\ell)^2$ for all $\ell\mid Mp$, 
and $\beta\in\Z_\ell$ for all $\ell\mid D_K$. Fix a square root $\sqrt{\beta}$ of $\beta$ 
in $\overline{\Q}$. Then $\sqrt{\beta}\in \Z_\ell$ for all primes $\ell\mid Mp$. 
For each $\ell\mid Mp$, 
we fix the isomorphism $i_\ell:B_\ell\simeq\M_2(\Q_\ell)$ by 
requiring that  
\begin{equation}\label{embeq}
i_\ell(\theta)=\mat {\mathrm{T}_{K/\Q}(\theta)}{-\mathrm{N}_{K/\Q}(\theta)}{1}{0}, 
\qquad i_\ell(J)=\sqrt{\beta}\mat{-1}{\mathrm{T}_{K/\Q}(\theta)}{0}{1}.\end{equation}
(Here for any field extension $L/F$, we denote by $\mathrm{T}_{L/F}$ and $\mathrm{N}_{L/F}$ the trace and norm maps.)
\begin{comment}
The existence of such an isomorphism $i_v$ can 
be justified as follows. Denote by $i\colon K\hookrightarrow \M_2(\Q)$ the $\Q$-linear map which takes $\theta$ to $\smallmat{\mathrm{T}_{K/\Q}(\theta)}{-\mathrm{N}_{K/\Q}(\theta)}{1}{0}$. 
Taking the $v$-adic completion gives a map, still denoted $i:K_v\rightarrow M_2(\Q_v)$, 
where $K_v=K\otimes_\Q\Q_v$, 
which realizes $K_v$ as a quadratic commutative subalgebra of $M_2(\Q_v)$; choose any isomorphism 
$\varphi:B_v\simeq \M_2(\Q_v)$; now $\varphi(\iota_K(K)\otimes_\Q\Q_v)$ is also 
a quadratic commutative subalgebra of $\M_2(\Q_p)$, and hence 
by the Skolem-Noether theorem there is an inner automorphism $x\mapsto bxb^{-1}$ of $B$ 
for some $b\in B^\times$ that takes $\varphi(\iota_K(K)\otimes_\Q\Q_v)$ to 
the image of $i(K_v)$; then $i_v=b\varphi b^{-1}$ works. 
\end{comment}
For each prime $\ell\nmid MDp$, choose isomorphisms $i_\ell\colon B_\ell\cong \M_2(\Q_\ell)$ such that $i_\ell(\mathcal{O}_{K,\ell})\subseteq \M_2(\Z_\ell)$, where $\mathcal{O}_{K,\ell}=\mathcal{O}_K\otimes_\Z\Z_\ell$. 
For each integer $m\geq 0$, let $R_{m}$ be the Eichler order of $B$ of level $Mp^m$ with respect to the chosen isomorphisms $i_\ell$ for all finite places $\ell\nmid D$. 

Define $N=MD$, fix an integer $c\geq 1$ with $(c,ND_Kp)=1$ 
and for each integer $n\geq 0$; let 
$\mathcal{O}_{cp^n}=\Z+cp^n\mathcal{O}_K$ be the order of $K$ of conductor $cp^n$. Class field theory gives an isomorphism 
$\Pic(\mathcal{O}_{cp^n})\cong \Gal(H_{cp^n}/K)$ for an abelian extension $H_{cp^n}$ of $K$, called the \emph{ring class field of $K$ of conductor $cp^n$}. 
Define the union of these fields 
$H_{cp^\infty}=\bigcup_{n\geq 1} H_{cp^n}.$
Since $c$ is prime to $p$, $H_{c}\cap H_{p^\infty}=H$, where $H=H_1$ is the Hilbert class field of $K$, 
so we have an isomorphism of groups 
\[\Gamma_\infty\defeq \Gal(H_{cp^\infty}/K)\cong  \Gal(H_c/K)\times\Gal(H_{p^\infty}/H).\]
Since $p$ is split in $K$, we have  $\Gal(H_{p^\infty}/H)\cong \Z_p^\times$; as usual we decompose $\Z_p^\times\cong \Delta\times\Gamma$, with $\Gamma=(1+p\Z_p)$ and $\Delta=(\Z/p\Z)^\times$.

We refer to \cite[Definition 3.1]{LV-MM} for the definition of Heegner points on $X_m$
in terms of optimal embeddings; to fix the notation, for $c\geq 1$ an integer prime to $pND_K$ 
if the point  
$x=[(f,g)]$ represents a Heegner point on $X_m$ 
then $f(\mathcal{O}_{cp^n})=f(K)\cap gU_mg^{-1}$; here $f\colon K\hookrightarrow B$ is viewed as a point in $\mathcal{H}^\pm$ by scalar extension to $\R$. Moreover, for $a\in\widehat{K}^\times$, 
by Shimura reciprocity law we have 
$x^\sigma=[(f,\hat{f}(a^{-1})g)]$
where $\hat{f}\colon\widehat{K}\rightarrow\widehat{B}$ is the adelization of $f$,
$\mathrm{rec}_K(a)=\sigma$, and $\mathrm{rec}_K$ is the geometrically normalized reciprocity map 
(\cite[Theorem 9.6]{shimura}). 

Let $c=c^+c^-$ with $c^+$ divisible by primes which are split in $K$ and $c^-$ divisible by primes which are inert in $K$. Choose decompositions $c^+=\mathfrak{c}^+\bar{\mathfrak{c}}^+$ and 
$M=\mathfrak{M}\overline{\mathfrak{M}}$ into cyclic quotient ideals of orders $c^+$ and $M$ respectively. 
For each prime $\ell$ and each integer $n\geq 0$, define 
\begin{itemize}
      \item $\xi_\ell=1$ if $\ell\nmid Mcp$.
           \item $\xi_{p}^{(n)}=\delta^{-1}\smallmat{\theta}{\bar\theta}{1}{1}\smallmat{p^n}{1}{0}{1}
           \in \GL_2(K_\mathfrak{p})=\GL_2(\Q_p)$. 
        \item $\xi_{\ell}=\delta^{-1}\smallmat{\theta}{\bar\theta}{1}{1}\smallmat{\ell^s}{1}{0}{1}\in \GL_2(K_\mathfrak{l})=\GL_2(\Q_\ell)$ if $\ell\mid c^+$ and $\ell^s$ is the exact power of $\ell$ dividing $c^+$, 
        where $(\ell)=\mathfrak{l}\bar{\mathfrak{l}}$ is a factorization into prime ideals in $\mathcal{O}_K$ and $\mathfrak{l}\mid \mathfrak{c}^+$. 
       \item $\xi_{\ell}=\smallmat{0}{-1}{1}{0}\smallmat{\ell^s}{0}{0}{1}\in \GL_2(\Q_\ell)$ if $\ell\mid c^-$ and $\ell^s$ is the exact power of $\ell$ dividing $c^-$.
         \item $\xi_\ell=\delta^{-1}\smallmat{\theta}{\bar\theta}{1}{1}\in \GL_2(K_\mathfrak{l})=\GL_2(\Q_\ell)$ if $\ell\mid M$, where $(\ell)=\mathfrak{l}\bar{\mathfrak{l}}$ is a factorization into prime ideals in $\mathcal{O}_K$ and $\mathfrak{l}\mid \mathfrak{M}$.
  \end{itemize}
We understand these elements $\xi^\star_\bullet$ as elements in $\widehat{B}^\times$ by implicitly using the isomorphisms $i_\ell$ defined before. 
With this convention, define $\xi^{(n)}=(\xi_\ell,\xi_{p}^{(n)})_{\ell\neq p}\in \widehat{B}^\times$.  
Define a map $x_{cp^m}\colon\Pic(\mathcal{O}_{cp^m})\rightarrow {X}_m(\C)$ by $[\mathfrak{a}]\mapsto[(\iota_K,a\xi^{(m)})]$, where
if $\mathfrak{a}$ is a representative of the ideal class $[\mathfrak{a}]$, then $a\in \widehat{K}^\times$ satisfies $\mathfrak{a}=a\widehat{\mathcal{O}}_{cp^m}\cap K$; here $a \in \widehat{K}^\times$ acts on $\xi^{(n)}\in \widehat{B}^\times$ via left multiplication by $\hat{\iota}_K(a)$. We often write $x_{cp^m}(a)$ or $x_{cp^m}(\mathfrak{a})$ 
for $x_{cp^m}([\mathfrak{a}])$.
One verifies (using the description of $i_\ell(\theta)$ in \eqref{embeq}) that $x_{cp^m}(a)$ are Heegner points of conductor $cp^m$ in $X_m(H_{cp^m}(\zeta_{p^m}))$, for all 
$a\in\Pic(\mathcal{O}_{cp^m})$, and all integers $m\geq0$. 
Each $x_{cp^m}(a)$ admits a smooth model over the ring of 
algebraic integers of $H_{cp^m}(\zeta_{p^m})$, and thus defines 
a point in $\widehat{\mathrm{Ig}}$, which we denote
$x(\mathfrak{a})$. 

\subsection{$p$-adic $L$-functions of modular forms}\label{sec-padic-Lfunctions} 
Write $p=\mathfrak{p}\bar{\mathfrak{p}}$ in $\mathcal{O}_K$ 
and let $\mathfrak{p}$ be the prime ideal corresponding to the fixed embedding $\iota_p\colon\bar\Q\hookrightarrow\bar\Q_p$. For an element $x\in\A_K^\times$, let $x_\mathfrak{p}\in K_p^\times$ and $x_{\bar{\mathfrak{p}}}\in K_p^\times$ denote the components of $x$ at $\mathfrak{p}$ and $\bar{\mathfrak{p}}$, respectively.
For an algebraic Hecke character 
$\xi\colon K^\times\backslash\A_K^\times\rightarrow\C^\times$, write $\xi=\xi_\mathrm{fin}\xi_\infty$ with $\xi_\mathrm{fin}\colon\widehat{K}^\times\rightarrow\C^\times$ and $\xi_\infty\colon K_\infty^\times\rightarrow\C^\times$ the finite and infinite restrictions of $\xi$ respectively, 
where $\widehat{K}^\times$ and $K_\infty^\times$ are respectively the groups of finite and infinite ideles. We say that $\xi\colon K^\times\backslash\A_K^\times\rightarrow\C^\times$ has infinity type 
$(m,n)$ if $\xi_\infty(x)=x_\infty^m\bar x_\infty^n$, and in this case denote by $\hat\xi\colon K^\times\backslash\widehat{K}^\times\rightarrow \overline\Q_p^\times$ its $p$-adic avatar, defined by 
\[\hat\xi(x)=(\iota_p\circ\iota_\infty^{-1})(\xi_\mathrm{fin}(x))x_{\mathfrak{p}}^mx_{\bar{\mathfrak{p}}}^n\] where recall that $\iota_\infty\colon\overline{\Q}\hookrightarrow \C$ is 
the fixed embedding. 
To simplify the notation, we sometimes write 
$\hat{\xi}(x)=\xi(x)x_\mathfrak{p}^mx_{\bar{\mathfrak{p}}}^n$
for $x\in\widehat{K}^\times$, 
understanding that $\xi(x)$ 
denotes $(\iota_p\circ\iota_\infty^{-1})(\xi_\mathrm{fin}(x))$.
If $\xi$ is a Hecke character of conductor $\mathfrak{c}\subseteq \mathcal{O}_K$ and $\mathfrak{b}$ is an ideal prime to $\mathfrak{c}$, we write $\xi(\mathfrak{b})$ for $\xi(b)$, where $b \in \widehat{K}^\times$ is a finite idele with trivial components at the primes dividing $\mathfrak{c}$ and such that $b  \widehat{\mathcal{O}}_K\cap K = \mathfrak{b}$. 

Let $c\mathcal{O}_K$ be the prime to $p$ part of the conductor $\mathfrak{c}$ of $\xi$ and recall the CM points $x(\mathfrak{a})$, defined in \S\ref{CM section}, for each $\mathfrak{a} \in \Pic(\mathcal{O}_c)$.
Since $x(\mathfrak{a})$ has a model over $\mathcal{V}=\Z_p^\unr\cap K^{\mathrm{ab}}$, where $K^\mathrm{ab}$ is the maximal abelian extension of $K$, we can consider the reduction $\bar{x}(\mathfrak{a})=x(\mathfrak{a}) \otimes_\mathcal{V} \overline{\F}_p$ of $x(\mathfrak{a})$ modulo $p$.

Fix an anticyclotomic Hecke character $\xi$ of infinity type $(k/2,-k/2)$ 
and $f\in S_k(\Gamma_m,F)$, where $F$ is a finite unramified extension of $\Q_p$.
{Suppose that for any $\mathfrak{a}$ as above, $f\in S_k^\ST(\Gamma_m,\mathcal{O})$, where 
the Serre--Tate expansion is computed with respect to $x(\mathfrak{a})$ and $\mathcal{O}$ is the 
valuation ring of $F$.} 
Then $f^{[p]}(T_{x(\mathfrak{a})})\in \mathcal{O}[[T_{x(\mathfrak{a})}]]$ for all $\mathfrak{a}$ by \cite[Proposition 4.7]{Brooks}.
Let $\mathcal{O}(\xi)$ denote the extension of $\mathcal{O}$ generated by the values of $\hat\xi$.
We recall the construction of the $p$-adic $L$-function associated with $f$ and $\xi$ in \cite[\S 4.6]{Magrone} (see also \cite[\S4]{BCK}). 
For each ideal class $[\mathfrak{a}] \in \Pic(\mathcal{O}_c)$ with $\mathfrak{a}\subseteq\mathcal{O}_c$, 
let $\mathrm{N}(\mathfrak{a})=c^{-1}\sharp(\mathcal{O}_c/\mathfrak{a})$ and 
define the measure $\mu_{f,\mathfrak{a}}$ on $\Z_p$ by the integration formula 
	\begin{equation}\label{def-p-adic-L-function}
\int_{\Z_p} (T_{x(\mathfrak{a})}+1)^x d\mu_{f,\mathfrak{a}}(x)=f^{[p]}_{\mathfrak{a}}(T_{x(\mathfrak{a})})\defeq {f}^{[p]}\left((T_{x(\mathfrak{a})}+1)^{\mathrm{N}(\mathfrak{a})^{-1}\sqrt{-D_K}^{-1}}\right)
	\end{equation}
The measures $\mu_{f,\mathfrak{a}}$ are supported on $\Z_p^{\times}$ (\cite[Remark 4.2]{Magrone}). 

Recall that $\mathrm{rec}_K\colon K^\times\backslash\widehat{K}^\times\rightarrow\Gal(K^\mathrm{ab}/K)$ denotes the geometrically normalized reciprocity map; let 
$\mathrm{rec}_{K,\mathfrak{p}}\colon K_\mathfrak{p}^\times\rightarrow\Gal(K_\mathfrak{p}^\mathrm{ab}/K_\mathfrak{p})$ be the local reciprocity map, and view 
$\Gal(K_\mathfrak{p}^\mathrm{ab}/K_\mathfrak{p})$ 
as a subgroup of $\Gal(K^\mathrm{ab}/K)$ by the fixed
embedding $\overline\Q\hookrightarrow\overline\Q_p$. 
Define ${\Gamma}_\infty=\Gal(H_{cp^\infty}/K)$ and let $\mathcal{O}_{\C_p}$ be the valuation ring of the completion $\C_p$ of $\overline\Q_p$.  
For each 
continuous function $\varphi\colon\Gamma_\infty\rightarrow \mathcal{O}_{\C_p}$, and each ideal class $[\mathfrak{a}]\in \Pic(\mathcal{O}_c)$, define the continuous 
function 
$\varphi \big| [\mathfrak{a}]\colon\Z_p^\times\rightarrow\mathcal{O}_{\C_p}$
by the formula 
\[\varphi|[\mathfrak{a}](u) = \varphi\left(\mathrm{rec}_K(a)\mathrm{rec}_{K,\mathfrak{p}}(u)\right),\]
 where $a \in \widehat{K}^\times$ is a finite idele with trivial components at the primes lying over $cp$ such that $a\widehat{\mathcal{O}}_c \cap K=\mathfrak{a}$ and we view an element $u \in \Z_p^\times$ 
 as an element in $K_\mathfrak{p}^\times$ via the canonical inclusion 
 $\Z_p^\times\subseteq K_\mathfrak{p}^\times$. 
 
Define an $\mathcal{O}(\xi)$-valued measure $\mathscr{L}_{f,\xi}$ on $\Gamma_\infty$ by 
the formula
	\begin{equation*} \label{L-f}
	\mathscr{L}_{f,\xi} (\varphi) = \sum_{\mathfrak{a} \in \Pic(\mathcal O_{c})} \xi(\mathfrak{a})\mathrm{N}(\mathfrak{a})^{-k/2}
	\int_{\Z_p^{\times}} \xi_{\mathfrak{p}}(u) (\varphi \big|[\mathfrak{a}])(u)d\mu_{f,\mathfrak{a}}(u)
	\end{equation*}
	for any continuous function $\varphi\colon\Gamma_\infty\rightarrow \mathcal{O}_{\C_p}$, where, as before,  
	$\xi(\mathfrak{a})$ denotes $\xi(a)$ for $a \in \widehat{K}^\times$ finite idele with trivial components at the primes dividing the conductor of $\xi$ and such that $a \widehat{\mathcal{O}}_c\cap K = \mathfrak{a}$. 

We now consider the case of $p$-stabilizations.

\begin{lemma}\label{lemma equality}\label{lemma5.1}Suppose that $f\in S_k^\ST(\Gamma_1,\mathcal{O})$ is the ordinary $p$-stabilization of $f^\sharp\in S_k(\Gamma_0,K)$. Then 
$\mathscr{L}_{f,\xi} = \mathscr{L}_{f^\sharp,\xi} $. 
\end{lemma}

\begin{proof} The $p$-depletions of $f$ and $f^\sharp$ are the same by the results recalled in \S\ref{stabilization}; therefore, taking the 
Serre--Tate power series expansion we have 
${f}^{[p]}_\mathfrak{a}(T_x)={f^\sharp}^{[p]}_\mathfrak{a}(T_x)$ for each ideal class 
$[\mathfrak{a}]$ in $\Pic(\mathcal{O}_c)$. 
In light of \eqref{def-p-adic-L-function}, this concludes the proof. 

\end{proof}}

\subsection{Families of Hecke characters}\label{characters}\label{sec-famHecke}
Let $G_\Q=\Gal(\overline\Q/\Q)$ and $\chi_\cyc\colon G_\Q\rightarrow\Z_p^\times$ be the cyclotomic character. 
Factor $\chi_\cyc$ as $\chi_\cyc=\chi_\mathrm{tame}\cdot\chi_\mathrm{wild}$, where 
$\chi_\mathrm{tame}\colon G_\Q\rightarrow\mu_{p-1}$ takes values in the group of $p-1$ roots of unity in $\Z_p^\times$ via the Teichmüller character and $\chi_\mathrm{wild}\colon G_\Q\rightarrow\Gamma$ takes values in the group of principal units $\Gamma=1+p\Z_p$; in other words, if we write an element $x\in\Z_p^\times$ as $x=\zeta_x\cdot\langle x\rangle$, where $\zeta_x=\omega(x)\in\mu_{p-1}$ and $x\mapsto\langle x\rangle$ is the projection $\Z_p^\times\rightarrow\Gamma$ from $\Z_p^\times$ to the 
group $\Gamma$ of principal units, then $\chi_\mathrm{tame}(\sigma)=\zeta_{\chi_\cyc(\sigma)}$ and $\chi_{\mathrm{wild}}(\sigma)=\langle\chi_\cyc(\sigma)\rangle$. 
We also denote by
$\boldsymbol{\chi}_{\cyc}\colon\Q^\times\backslash\A_\Q^\times\rightarrow\Z_p^\times$,  $\boldsymbol{\chi}_\mathrm{tame}\colon\Q^\times\backslash\A_\Q^\times\rightarrow\mu_{p-1}$ and $\boldsymbol{\chi}_\mathrm{wild}\colon\Q^\times\backslash\A_\Q^\times\rightarrow\Gamma$ the composition of $\chi_\cyc$, 
$\chi_\mathrm{tame}$ and $\chi_\mathrm{wild}$, respectively, with the reciprocity map $\mathrm{rec}_\Q$; we then have 
$\boldsymbol{\chi}_\mathrm{tame}(x)=\zeta_{\boldsymbol{\chi}_\cyc(x)}$ and $\boldsymbol{\chi}_{\mathrm{wild}}(x)=\langle\boldsymbol{\chi}_\cyc(x)\rangle$.

Write $z\mapsto [z]$ for both inclusions of group-like elements 
$\Z_p^\times\hookrightarrow\Z_p[[\Z_p^\times]]^\times$ and 
$\Gamma\hookrightarrow\Lambda^\times$.
Recall the critical character $\Theta\colon G_\Q\rightarrow\Lambda^\times$ defined in \cite[Definition 2.1.3]{Howard-Inv} by 
\[\Theta(\sigma)=\chi_\mathrm{tame}^{\frac{k_0-2}{2}}(\sigma)\cdot[\chi_{\wild}^{1/2}(\sigma)],\] where 
$x\mapsto x^{1/2}$ is the unique square root of $x\in\Gamma$ and $k_0$ is a fixed even integer such that 
$k_0\equiv 2\pmod{2(p-1)}$. 
We still write $\Theta\colon G_\Q\rightarrow\I^\times$ for the composition of $\Theta$ with the canonical inclusion $\Lambda\hookrightarrow\I$.
Write $\boldsymbol{\theta}\colon\Q^\times\backslash\A_\Q^\times\rightarrow\I^\times$ for the composition of $\Theta$ with the geometrically normalized reciprocity map $\mathrm{rec}_\Q.$
We denote $\Q_p^\cyc=\Q(\zeta_{p^\infty})=\bigcup_{n\geq 1}\Q(\zeta_{p^n})$ 
the $p$-cyclotomic extension of $\Q$, and define $G_\infty^\cyc=\Gal(\Q(\zeta_{p^\infty})/\Q)$. The cyclotomic character induces an isomorphism 
$\chi_\cyc\colon G_\infty^\cyc\overset\sim\rightarrow\Z_p^\times$. 
Since $\Theta$ factors through $G_\infty^\cyc$, 
precomposing it with the inverse of the 
cyclotomic character, we obtain a character of $\Z_p^\times$ which we denote with $\boldsymbol{\vartheta}\colon\Z_p^\times\rightarrow\I^\times$. 
If $\kappa \colon\I\rightarrow\overline\Q_p$ is an arithmetic morphism  of signature $(k,\varepsilon )$ we put ${\theta}_\kappa =\kappa \circ\boldsymbol{\theta}$ 
and ${\vartheta}_\kappa =\kappa \circ\boldsymbol{\vartheta}$. For any $x\in\Z_p^\times$,
if $k \equiv k_0\mod{2(p-1)}$, then we have
\[{\vartheta}_\kappa (x)
=\varepsilon ^{1/2}(\langle x\rangle)\cdot  x^{\frac{k -2}{2}}.
\] 
 
Denote by $\mathbf{N}_{K/\Q}\colon\A_K^\times\rightarrow\A_\Q^\times$ the adelic norm map, by $\mathbf{N}_\Q\colon\A^\times_\Q\rightarrow\Q^\times$ the adelic absolute value and let 
$\mathbf{N}_K\colon\A^\times_K\rightarrow\Q^\times$ denote the composition $\mathbf{N}_K=\mathbf{N}_\Q\circ\mathbf{N}_{K/\Q}$.
Define 
the 
character ${\boldsymbol{\chi}}\colon K^\times\backslash\widehat{K}^\times\rightarrow\I^\times$  
by ${\boldsymbol{\chi}}=\boldsymbol{\theta}\circ \mathbf{N}_{K/\Q}.$
For an arithmetic morphism  $\kappa \colon \mathbb{I}\rightarrow\mathcal{O}_\kappa$ of signature $(k,\varepsilon)$ define 
$\hat{\chi}_\kappa =\kappa \circ\boldsymbol{\chi}$.
Since $\chi_\cyc \circ \mathrm{rec}_\Q$ is the $p$-adic avatar of the adelic absolute value $\boldsymbol{\mathrm{N}}_\Q\colon\A_\Q^\times\rightarrow\Q^\times$, we obtain, for $x\in\widehat{K}^\times$ and $k \equiv k_0\mod{2(p-1)}$, 
\begin{equation}\label{chi_k}
\hat\chi_\kappa (x)= 
\varepsilon ^{-1/2}(\langle \mathbf{N}_K(x)x_\mathfrak{p}  x_{\bar{\mathfrak{p}}}\rangle)\cdot
(\mathbf{N}_K(x)x_\mathfrak{p} x_{\bar{\mathfrak{p}}})^{-\frac{k -2}{2}}.
\end{equation}

Let $\lambda\colon K^\times\backslash\A_K^\times\rightarrow\C^\times$ 
be a Hecke character of infinity type $(1,0)$, unramified at $p$ and whose $p$-adic avatar $\hat\lambda\colon K^\times\backslash\widehat{K}^\times\rightarrow{\overline\Q_p^\times}$ takes values in $\mathcal{O}^\times$. Denote now by $\bar\lambda$ the complex conjugate character 
of $\lambda$, defined by $x\mapsto\lambda(\bar{x})$, where $x\mapsto \bar x$ is given by the complex conjugation on $K$. 
Then $\bar\lambda$ has infinity type $(0,1)$ and the $p$-adic avatar 
of $\lambda\bar\lambda$ is equal to the product $\chi_\lambda\cdot\chi_{\cyc,K}$ where  
$\chi_{\cyc,K}=\chi_\cyc\circ \rec_\Q \circ \mathbf{N}_{K/\Q}$ and $\chi_\lambda$ is a finite order character unramified at $p$. 

Let $F$ be the field of fractions of $\mathcal{O}$ and note that $\mathcal{O}^\times \cong {\mu}(F) \times \Z_p^ d$, as topological groups,  where $d = [F:\Q_p]$ and $\mu(F)$ is the group of roots of unity in $F$. Therefore each element $x \in \mathcal{O^\times}$ can be written uniquely as a product $\zeta_x\cdot\langle x \rangle$, where $\zeta_x$ is the projection of $x$ in ${\mu}(F)$ and $\langle x \rangle$ is the projection on $\Z_p^d$. Let $\mathcal{O}^\times_{\text{free}} \cong \Z_p^d$ be the maximal $\Z_p$-free quotient of $\mathcal{O}^\times$ and denote by $\langle - \rangle$ the projection $\mathcal{O}^\times \twoheadrightarrow \mathcal{O}^\times_{\text{free}}$. Let $W=\langle \im \hat\lambda \rangle$ be the projection of the image of $\hat{\lambda}$ in $\mathcal{O}_{\text{free}}^\times$. If $\lambda$ has conductor $\mathfrak{c}$ prime to $p$, then $\hat{\lambda}$ factors through $\Gal(K({\p^\infty\mathfrak{c}})/K)$, where $K({\p^\infty\mathfrak{c}})=\bigcup_{n\geq 1}K({\p^n\mathfrak{c}})$ and $K({\p^n\mathfrak{c}})$ is the ray class field of $K$ of conductor $\mathfrak{p}^n\mathfrak{c}$. Since $\mathcal{O}^\times_{\text{free}}$ is a free pro-$p$ group, the composition $\langle \hat\lambda \rangle$ of $\hat{\lambda}$ with the projection $\langle - \rangle$ factorizes through the maximal free pro-$p$ quotient of $\Gal(K({\p^\infty \mathfrak{c}})/K)$ which is a cyclic pro-$p$ group isomorphic to  $\Z_p^2$. 
Since $\lambda$ has infinity type $(1,0)$, $W$ is isomorphic to a subgroup of 
the maximal free pro-$p$ quotient of $\Gal(K({\p^\infty \mathfrak{c}})/K)$ which is isomorphic to 
$\Z_p$. Since $W$ contains $\Gamma$, then we can see $\Gamma$ as a subgroup of $W$ of finite index, \emph{cf}. \cite[\S2.7]{Castella}, \cite[pp. 64--65]{HidaHF}. Write $p^m = [W : \Gamma]$. Let $w \in W$ be a topological generator of $W$ such that $\gamma=w^{p^m}$ is a topological generator of $\Gamma$. 

%Let $\kappa\colon \I \rightarrow \mathcal{O}_\kappa $ be an arithmetic morphism with signature $(k,\varepsilon)$. Fix a $p^m$-th root $u \in \overline{\Q}_p$ of $\varepsilon(\gamma)$ and consider a finite extension $M_\kappa $ of $F_\kappa$ containing $u$. We can extend $\varepsilon $ to a continuous morphism $\tilde\varepsilon\colon W \rightarrow \mathcal{O}_{M_\kappa}^\times$, with $\mathcal{O}_{M_\kappa}$ the valuation ring of $M_\kappa$, sending $w$ to $u$.  This induces a unique continuous morphism $\tilde\kappa\colon \mathcal{O}\pwseries{W} \rightarrow \mathcal{O}_{M_\nu }$ by $\tilde\kappa(w)=\tilde\varepsilon(w)w^{k-2}$. Assume in the following that $\mathcal{O}\pwseries{W}$ is contained in $\I$ and that $\nu $ restricted to $W$ is given by $\nu (v)=\psi_\nu (v)v^{k_\nu -2}$ for any $v \in W$.

We now define a character 
${\boldsymbol{\lambda}}\colon K^\times\backslash\widehat{K}^\times\rightarrow\mathcal{O} [[W^\times]]$
by the formula 
$
{\boldsymbol{\lambda}}(x)=\hat{\lambda}(x)[\langle\hat{\lambda}(x)\rangle^{1/2}],$
where we view $\hat{\lambda}(x) \in \mathcal{O}^\times \hookrightarrow \mathcal{O} [[W^\times]]$ via the map $a\mapsto a\cdot 1_{W}$, with $1_W$ the identity element of $W$, $\langle \hat{\lambda}(x) \rangle$ denote the projection of $\hat{\lambda}(x)$ in $W$ and $z\mapsto [z]$ is the inclusion of group-like elements 
$W\hookrightarrow\mathcal{O}[[W]]^\times$.
Denote by $x \mapsto \hat{\lambda}(\bar{x})^{-1}=\lambda(\bar{x})^{-1}x_{\bar{\mathfrak{p}}}^{-1}$ (for $x \in \widehat{K}^{\times}$) the $p$-adic avatar of the Hecke character given by $x \mapsto \lambda(\bar{x})^{-1}$ (for $x \in \A^\times_K$) of infinity type $(0,-1)$. Then define the character
\[
{\boldsymbol{\lambda}}^{-1}(\bar{x})=\hat{\lambda}(\bar{x})^{-1}[\langle\hat{\lambda}(\bar{x})^{-1}\rangle^{1/2}]\]
taking values in $\mathcal{O}[[W]]^\times$. 
Finally, define the character ${\boldsymbol{\xi}}\colon K^\times\backslash\widehat{K}^\times\rightarrow \mathcal{O}[[W]]^\times$ given by
$
{\boldsymbol{\xi}}(x)={\boldsymbol{\lambda}}(x)\cdot{\boldsymbol{\lambda}}^{-1}(\bar x).$
Note that ${\boldsymbol{\xi}}_{|\widehat{\Q}^\times}$ is trivial. 

Let $\kappa\colon\I\rightarrow F_\kappa$ be an arithmetic morphism of signature $(k ,\varepsilon)$ 
and write $\hat{\lambda}_\kappa =\kappa \circ{\boldsymbol{\lambda}}$. 
Assume that $\mathcal{O}\pwseries W$ is contained in $\I$, and view $\boldsymbol{\lambda}$ as an $\I$-valued character ${\boldsymbol{\lambda}}\colon K^\times\backslash\widehat{K}^\times\rightarrow\I^\times$.
Then, for $x \in \widehat{K}^\times$ and $k \equiv k_0\mod{2(p-1)}$, we have 
\[\hat{\lambda}_\kappa (x)=\varepsilon ^{1/2}(\langle \hat{\lambda}(x)\rangle)\cdot {\lambda}(x) ^{k /2} x_{\mathfrak{p}}^{k/2}.\]
Hence $\hat{\lambda}_\kappa$ is the $p$-adic avatar 
of an algebraic Hecke character $\lambda_\kappa $ of infinity type $(k/2,0)$. Set as before $\hat{\xi}_\kappa =\kappa \circ {\boldsymbol{\xi}}$. For any $x \in \widehat{K}^\times$, we have 
\begin{equation}\label{hatxi}
\hat{\xi}_\kappa(x)= \varepsilon ^{1/2}(\langle{\lambda}(x
\bar{x}^{-1})x_\mathfrak{p} x^{-1}_{\bar{\mathfrak{p}}}\rangle)\cdot {\lambda}(x\bar{x}^{-1}) ^{k /2} \cdot x_{\mathfrak{p}}^{k/2} x_{\bar{\mathfrak{p}}}^{-k/2}.\end{equation}
Therefore, 
$\hat{\xi}_\kappa $ is the $p$-adic avatar of an anticyclotomic Hecke character $\xi_\kappa $ of infinity type $(k /2, -k /2)$. 

\subsection{$p$-adic $L$-functions for families of modular forms}
%As in \S\ref{sec.controltheorem}, fix a primitive branch $\mathbb{I}$ of the Hida family passing through a $p$-stabilized newform $f\in S_{k_0}(\Gamma_0(Np))$ of trivial character and even weight $k_0\equiv 2\mod{2(p-1)}$, so that $\mathbb{I}$ is a finite flat extension of $\mathcal{O}[[1+p\Z_p]]$ where $\mathcal{O}$ is the valuation ring of a finite unramified extension of $\Q_p$. 
Let, as in \S\ref{subsec.STexp}, $\eta_\mathbb{I}$ be a fixed $\mathbb{I}$-generator of $\mathbf{D}(\mathbf{T}^+)$. 
Let 
$\mathcal{F}_\I(T_x)$ be the associated power series in $\widetilde{\I}[[T_{x(\mathfrak{a})}]]$  as in \eqref{FIT}. 
Let $c\mathcal{O}_K$, with $c\geq 1$ and $(c,pND_K)=1$ be the conductor of $x \mapsto\boldsymbol{\lambda}(x)\boldsymbol{\lambda}^{-1}(\bar{x})$.
Consider the CM points $x(\mathfrak{a})$ with $\mathfrak{a} \in \Pic(\mathcal{O}_c)$, defined in \S \ref{CM section}.
Recall that $x(\mathfrak{a})$ has a model defined over ${\Z_p^\unr}$, and define 
the fiber product $x(\mathfrak{a})_{{\I}} \defeq x(\mathfrak{a}) \otimes_{\Z_p^\unr}\widetilde{\I}.$
Define now a $\widetilde{\I}$-valued measure $\mu_{\eta_\I,\mathfrak{a}}$ on $\Z_p$ by
\begin{equation*} \label{mu_a}
\int_{\Z_p} (T_{x(\mathfrak{a})}+1)^x d\mu_{\eta_\I,\mathfrak{a}}(x) = 
{\mathcal{F}}_\I^{[p]} \left( (T_{x(\mathfrak{a})}+1)^{\mathrm{N}(\mathfrak{a}^{-1})\sqrt{(-D_K)}^{-1}}\right) \in \widetilde{\I}[[T_{x(\mathfrak{a})}]],
\end{equation*}
where as before 
$\mathcal{F}_\I^{[p]}(T_x)=\sum_{p\nmid n}a_nT_{x(\mathfrak{a})}^n$ is the $p$-depletion of $\mathcal{F}_\I(T_x)= \sum_{n\geq 0}a_nT_{x(\mathfrak{a})}^n$. 

\begin{definition} \label{Lfct}
The $p$-adic $L$-function associated with $\eta_\mathbb{I}$ and $\boldsymbol{\xi}$ is  the $\widetilde{\I}$-valued measure on the Galois group 
$\Gamma_\infty=\Gal(H_{cp^{\infty}}/K)$ given on any continuous function $\varphi\colon\Gamma_\infty\rightarrow \widetilde{\I}$ by 
\[
\mathscr{L}_{\I,\boldsymbol{\xi}} (\varphi) = \sum_{\mathfrak{a} \in \Pic(\mathcal O_{c})}
{\boldsymbol{\chi}}^{-1}
{\boldsymbol{\xi}}(\mathfrak{a})
\mathrm{N}(\mathfrak{a})^{-1}
	\int_{\Z_p^{\times}} 
	(\varphi \big|[\mathfrak{a}])(u)d\mu_{\eta_\I,\mathfrak{a}}(u).
	\]
\end{definition}

Clearly, $\mathscr{L}_{\I,\boldsymbol{\xi}}$ is well defined up to multiplication by an element in $\I^\times$. 
We may view 
$\mathscr{L}_{\I,\boldsymbol{\xi}} $ as a power series in $\widetilde{\I}[[{\Gamma_\infty}]]$. 
For any continuous character $\varphi\colon\Gamma_\infty\rightarrow\overline\Q_p^\times$ and any arithmetic morphism $\kappa \colon\I\rightarrow \mathcal{O}_\kappa $ we write 
$\mathscr{L}_{\I,\boldsymbol\xi}^{\mathrm{an}}(\kappa )=\kappa\circ\mathscr{L}_{\I,\boldsymbol\xi}^{\mathrm{an}}$ and $\mathscr{L}_{\I,\boldsymbol\xi}^{\mathrm{an}}(\kappa,\varphi)=\mathscr{L}_{\I,\boldsymbol\xi}^{\mathrm{an}}(\kappa )(\varphi)$.

\subsection{Interpolation}\label{sec5.4}  
The following result generalizes \cite[Theorem 2.11]{Castella} to the quaternionic setting. Fix 
$\lambda\colon K^\times\backslash\A_K^\times\rightarrow\overline\Q^\times$ an algebraic Hecke character of infinity type $(1,0)$ of conductor $\mathfrak{c}\subseteq \mathcal{O}_K$ prime to $pND_K$ whose $p$-adic avatar takes values in $\mathcal{O}^\times$ and let  $\boldsymbol{\xi}$ be defined from $\lambda$ as explained in \S\ref{characters}. Let $\I$ be a primitive Hida family as in \S\ref{sec-famHecke}, 
fix a generator $\eta_\I$ as before and let $\mathcal{F}_\I(T_x)\in\widetilde{\I}[[T_x]]$ be the associated power series. For any arithmetic morphism $\kappa\colon \I\rightarrow\mathcal{O}_\kappa$, of level $\Gamma_m$ and signature $(k,\epsilon_\mathrm{triv})$, where $\epsilon_\mathrm{triv}$ is the trivial character  
and $k \equiv k_0\equiv 2 \mod 2(p-1)$, 
let $f_\kappa$ be the unique modular form in $S_k(\Gamma_m,F_\kappa)$ whose power series expansion coincides with $\mathcal{F}_\kappa(T_x)$. 

\begin{theorem}\label{mainthh}
Let $\kappa\colon\I\rightarrow\mathcal{O}_\kappa$ be an arithmetic morphism of weight $k \equiv k_0 \mod 2(p-1)$. 
Then 
$\mathscr{L}_{\I,\boldsymbol\xi}^{\mathrm{an}}(\kappa )=
\vartheta^{-1}_\kappa (c) \mathscr{L}_{f_\kappa ,\xi_\kappa }$.
\end{theorem}

\begin{proof} 
For any continuous $\varphi\colon \Z_p \to \widetilde{\I}$, using its Mahler expansion, we obtain for $\mathfrak{a} \in \Pic(\mathcal{O}_c)$:
\[
\kappa \left(\int_{\Z_p^\times} \varphi(u) d{\mu_{\eta_\I,\mathfrak{a}}}(u)\right) = \int_{\Z_p^\times}(\kappa  \circ \varphi)(u) d{\mu_{f_\kappa ,\mathfrak{a}}}(u).
\]
Recall that, for an ideal $\mathfrak{a}\subseteq\mathcal{O}_c$, we write $\mathrm{N}(\mathfrak{a})=c^{-1}\sharp(\mathcal{O}_c/\mathfrak{a})$; if $\mathfrak{a}=a\widehat{\mathcal{O}}_c\cap K$ for an element $a\in\widehat{\mathcal{O}}_K$,
we have 
$\mathrm{N}(\mathfrak{a})=c^{-1}\cdot\mathbf{N}^{-1}_K(a)$.  
Choose representatives $\mathfrak{a}$ of  $\Pic(\mathcal{O}_c)$ such that 
$\mathfrak{p}\nmid \mathfrak{a}$
and $\bar{\mathfrak{p}}\nmid \mathfrak{a}$; by \eqref{chi_k} we then have  
\[
\hat{\chi}^{-1}_\kappa (\mathfrak{a})= \hat{\chi}^{-1}_\kappa(a) = \mathbf{N}_K(a)^{k /2-1} = \mathrm{N}(\mathfrak{a})^{-k /2+1} c^{-k/2+1}
\]
and $\xi_\kappa(\mathfrak{a})=\hat{\xi}_\kappa(\mathfrak{a})$. 
Therefore, since $\xi_\kappa$ is unramified at $p$, we have 
	\begin{equation*}
	\begin{split}
\mathscr{L}_{\I,\boldsymbol\xi}(\kappa,\varphi) &= \sum_{\mathfrak{a} \in \Pic(\mathcal O_{c})}
({\hat{\chi}}_\kappa ^{-1}
{\hat{\xi}}_\kappa )(\mathfrak{a})
\mathrm{N}(\mathfrak{a})^{-1}
\int_{\Z_p^{\times}}
(\kappa\circ\varphi \big|[\mathfrak{a}])(u)d\mu_{\eta_\I,\mathfrak{a}}(u)\\
&= \sum_{\mathfrak{a} \in \Pic(\mathcal{O}_{c})}c^{-k/2+1} \mathrm{N}(\mathfrak{a})^{-\frac{k}{2}}{\xi}_\kappa (\mathfrak{a}) \int_{\Z_p^{\times}} \xi_{\kappa ,\mathfrak{p}}(u)(\kappa  \circ \varphi \big|[\mathfrak{a}])(u)d\mu_{f_\kappa ,\mathfrak{a}}(u)
\\
&=c^{-k/2+1} \mathscr{L}_{f_\kappa ,\xi_\kappa}(\kappa \circ \varphi).
\end{split}
\end{equation*}
In particular we have 
$\mathscr{L}_{\I,\boldsymbol\xi}(\kappa ,\mathbf{1}_U) = c^{-k /2+1} \mathscr{L}
_{f_\kappa ,\xi_\kappa }(\mathbf{1}_U)$
for any open compact subset $U$ of the Galois group  $\Gamma_\infty$, where $\mathbf{1}_U$ is the characteristic function of $U$. We conclude the proof via the equivalence between measures on $\Gamma_\infty$ and additive functions on the set of open compact subsets of $\Gamma_\infty$.
\end{proof}

We close by stating the interpolation properties satisfied by this $p$-adic $L$-function. Let $f_\kappa \in S_k(\Gamma_1,\mathcal{O})$ have trivial character as before. Then, since $f_\kappa $ is ordinary, either $f_\kappa $ has weight $2$ and is a newform of level $Np$, 
or there is a newform $f_\kappa ^\sharp\in S_k(\Gamma_0,K)$ whose ordinary $p$-stabilization is $f_\kappa $. Suppose we are in the second case.

\begin{corollary}
Let $\kappa\colon\I\rightarrow\mathcal{O}_\kappa$ be an arithmetic morphism of weight $k \equiv k_0 \mod 2(p-1)$. If $\hat\varphi\colon\Gamma_\infty\rightarrow\overline\Q_p^\times$ corresponds via $\mathrm{rec}_K$ to the $p$-adic avatar of an anticyclotomic Hecke character $\varphi$ of infinity type $(n,-n)$ 
for some integer $n\geq 0$, 
then
\[\left(\mathscr{L}_{\I,\boldsymbol\xi}(\kappa,\hat\varphi)\right)^2=C\cdot L(f_\kappa ^\sharp,\xi_\kappa\varphi,k /2)
\] where 
$C$ is a non-zero constant (depending on $\kappa$, $\xi_\kappa $, $\varphi$ and $K$, and, up to a $p$-adic unit, on the choice of $\eta_\I$).
\end{corollary}
\begin{proof}
In view of Lemma \ref{lemma5.1}, the Serre--Tate expansions of $f_\kappa^{[p]}=(f^\sharp_\kappa)^{[p]}$ are the same, putting us in the conditions of Theorem \ref{mainthh}. The relation between the analytic $L$-function $\mathscr{L}_{f_\kappa ,\xi_\kappa }$ and the algebraic variant in the right-hand side, and the nature of the constant $C$, is discussed in \cite[Theorem 4.6]{Magrone}; when $\varphi$ is unramified, see \cite[8.11]{Brooks} and \cite[Theorem 0.1]{Mori1} for an explicit expression of $C$.
\end{proof}

\bibliographystyle{amsalpha}
\bibliography{references}
\end{document}